\RequirePackage{snapshot}

\newif\ifsubsections
\subsectionstrue 
\documentclass[]{pcmi}

\usepackage[sc]{mathpazo}          % Palatino with smallcaps as text font
\usepackage{eulervm}               % Euler math
\usepackage[scaled=0.86]{berasans} % Bera for san serif family
\usepackage[scaled=1]{inconsolata} % Inconsolata for fixed width
\usepackage[T1]{fontenc}

\usepackage[%
	protrusion=true,
	expansion=false,
	auto=false
	]{microtype}

\usepackage{xcolor}
\usepackage{graphicx}
\graphicspath{{./figures/}}

\ifdraft
	\definecolor{linkred}{rgb}{0.7,0.2,0.2}
	\definecolor{linkblue}{rgb}{0,0.2,0.6}
\else
	\definecolor{linkred}{rgb}{0.0,0.0,0.0}
	\definecolor{linkblue}{rgb}{0,0.0,0.0}
\fi

\usepackage{amssymb,amsfonts,amsthm}
\usepackage[makeroom]{cancel}
\usepackage{mathtools}
\usepackage{pifont}% AO: removed mathrsfs as the script it uses clashes with the Euler math used here and replaced it with eucal
\usepackage[mathscr]{eucal}
\usepackage{slashed,mathabx} 
\usepackage[bbgreekl]{mathbbol}

\usepackage[all]{xy}

\PassOptionsToPackage{hyphens}{url} 
\usepackage[
    setpagesize=false,
    pagebackref,
	pdfpagelabels=false,
    pdfstartview={FitH 1000},
    bookmarksnumbered=false,
    linktoc=all,
    colorlinks=true,
    anchorcolor=black,
    menucolor=black,
    runcolor=black,
    filecolor=black,
    linkcolor=linkblue,%IM black if not in draft mode
	citecolor=linkblue,%IM black if not in draft mode
	urlcolor=linkred,%IM black if not in draft mode
]{hyperref}%IM
\usepackage[backrefs,msc-links,nobysame]{amsrefs}

\customizeamsrefs 

\theoremstyle{plain}
\newtheorem{Proposition}[equation]{Proposition}
\newtheorem{Lemma}[equation]{Lemma}
\newtheorem{Corollary}[equation]{Corollary}
\newtheorem{Theorem}[equation]{Theorem}

\theoremstyle{definition}
\newtheorem{Definition}[equation]{Definition}
\newtheorem{Exercise}[equation]{Exercise}

\let\oldsubsubsection\subsubsection
\renewcommand{\subsubsection}[1]{%
\setcounter{subsubsection}{\value{equation}}%
\oldsubsubsection{#1}% invoke "classic" subsubsection (incrementing this counter to the "next" equation number)
\refstepcounter{equation}% increment equation so it skips over the number of the new subsubsection
}

\def\subsubfleuron{\leavevmode\kern-\parindent
\rlap{$\blacktriangleright$}%
\kern\parindent\relax}

\DeclareRobustCommand{\hhbar}{%
  {\mathord{\text{\lower0.05ex\hbox{$\mathchar'26$}{$\mkern-7mu h$}}}}%
}
\makeatletter
\def\evm@hbar{{\hhbar}}
\makeatother

\newcommand{\C}{\mathbb{C}}
\newcommand{\Z}{\mathbb{Z}}
\newcommand{\Q}{\mathbb{Q}}
\newcommand{\R}{\mathbb{R}}
\newcommand{\N}{\mathbb{N}}
\newcommand{\bbO}{\mathbb{O}}
\newcommand{\mM}{\mathbb{M}}
\newcommand{\mMh}{{\widetilde{\mM}}}
\newcommand{\bbA}{\mathbb{A}}

\newcommand{\bF}{\mathbf{F}}
\newcommand{\bY}{\mathsf{Y}}
\newcommand{\bbV}{\mathbb{V}}
\newcommand{\bbW}{\mathbb{W}}
\newcommand{\bv}{\mathsf{v}}
\newcommand{\bw}{\mathsf{w}}

\newcommand{\aroof}{\widehat{\mathsf{a}}}
 
\newcommand{\bX}{\mathsf{X}}
\newcommand{\bJ}{\mathsf{J}}
\newcommand{\bT}{\mathsf{T}}
\newcommand{\bZ}{\mathsf{Z}}
\newcommand{\bW}{\mathsf{W}}

\newcommand{\bA}{\mathsf{A}}
\newcommand{\fm}{\mathsf{m}}

\newcommand{\bP}{\mathbb{P}}
\newcommand{\bE}{\mathsf{E}}
\newcommand{\bS}{\mathsf{S}}
\newcommand{\cA}{\mathscr{A}}
\newcommand{\cK}{\mathscr{K}}
\newcommand{\cL}{\mathscr{L}}

\newcommand{\cO}{\mathscr{O}}
\newcommand{\cE}{\mathscr{E}}
\newcommand{\cU}{\mathscr{U}}
\newcommand{\cN}{\mathscr{N}}
\newcommand{\cH}{\mathscr{H}}
\newcommand{\cW}{\mathscr{W}}
\newcommand{\cQ}{\mathscr{Q}}
\newcommand{\cM}{\mathscr{M}}
\newcommand{\cV}{\mathscr{V}}
\newcommand{\cT}{\mathscr{T}}

\newcommand{\cB}{\mathscr{B}}
\newcommand{\cR}{\mathscr{R}}
\newcommand{\cC}{\mathscr{C}}

\newcommand{\fZ}{\mathfrak{Z}}
\newcommand{\fF}{\mathfrak{F}}

\newcommand{\Dir}{\slashed D} 
\newcommand{\cHom}{\mathscr{H}\!om} 
\newcommand{\cTor}{\mathscr{T}\!or} 
\newcommand{\bfM}{\mathbf{M}}
\newcommand{\bfS}{\mathbf{S}}
\newcommand{\bfG}{\mathbf{G}}
\newcommand{\Camp}{C_\textup{ample}}

\newcommand{\cI}{\mathscr{I}}

\newcommand{\lan}{\left\langle} 
\newcommand{\ran}{\right\rangle}

\newcommand{\fg}{\mathfrak{g}} 
\newcommand{\fgh}{\hat{\mathfrak{g}}}

\newcommand{\Free}{\mathsf{Free}}

\newcommand{\rd}{/\!\!/\!\!/\!\!/}
\newcommand{\rdd}{/\!\!/}
\newcommand{\Ct}{\mathbb{C}^\times}
\newcommand{\Tube}{\mathsf{Tube}}

\newcommand{\Vertex}{\mathsf{Vertex}}
\newcommand{\Vertext}{\widetilde{\mathsf{Vertex}}}

\newcommand{\bolde}{\boldsymbol{\varepsilon}}

\newcommand{\bfC}{\mathbf{C}}

\newcommand{\Topp}{T^{1/2}_\textup{opp}}
\newcommand{\fC}{\mathfrak{C}}

\newcommand{\Hd}{{H}^{\raisebox{0.5mm}{$\scriptscriptstyle \bullet$}}}
\newcommand{\bSd}{{\mathsf{S}}^{\raisebox{0.5mm}{$\scriptscriptstyle
      \bullet$}}}
\newcommand{\bSdh}{{\widehat{\mathsf{S}}}^{\raisebox{0.5mm}{$\scriptscriptstyle
      \bullet$}}}
\newcommand{\Extd}{{\Ext}^{\raisebox{0.5mm}{$\scriptscriptstyle \bullet$}}}

\newcommand{\eqf}{\overset{\textup{\tiny \it{f}}}=}
\newcommand{\eqdef}{\overset{\textup{\tiny def}}=}

\newcommand{\glh}{\widehat{\mathfrak{gl}}} 

\newcommand{\bra}[1]{\left\langle #1 \right|}
\newcommand{\brar}[1]{\left\langle #1 \right\|}
\newcommand{\ket}[1]{\left| #1 \right\rangle}
\newcommand{\ketr}[1]{\left\| #1 \right\rangle}

\newcommand{\transp}{\textup{\tiny transp}}

\renewcommand{\dbar}{\bar\partial} 
\newcommand{\bx}{\textup{\ding{114}}}
\newcommand{\strr}{\textup{\ding{73}}}
\newcommand{\bk}{\boldsymbol{\kappa}}
\newcommand{\tX}{\widetilde{X}}
\newcommand{\Mbar}{\overline{M}}

\DeclareMathOperator{\Def}{Def}

\DeclareMathOperator{\Obs}{Obs}
\DeclareMathOperator{\Lie}{Lie}

\DeclareMathOperator{\Res}{Res}

\DeclareMathOperator{\Gr}{Gr}

\DeclareMathOperator{\Hilb}{Hilb}
\DeclareMathOperator{\End}{End}
\DeclareMathOperator{\Ext}{Ext}
\DeclareMathOperator{\Hom}{Hom}
\DeclareMathOperator{\Aut}{Aut}
\DeclareMathOperator{\rk}{rk}
\DeclareMathOperator{\ev}{ev}
\DeclareMathOperator{\Pic}{Pic}
\DeclareMathOperator{\Bl}{Bl}
\DeclareMathOperator{\supp}{supp}

\DeclareMathOperator{\Stab}{Stab}
\DeclareMathOperator{\Attr}{Attr}
\DeclareMathOperator{\diag}{diag}
\DeclareMathOperator{\codim}{codim}
\DeclareMathOperator{\str}{str}
\DeclareMathOperator{\indx}{index}
\DeclareMathOperator{\Prob}{Prob}
\DeclareMathOperator{\dist}{distance}

\newcommand{\Ld}{{\Lambda}^{\!\raisebox{0.5mm}{$\scriptscriptstyle
      \bullet$}}\!}
\newcommand{\Omd}{{\Omega}^{\raisebox{0.5mm}{$\scriptscriptstyle
      \bullet$}}\!}
\newcommand{\Ldm}{{\Lambda}^{\!\raisebox{0.5mm}{$\scriptscriptstyle \bullet$}\,}_{\scriptscriptstyle{-}}}

\newcommand{\vir}{\textup{vir}}
\newcommand{\orb}{\textup{orb}}
\newcommand{\opp}{\textup{opp}}
\newcommand{\points}{\textup{points}}
\newcommand{\mov}{\textup{moving}}
\newcommand{\nega}{\mathbf{n}}
\newcommand{\posi}{\mathbf{p}}

\newcommand{\pt}{\textup{pt}}
\newcommand{\bmu}{\boldsymbol{\mu}}
\newcommand{\bxi}{\boldsymbol{\xi}}
\newcommand{\QM}{\mathsf{QM}}

\newcommand{\cF}{\mathscr{F}}
\newcommand{\whcF}{\widehat{\mathscr{F}}}
\newcommand{\cG}{\mathscr{G}}

\newcommand{\tO}{\widehat{\mathscr{O}}}

\DeclareMathOperator{\Ker}{Ker}
\DeclareMathOperator{\Coker}{Coker}
\DeclareMathOperator{\tr}{tr}
\DeclareMathOperator{\weight}{weight}
\DeclareMathOperator{\Spec}{Spec}
\DeclareMathOperator{\Ann}{Ann}
\DeclareMathOperator{\Tor}{Tor}

\begin{document}

%%%%%%%%%%%%%%%%%%%%%%%%%%%%%%%%%%%%%%%%%%%%%%%%%%%%%%%%%%%%%%%%%%%%%
%
% Filling in title page, running head and author fields.
%
% Each point in this template at which you will need to fill in your own
%     information is indicated by a dummy \replace macro (shown above). 
%     Replace this macro and its argument by your corresponding information.
%
%     For example, Gauss would replace  
%         \author{\replace{Ian Morrison}}
%     by 
%         \author{Carl Friedrich Gauß}
%
%     Note that you may add a short title for running heads, if needed as an optional
%         argument as in the example below
%         \title[Guide for PCMI Authors]{Guide for Lecturers in the Park City Mathematics Institute}
%
%
%%%%%%%%%%%%%%%%%%%%%%%%%%%%%%%%%%%%%%%%%%%%%%%%%%%%%%%%%%%%%%%%%%%%%

\title[K-theoretic computations in 
enumerative geometry]{Lectures on K-theoretic computations in 
enumerative geometry} 

%%%%%%%%%%%%%%%%%%%%%%%%%%%%%%%%%%%%%%%%%%%%%%%%%%%%%%%%%%%%%%%%%%%%%
%    
%    Author information--add further authors as needed
%    
%%%%%%%%%%%%%%%%%%%%%%%%%%%%%%%%%%%%%%%%%%%%%%%%%%%%%%%%%%%%%%%%%%%%%
\author{Andrei Okounkov}
\date{
%New York and Park City, 
Spring and Summer 2015} 
\address{
\newline Department of Mathematics,
Columbia University,
New York, NY 10027, 
\newline Institute for Problems of Information Transmission
Bolshoy Karetny 19, Moscow 127994, Russia
\newline 
Laboratory of Representation
Theory and Mathematical Physics 
Higher School of Economics 
Myasnitskaya 20, Moscow 101000, Russia 
}
\email{	okounkov@math.columbia.edu}

%%%%%%%%%%%%%%%%%%%%%%%%%%%%%%%%%%%%%%%%%%%%%%%%%%%%%%%%%%%%%%%%%%%%%
%    
%    Classification and abstract
%    
%%%%%%%%%%%%%%%%%%%%%%%%%%%%%%%%%%%%%%%%%%%%%%%%%%%%%%%%%%%%%%%%%%%%%
\keywords{Park City Mathematics Institute}

%%%%%%%%%%%%%%%%%%%%%%%%%%%%%%%%%%%%%%%%%%%%%%%%%%%%%%%%%%%%%%%%%%%%%
%    
%    Make the title page
%    
%%%%%%%%%%%%%%%%%%%%%%%%%%%%%%%%%%%%%%%%%%%%%%%%%%%%%%%%%%%%%%%%%%%%%
\maketitle

\tableofcontents

%%%%%%%%%%%%%%%%%%%%%%%%%%%%%%%%%%%%%%%%%%%%%%%%%%%%%%%%%%%%%%%%%%%%%
%    
%    Your lecture notes replace the remainder of this document.
%    
%%%%%%%%%%%%%%%%%%%%%%%%%%%%%%%%%%%%%%%%%%%%%%%%%%%%%%%%%%%%%%%%%%%%%

% Authors: insert the body of your lectures here

\section{Aims \& Scope} 

\subsection{K-theoretic enumerative geometry}

\subsubsection{}
These lectures are for graduate students who want to learn how 
to do the \emph{computations} from the title. Here I put emphasis 
on computations because I think it is very important to keep 
a connection between abstract notions of algebraic geometry, which can be very 
abstract indeed, and something we can see, feel, or test with 
code. 

While it is a challenge to adequately illustrate a modern algebraic geometry 
narrative, one can become familiar with main characters of these notes 
by working out examples, and my hope is that these notes will be placed
alongside a pad of paper, a pencil, an eraser, and a symbolic computation
interface. 

\subsubsection{} 

Modern enumerative 
geometry is not so much about numbers as it is about 
deeper properties of the moduli spaces that parametrize 
the geometric objects being enumerated. 

Of course, once a relevant moduli space $\mM$ is constructed one 
can study it as one would study any other algebraic scheme
(or stack, depending on the context). Doing this in any generality
 would appear to be seriously challenging, as even the dimension
of some of the simplest moduli spaces considered here 
(namely, the Hilbert scheme of points of 3-folds) is not known. 

\subsubsection{}

A productive middle ground is to compute the Euler characteristics
$\chi(\cF)$ 
of naturally defined coherent sheaves $\cF$ on $\mM$, as representations of 
a group $G$ of automorphisms of the problem. This goes beyond 
intersecting natural cycles in $\mM$, which is the realm of the 
traditional enumerative geometry, and is a nutshell description 
of equivariant K-theoretic enumerative geometry. 

The group $G$ will typically be connected and reductive, and 
the $G$-character of $\chi(\cF)$ will be a Laurent polynomial
on the maximal torus $T\subset G$ provided $\chi(\cF)$ is a
finite-dimensional virtual $G$-module. 
Otherwise, it will be a rational function on $T$. 
 In either case, it is a very concrete mathematical 
object, which can be turned and spun to be seen from many 
different angles.

\subsubsection{}
Enumerative geometry has received a very significant 
impetus from modern high energy physics, and this is even 
more true of its K-theoretic version. From its very origins, 
K-theory has been inseparable from indices of 
differential operators and related questions in mathematical 
quantum mechanics. For a mathematical physicist, 
the challenge is to study quantum systems with infinitely 
many degrees of freedom, systems that describe fluctuating
objects that have some spatial extent. 

While the dynamics of such systems is, obviously, very complex, 
their vacua, especially supersymmetric vacua, i.e.\ the quantum 
states in the kernel 
of a certain infinite-dimensional Dirac operator
\begin{equation}
\xymatrix{
\cH_\textup{even} \ar@/^/[rr]^{\Dir} && \cH_\textup{odd} \ar@/^/[ll]^\Dir
}\label{HDH} 
\end{equation}
may often be 
understood in finite-dimensional terms. In particular, the 
computations of supertraces
\begin{align}
\str e^{-\beta \Dir^2} A & = \tr_{(\Ker\Dir)_{\textup{even}}} A -
\tr_{(\Ker\Dir)_{\textup{odd}}} A \notag \\
&=\str_{\indx \Dir} A 
\label{strA}
\end{align}
of natural operators $A$ commuting with $\Dir$ 
may be equated with the kind of 
computations that is done in these notes. 

Theoretical physicists developed very powerful and insightful ways 
of thinking about such problems and theoretical physics serves as a 
very important source of both inspiration and applications for 
the mathematics pursued here. We will see
many examples of this below. 

\subsubsection{}

What was said so far encompasses a whole universe of physical 
and enumerative contexts. While resting on certain common principles, 
this universe is much too rich and diverse to be reasonably 
discussed here. 

These lectures are written with a particular goal in mind, which is 
to introduce a student to computations in two intertwined 
subjects, namely:
\begin{itemize}
\item[---] K-theoretic 
Donaldson-Thomas theory, and 
\item[---] quantum K-theory of Nakajima varieties. 
\end{itemize}
Some key features of these theories, and of the relationship 
between them, may be summarized as follows. 

\subsection{Quantum K-theory of Nakajima varieties}\label{s_Q_Nak}

\subsubsection{}
Nakajima varieties \cite{Nak1,Nak2} are associated to a quiver (that is, 
a finite graph with possibly parallel edges and loops), a pair of dimension 
vectors, and a choice of stability chamber.
 They form a remarkable family of \emph{equivariant 
symplectic resolutions} \cite{Kal2} and have found many geometric and 
representation-theoretic applications. Their construction will 
be recalled in Section \ref{s_Nak}. 

For quivers of affine ADE type, and a suitable choice of stability,
 Nakajima varieties are the moduli spaces of 
framed coherent sheaves on the corresponding ADE
surfaces, e.g.\ on $\C^2$ for a quiver with one vertex and 
one loop. These moduli spaces play a very important role in 
supersymmetric gauge theories and algebraic geometry. 

For general quivers, Nakajima varieties share many properties of the 
moduli
spaces of sheaves on a symplectic surface. 
In fact, from their construction, they may be interpreted as
moduli spaces of stable objects in 
certain abelian categories which have the same duality 
properties as coherent sheaves on a symplectic surface. 

\subsubsection{}
{}From many different perspectives, rational 
curves in Nakajima varieties are of particular interest. 
Geometrically, a map 
$$
\textup{curve $C$} \to \textup{Moduli of sheaves on a surface $S$}
$$
produces a coherent sheaf on a threefold $C\times S$. 
One thus expects a 
relation between  enumerative geometry of sheaves 
on threefolds fibered in ADE surfaces and 
enumerative geometry of maps from a fixed curve 
to affine ADE Nakajima varieties.  

Such a relation indeed exists and has been already profitably 
used in cohomology \cite{OP1,OP2,MauObl,moop}. For it to work well, it is important 
that the fibers are symplectic. Also, because the source $C$ of the 
map is fixed and doesn't vary in moduli, it can be taken to be 
a rational curve, or a union of rational curves. 

Rational curves in other Nakajima varieties lead to enumerative
problems of similar 3-dimensional flavor, even when they 
are lacking a direct geometric interpretations as counting sheaves on 
some threefold.  

\subsubsection{}

In cohomology, counts of rational curves in a 
Nakajima variety $X$ are conveniently 
packaged in terms of equivariant \emph{quantum cohomology}, which is 
a certain deformation of the cup product in $\Hd_G(X)$ with 
deformation base 
$H^2(X)$. 
A related structure is the equivariant \emph{quantum connection}, or 
Dubrovin connection, on the trivial $\Hd_G(X)$ bundle with 
base $H^2(X)$. 

While such packaging of enumerative information may be performed
for an abstract algebraic variety $X$, for Nakajima varieties these 
structures are described in the language of geometric representation 
theory, and namely in terms of an action of a certain Yangian $\bY(\fg)$ 
on  $\Hd_G(X)$. In particular, the quantum connection is identified 
with the trigonometric Casimir connection for $\bY(\fg)$, studied in
\cite{TolCas} for finite-dimensional Lie algebras, see also
\cite{TV}. 

The construction of the required Yangian $\bY(\fg)$, and the 
identification of quantum structures in terms of $\bY(\fg)$, are
the main results of \cite{MO1}. That work was inspired by conjectures put 
forward by Nekrasov and Shatashvili \cite{NS1,NS2}, on the one hand, and by 
Bezrukavnikov and his collaborators (see e.g.\ \cite{Et}), on the other. 

A similar geometric representation theory description of quantum 
cohomology is expected for all equivariant symplectic resolutions, 
and perhaps a little bit beyond, see for example \cite{FFFR}. Further, there 
are conjectural links, due to Bezrukavnikov and his collaborators,
 between quantum connections for symplectic 
resolutions $X$ and representation theory 
of their quantizations (see, for example, \cite{BezF,BK,BM,BezLo,Kal1}),  their derived
autoequivalences etc. 

\subsubsection{}

The extension \cite{MO2} of our work \cite{MO1} with 
D.~Maulik to K-theory requires
several new ideas, as certain arguments that are used again
and again in \cite{MO1} are simply not available in K-theory. For instance, 
in equivariant cohomology, a proper integral of correct dimension is 
a nonequivariant constant, which may be computed using an arbitrary 
specialization of the equivariant parameters (it is typically very convenient to set 
the weight $\hbar$ of the symplectic form to $0$ and send all other
equivariant variables to infinity). 

By contrast, in equivariant $K$-theory, the
only automatic conclusion about a proper pushforward to a point is
that it is a Laurent polynomial in the equivariant variables, and,
most of the time, this cannot be improved. If this Laurent polynomial does
not involve some 
equivariant variable variable $a$ then this is called \emph{rigidity}, and is
typically shown by a careful study of  the $a^{\pm 1} \to \infty$ limit. We will 
see many variations on this theme below. 

Also, very seldom there is rigidity 
with respect to the weight  $\hbar$ of the symplectic form 
and nothing can 
be learned from making that weight trivial. Any argument 
that involves such step in cohomology needs to be modified, most 
notably the proof of the commutation of quantum connection with 
the difference Knizhnik-Zamolodchikov connection, see Sections 1.4 and 
9 of \cite{MO1}. In the logic of \cite{MO1}, quantum connection is
characterized by this commutation property, so it is very important to
lift the argument to K-theory. 

One of  the goals of these notes is to serve as an introduction to the
additional techniques required for working in K-theory.  In
particular, the quantum Knizhnik-Zamolodchikov connection appears
in Section \ref{qKZ} as the computation of the shift operator 
corresponding to a minuscule cocharacter. Previously in Section 
\ref{s_shift_eq} we construct a difference connection which shifts 
equivariant variables by cocharacters of the maximal torus and show 
it commutes with the K-theoretic upgrade of the quantum 
connection. That upgrade is a difference connection that shifts the 
K\"ahler variables 
$$
z \in H^2(X,\C) \big/ 2 \pi i H^2(X,\Z)
$$
by cocharacters of this torus, that is, by a lattice isomorphic to
$\Pic(X)$. This quantum difference 
connection is constructed in Section \ref{s_shift_K}. 

\subsubsection{}

Technical differences notwithstanding, the eventual description of 
quantum $K$-theory of Nakajima varieties is exactly what one 
might have guessed recognizing a general pattern in 
geometric representation theory. The Yangian $\bY(\fg)$, which 
is a Hopf algebra deformation of $\cU(\fg\otimes \C[t])$, is a first 
member of a hierarchy in which the Lie algebra
$$
\fg\otimes \C[t] = \textup{Maps}(\C,\fg)
$$
is replaced, in turn,  by a central extension of maps from 
\begin{itemize}
\item[---] the multiplicative 
group $\C^\times$, or 
\item[---] an elliptic curve. 
\end{itemize}
Geometric realizations of 
the corresponding quantum groups are constructed in equivariant 
cohomology, 
equivariant K-theory, and equivariant elliptic cohomology, respectively, 
see \cite{MO2,OS} and also \cite{AO} for the construction of an elliptic 
quantum group from Nakajima varieties. 

Here we are on the middle level of the hierarchy, where the quantum 
group is denoted $\cU_\hbar(\fgh)$. The variable $q$ is the deformation 
parameter; its geometric meaning is the equivariant weight of 
the symplectic form. For quivers of finite type, these are identical 
to Drinfeld-Jimbo quantum groups from textbooks. For other 
quivers, $\cU_\hbar(\fgh)$ is constructed in the style of 
Faddeev, Reshetikhin, and Takhtajan from geometrically constructed
$R$-matrices, see \cite{MO1}. In turn, the construction of $R$-matrices, 
that is, solutions of the Yang-Baxter equations with a spectral 
parameter, rests on the $K$-theoretic version of stable envelopes
of \cite{MO2}. We discuss those in Section \ref{s_stab}.  

Stable envelopes in K-theory differ from their cohomological ancestors
in one important feature: they depend on an additional parameter,
called \emph{slope}, which is a choice of an alcove of a certain 
periodic hyperplane arrangement in $\Pic(X)\otimes_{\Z} \R$. 
This is the same data as appears, for instance,
 in the study of quantization of $X$ 
over a field of large positive characteristic. A technical advantage
of such slope dependence is a factorization of $R$-matrices into 
infinite product of certain \emph{root $R$-matrices}, which
generalizes the classical results obtained by Khoroshkin and 
Tolstoy in the case when the Lie algebra $\fg$ is of finite dimension. 

\subsubsection{}
The identification of the quantum difference connection in term of
$\cU_\hbar(\fgh)$ was long expected to be the lattice part of the 
quantum dynamical Weyl group action on $K_G(X)$. 
For finite-dimensional
Lie algebras, this object was studied by Etingof and Varchenko in 
\cite{EV} and many related papers, and it was shown in \cite{Bal} to 
correctly degenerate to the Casimir connection in the suitable 
limit. 

Intertwining operators between Verma modules, which is 
the main technical tool of \cite{EV}, are only available for real simple 
roots and so don't yield a large enough dynamical Weyl group as 
soon as the quiver is not of finite type. However, using root 
$R$-matrices, one can define and study the quantum dynamical Weyl 
group for arbitrary quivers. This is done in \cite{OS}. 
Once available, a representation-theoretic identification 
of the quantum difference connections gives an essentially 
full control over the theory. 

In these notes we stop where the analysis of \cite{OS} starts: we 
show that quantum connection commutes with qKZ, which is 
one of the key features of dynamical Weyl groups. 

\subsubsection{}
The monodromy of the quantum difference connection is 
characterized in \cite{AO} in terms of an elliptic analog
of $\cU_\hbar(\fgh)$. The categorical meaning of the 
corresponding operators is an area of active current 
research.

\subsubsection{}\label{s_mirror}
These notes are meant to be a partial sample of basic techniques and 
results, and  this is not an attempt to write 
an A to Z technical manual on the subject, nor to present a 
panorama of geometric applications that these techniques 
have. 

For instance, one of the most exciting topics in quantum 
K-theory of symplectic resolutions is the duality, known under 
many different names including ``symplectic duality'' or 
``3-dimensional mirror symmetry'', see \cite{NakCoul} for an introduction.
 Nakajima varieties may 
be interpreted as the moduli space of Higgs vacua in certain 
supersymmetric gauge theories, and the computations in their 
quantum K-theory may be interpreted as indices of the corresponding 
gauge theories on real 3-folds of the form $C \times S^1$. 
A physical duality equates these indices for certain pairs of 
theories, exchanging the K\"ahler parameters on one side with 
the equivariant parameters on the other.

 In the context of these
notes, this means that an exchange of the K\"ahler and 
equivariant difference equations of Section \ref{s_diff}, which 
may be studied as such and generalizes various dualities
known in representation theory. 
This is just one example of a very interesting phenomenon 
that lies outside of the scope of these lectures.

\subsection{K-theoretic 
Donaldson-Thomas theory}

\subsubsection{}

Donaldson-Thomas theory, or DT theory for short,
is an enumerative theory of sheaves on a fixed smooth 
quasiprojective 
threefold $Y$, which need not be Calabi-Yau to point out 
one frequent misconception. There are many categories similar to 
the category of coherent sheaves on a smooth threefold, and one 
can profitably study DT-style questions in that larger 
context. In fact, we already met with such generalizations 
in the form of quantum K-theory of general Nakajima 
quiver varieties. Still, I consider sheaves on threefolds to be the core
object of study in DT theories. 

An example of a sheaf to have in mind could be the structure sheaf $\cO_C$
of a  curve, or more precisely, a $1$-dimensional subscheme, $C\subset Y$. 
The corresponding DT moduli space $\mM$ is the Hilbert scheme 
of curves in $Y$ and what we want to compute from them is the 
K-theoretic version of counting curves in $Y$ of given degree 
and arithmetic genus satisfying some further geometric constraints like 
incidence to a given point or 
tangency to a given divisor.

There exist other enumerative theories of curves, notably the 
Gromov-Witten theory, and, in cohomology, there is a rather 
nontrivial equivalence between the DT and GW counts, first 
conjectured in \cite{mnop1,mnop2} and explored in many papers since. 
At present, it is not known whether the GW/DT correspondence 
may be lifted to K-theory, as one stumbles early on  
trying to mimic the DT computations on the GW side
\footnote{Clearly, the subject of these note has not even 
begun to settle, and our present view of many key phenomena
throughout the paper  
may easily change overnight.}, see for example the discussion 
in Section \ref{s_M0n}. 

\subsubsection{}\label{sL4L5}

Instead, in K-theory there is 
a different set of challenging conjectures \cite{NO} which may 
serve as one of the goalposts for the development of the theory. 

This time the conjectures relate
DT counts of curves in $Y$ to a very different kind of curve counts
% \footnote{At present, the $5$-dimensional curve counting 
% remains in a serious need of foundational work and is, perhaps, 
% not ready for a discussion in these notes.}  
in a Calabi-Yau 5-fold $Z$ which is a total space 
\begin{equation}
Z = 
\begin{matrix}
\cL_4 \oplus \cL_5  \\
\downarrow\\
Y
\end{matrix} \,, \quad \cL_4 \otimes \cL_5 = \cK_Y \,,
\label{ZLL} 
\end{equation}
of a direct sum of two line bundles on $Y$. One interesting 
feature of this correspondence is the following. One the DT 
side, one forms a generating function over all arithmetic 
genera and then its argument  $z$ becomes an element 
$z\in \Aut(Z,\Omega^5)$ which acts by $\diag(z,z^{-1})$ 
in the fibers of $\cL_4 \oplus \cL_5$. Here $\Omega^5$ is 
the canonical holomorphic 5-form on $Z$. 

A K-theoretic curve count in $Z$ is naturally a virtual 
representation of the group $\Aut(Z,\Omega^5)$ and, in particular, 
$z$ has a trace in it which is a rational function. This 
rational function is then equated to something one computes
on the DT side by summing over all arithmetic genera. We see
it is a nontrivial operation and, also, that equivariant 
K-theory is the natural 
setting in which such operations make sense. More general 
conjectures proposed in \cite{NO} similarly identify certain 
equivariant variables for $Y$ with variables that keep track 
of those curve degree for 5-folds that are lost in $Y$. 

For various DT computations below, we will point out their 
5-dimensional interpretation, but this will be the extent of 
our discussion of 5-dimensional curve counting. It is still in its
infancy and not ready to be presented in an introductory lecture
series. It is quite different from either the DT or GW curve 
counting in that it lacks a parameter that keeps track of 
the curve genus. Curves of any genus are counted equally, but 
the notion of stability is set up so that that only finitely many 
genera contribute to any given count. 

\subsubsection{}
When faced with a general threefold $Y$, a natural instinct is to 
try to cut $Y$ into simpler pieces from which the curve counts in $Y$ 
may be reconstructed. There are two special scenarios 
in which this works really well, they can be labeled 
\emph{degeneration} and \emph{localization}. 

In the first scenario, we put $Y$ in a 1-parameter 
family $\widetilde{Y}$ with a smooth total space 
$$
\xymatrix{ 
Y \ar@{^{(}->}[r] \ar[d]& \widetilde{Y} \ar[d] &  Y_1 \cup_D Y_2 \ar[d]
 \ar@{_{(}->}[l] \\
1 \ar@{^{(}->}[r] & \C&  0
 \ar@{_{(}->}[l]
}
$$
so that a special fiber of this family is a union $Y_1 \cup_D Y_2$ 
of two smooth 3-folds along a smooth divisor $D$. 
In this case the curve counts in $Y$ may be reconstructed from 
certain refined curve counts in each of the $Y_i$'s. These refined 
counts keep track of the intersection of the curve with 
the divisor $D$ and are called \emph{relative} DT counts. 
The technical foundations of the subject are laid in \cite{LiWu}. We will
get a sense how this works in Section \ref{s_degen_glue}. 

The work of  Levine and Pandharipande \cite{LevPand} supports the idea that 
using degenerations one should be able to reduce curve counting 
in general 3-folds to that in \emph{toric} 3-folds. Papers 
by Maulik and Pandharipande \cite{MP_top} and by 
Pandharipande and Pixton \cite{PandPix} 
offer spectacular examples of this idea being put into action. 

\subsubsection{}
Curve counting in toric 3-folds may be broken into further pieces 
using equivariant localization. Localization is a general principle by
which all 
computations in $G$-equivariant K-theory of $\mM$ depend only on the 
formal neighborhood of the $T$-fixed locus $\mM^T$, where $T\subset G$ is 
a maximal torus in a connected group $G$. We will rely on localization 
again and again in these notes. Localization is particularly powerful 
when used in \emph{both} directions, that is, going from the global 
geometry to the localized one and back, because each point of view
has its own advantages and limitations. 

A threefold $Y$ is toric if $T\cong(\Ct)^3$ acts with an open orbit on
$Y$. It then follows that $Y$ has finitely many orbits and, in
particular, 
finitely many orbits of dimension
$\le 1$. Those are the fixed points and the $T$-invariant curves, and
they correspond to the 1-skeleton of the toric polyhedron of $Y$.  
{}From the localization viewpoint, $Y$ may very well be replaced
by this 1-skeleton. All nonrelative curve counts in $Y$ may be done in terms 
of certain 3- and 2-valent tensors, called vertices and edges, 
associated to fixed points and $T$-invariant curves, respectively. 
See, for example, \cite{ECM} for a pictorial introduction. 

The underlying vector space for these tensors is 
\begin{itemize}
\item[---] the equivariant K-theory of $\Hilb(\C^2,\textup{points})$, or equivalently
\item[---] the standard Fock space, or the algebra of symmetric
  functions, 
\end{itemize}
with an extension of scalars to include all equivariant variables as 
well as the variable $z$ that keeps track of the arithmetic genus. 
Natural bases of this vector space are indexed by partitions 
and curve counts 
are obtained by contracting all indices, in great similarity to many TQFT
computations. 

 In the basis of torus-fixed points of the Hilbert
scheme, edges are simple diagonal tensors, while vertices are
something
complicated. More sophisticated bases spread the complexity more
evenly. 

\subsubsection{}
These vertices and edges, and related objects, are the nuts
and bolts of the theory and the ability to compute with them 
is a certain measure of professional skill in the subject. 

A simple, but crucial observation is that the geometry of ADE 
surface fibrations captures
all these vertices and edges \footnote{In fact, formally, 
it suffices to understand $A_n$-surface fibrations with $n\le 2$.} . This 
bridges DT theory with topics discussed in Section \ref{s_Q_Nak}, 
and was already put to a very good use in \cite{moop}. 

In \cite{moop}, there are two kind of 
vertices: \emph{bare}, or standard, and \emph{capped}. 
It is convenient 
to extend such taxonomy to general Nakajima varieties (or to general quasimap
problems, for that matter).  For general Nakajima varieties, in
cohomology, the 
notion of a bare 1-leg vertex coincides with Givental's notion of
I-function, and there is no real analog of two- or three-legged 
vertex for general Nakajima varieties 
\footnote{The Hilbert scheme $\Hilb(A_{n-1})$ of the 
$A_{n-1}$-surface is dual, in the sense of Section \ref{s_mirror},  
to the moduli space $M(n)$ of framed sheaves of rank $n$ of
$A_{0}\cong \C^2$, which is a Nakajima variety for the quiver 
with one vertex and one loop. 
The splitting of the 1-leg vertex for the $A_{n-1}$-surface 
into $n$ simpler vertices is a phenomenon which is dual
to $M(n)$ being an $n$-fold tensor product of $\Hilb(\C^2)$ in 
the sense of \cite{NakCoul}. For a general Nakajima variety, there is no 
direct analog of this.}. 

Bare and capped vertices are
the same tensors expressed in two different bases, and which have 
different geometric meaning and different properties. 
In particular, capped $1$-leg vertices are trivial, in that the 
contributions of all subscheme of nonminimal Euler characteristic 
cancel out. One can thus determine all vertices if one knows 
the transition matrix from bare vertices to the capped ones, 
commonly called the capping operator. The theory is built 
so that this operator is the fundamental solution of the 
quantum differential equation and this is how \cite{moop} works. 

\subsubsection{}

All these notions have a direct analog in K-theory and are
the subject of Section \ref{s_NB}. In fact, in the lift from 
cohomology to K-theory there is always at least a line 
bundle worth of natural ambiguities, and here we work with 
a certain specific twist $\tO_\vir$ of the virtual structure
sheaf $\cO_\vir$. This object, which we call the symmetrized virtual 
structure sheaf, is discussed at many points in these lectures 
and has many advantages over the ordinary virtual structure 
sheaf. {}From the point of view of mathematical physics, $\tO_\vir$
rather than $\cO_\vir$ is related to the index of the relevant Dirac
operator, and the roof over it is put to remind us of that. 
The importance of working with $\tO_\vir$ was emphasized 
already by Nekrasov in \cite{Zth}. 

With this, one defines the bare and capped vertices as before 
and sees that the capping operator is the fundamental 
solution of the quantum difference equation. The specific choice of 
$\tO_\vir$ becomes important in the proof of triviality of $1$-leg 
capped vertices, as this turns out to be a rigidity result which 
uses a certain self-duality of $\tO_\vir$ in a very essential way. 

The same argument, in fact, proves more, it proves a similar 
triviality of capped vertices with descendents (an object also 
discussed in Section \ref{s_NB}), once the framing is taken
sufficiently large with respect to the descendent insertion. 
This makes it possible to determine bare vertices with 
descendents as well as the explicit correspondence between the 
descendents and relative insertions and is an area of active 
current research. The latter correspondence plays a very 
important role, in particular, in reducing DT counts in general 
threefolds to those in toric varieties as in \cite{PandPix} and 
it would be very desirable to have an explicit description of it.

\subsection{Old vs.\ new} 

\subsubsection{}

These lectures are written for graduate students in the field
and many of the pages that follow discuss topics that have 
become classical or at least standard in the subject. I tried
to present them from an angle that is best suited for the 
applications I had in mind, but there is no way I could have
achieved an improvement over several existing treatments 
in the literature. 

In particular, my favorite  introduction to equivariant 
K-theory is Chapter 5 of the book \cite{CG} by Chriss and Ginzburg. 
For a better introduction to Nakajima quiver varieties, see
the lectures \cite{GinzNak} by Ginzburg and, of course, the original 
papers \cite{Nak1,Nak2} by Nakajima. For a discussion of 
moduli spaces of quasimaps, I recommend the reader opens 
the paper \cite{CKM}. For a general introduction to 
enumerative geometry, my favorite source are the lectures
by Pandharipande in \cite{Mirror}, even though those 
lectures discuss topics that are essentially disjoint from 
our goals here. 

\subsubsection{}
There is also a number of results in these notes that are new
or, at least, not available in the existing literature. Among them  
are the following. 

In Theorem \ref{t_Hilb3}, we prove a conjecture of Nekrasov 
\cite{Zth} on the K-theoretic degree 0 DT invariants of 
smooth projective threefolds. This generalizes the 
cohomological result of \cite{mnop2,LevPand}. 

In Theorem \ref{t_Sn}, we prove the main conjecture of 
\cite{NO} for reduced smooth curves in threefolds. This is, of course, 
a very special case of the general conjectures made in \cite{NO}, but 
perhaps adds to their credibility. In fact, we prove a statement on 
the level of derived categories of coherent sheaves and it would be 
extremely interesting to see if there is a similar upgrade of the 
K-theoretic conjectures made in \cite{NO}. 

In Theorem \ref{t_large_framing}, we prove what we call the 
\emph{large framing vanishing}, which is the absence of 
quantum corrections to the capped vertex with descendents 
once the framing (and polarization) is chose sufficiently 
large with respect to the given descendent. This has direct 
applications to an effective description of the correspondence 
between relative and absolute insertions in K-theoretic DT 
theory and is an area of active current research. 

In Theorems \ref{t_bfM} and \ref{t_bfS} we give a simple 
construction of the difference operators that shift the 
K\"ahler and equivariant variables in certain building blocks 
of the theory. Many yeas ago, Givental initiated a very broad 
program of developing quantum K-theory and, in particular, 
the paper \cite{GivTon} gives a very general proof of the 
existence of quantum difference equations in quantum K-theory. 
In this paper, we work with quasimaps instead of the moduli 
spaces of stable maps, but it is entirely possible that the 
methods of \cite{GivTon} can be adapted to work in our situation. 
However, our situation is much more special than that in
\cite{GivTon} because the classes $\tO_\vir$ 
in the K-theory of the moduli 
spaces that we consider behave
much better than the virtual structure sheaves.  This opens 
the door to many arguments that would not be available in general. 
Most importantly, we make a systematic use of self-dual 
features of $\tO_\vir$, which often result in rigidity. 
In particular, the large framing vanishing is fundamentally 
a rigidity result. 

Also, the special features of the 
theory presented here allow for an \emph{explicit}
determination of the corresponding 
difference equations in the language of representation 
theory. In particular, among the operators 
that shift the equivariant variables we find the quantum 
Knizhnik-Zamolodchikov equations, while the shifts of the 
K\"ahler variables are shown in \cite{OS} to come from the 
corresponding quantum dynamical Weyl group. 
The quantum 
Knizhnik-Zamolodchikov equations appear for a shift by 
a \emph{minuscule} cocharacter of the torus. This is a 
direct K-theoretic analog of the computation in \cite{MO1}, 
although it requires a significantly more involved argument to 
be shown. This is done in Theorem \ref{t_R}.

\subsection{Acknowledgements} 

\subsubsection{}
Both the old and the new parts of these notes owe a lot to 
many people. In the exposition of the fundamentals, I have 
been influenced by the sources cited above as well by 
countless other papers, lectures, and conversations. 
Evidently, I lack both the space and the qualification to make a 
detailed spectral analysis of these influences. For the same reason, 
I have made no attempt to put the material into anything resembling a 
historical perspective. 

In the new part, I discuss bits of a rather large research 
project, various parts of which are joint with M.~Aganagic, 
R.~Bezrukavnikov, D.~Maulik, N.~Nekrasov, A.~Smirnov, and others. 
It is hard to overestimate the influence this collaboration 
had on me and on my thinking about the subject. I must add 
P.~Etingof to this list of major inspirations. 

\subsubsection{}

I am very grateful to the Simons foundation for being financially 
supported as a Simons investigator and to the Clay foundation 
for supporting me as a Clay Senior Scholar during PCMI 2015. 

These note are based on the lectures I gave at Columbia in Spring of
2015 and at the Park City Mathematical Institute in Summer of 2015. 
I am very grateful to the participants of these lectures and to PCMI 
for its warm hospitality and intense intellectual atmosphere.
 Special thanks are to Rafe Mazzeo, Catherine
Giesbrecht, Ian Morrison, and the anonymous referee.

\section{Before we begin}

The goal of this section is to have a brief abstract discussion 
of several construction in equivariant K-theory which will appear 
and reappear in more concrete situations below. This 
section is not meant to be an introduction to equivariant 
K-theory; Chapter 5 of \cite{CG} is highly recommended for that.

\subsection{Symmetric and exterior algebras}\label{s_bSd} 

\subsubsection{}

Let $T$ be a torus and $V$ a finite dimensional $T$-module. 
Clearly, $V$ is a direct sum of 1-dimensional $T$-modules, which are 
called the \emph{weights} of $V$. The weights 
$\mu$ are recorded by the character of $V$
$$
\bbchi_V(t) = \tr_V t = \sum t^{\mu}\,, \quad t \in \bT\,, 
$$
where repetitions are allowed among the $\mu$'s. 

We denote by $K_T$ or $K_T(\pt)$ the K-group of the category of 
$T$-modules. Of course, any exact sequence 
$$
0 \to V' \to V \to V'' \to 0
$$
is anyhow split, so there is no need to impose the relation 
$\left[V\right] = \left[V'\right] + \left[V''\right]$ in this case. 
The map $V \mapsto \bbchi_V$ gives an isomorphism 
$$
K_T \cong \Z[t^\mu]\,,
$$
with the group algebra of the 
character lattice of the torus $T$.  Multiplication in $\Z[t^\mu]$ 
corresponds to $\otimes$-product in $K_T$. 

\subsubsection{}
For $V$ as above, we can form its 
symmetric powers $\bS^2 V,\bS^3 V, \dots$, including 
$S^0 V = \C$.  These are
$GL(V)$-modules and hence also $\bT$ modules. 

\begin{Exercise}
Prove that 
\begin{align}
\sum_{k\ge 0}\, s^k\, \bbchi_{\bS^k V}(t) & = 
\prod \frac{1}{1-s\, t^\mu}  \notag \\
&= \exp \sum_{n>0} \frac1n \, s^n \, \bbchi_V(t^n) \,. \label{e_SV}
\end{align}
\end{Exercise}

We can view the functions in \eqref{e_SV} as an element of
$K_\bT[[s]]$ or as a character of an infinite-dimensional graded 
$\bT$-module $\bSd V$ with finite-dimensional graded subspaces, 
where $s$ keeps track of the degree. 

For the exterior powers we have, similarly, 

\begin{Exercise}
Prove that 
\begin{align}
\sum_{k\ge 0}\, (-s)^k\, \bbchi_{\Lambda^k V}(t) & = 
\prod {(1-s\, t^\mu)}  \notag \\
&= \exp \left(-\sum_{n>0} \frac1n \, s^n \, \bbchi_V(t^n) 
\right) \,. \label{e_LV}
\end{align}
\end{Exercise}

The functions in \eqref{e_SV} and \eqref{e_LV} are reciprocal of 
each other. The representation-theoretic object behind 
this simple observation is known as the 
\emph{Koszul complex}. 

\begin{Exercise}
Construct an $GL(V)$-equivariant exact sequence 
\begin{equation}
\dots \to \Lambda^2 V \otimes \bSd V \to 
\Lambda^1 V \otimes \bSd V \to 
 \bSd V \to \C \to 0 \label{Koz} \,,
\end{equation}
where $\C$ is the trivial representation. 
\end{Exercise}

\begin{Exercise}\label{KozV0} 
Consider $V$ as an algebraic variety on which $GL(V)$ acts. 
Construct a $GL(V)$-equivariant resolution of the 
structure sheaf of $0\in V$ by vector bundles on $V$. 
Be careful not to get \eqref{Koz} as your answer. 
\end{Exercise}

\subsubsection{}
Suppose $\mu=0$ is not a weight of $V$, which means that 
\eqref{e_SV} does not have a pole at $s=1$. Then we can 
set $s=1$ in \eqref{e_SV} and define 
\begin{equation}
\bbchi_{\bSd V}  = \prod \frac{1}{1-t^\mu} = \exp 
\sum_n \frac1n \,  \bbchi_V \, t^n \,, \label{SV} 
\end{equation}
This is a well-defined element of \emph{completed} 
K-theory of $\bT$ provided 
$$
\| t^\mu \| < 1 
$$
for all weights of $V$ with respect to some norm $K_\bT$. 

Alternatively, and with only the $\mu\ne 0$ assumption 
on weights, \eqref{SV} is a well-defined element of the 
\emph{localized} K-theory of $\bT$ 
$$
K_{\bT,\textup{localized}} = \Z\left[t^\nu,\frac{1}{1-t^\mu}\right]
$$
where we invert some or all elements $1-t^\mu \in K_\bT$. 

Since 
$$
K_\bT  \hookrightarrow K_{\bT,\textup{localized}} 
$$
characters of finite-dimensional modules may be computed 
in localization without loss of information. However, certain different 
infinite-dimensional modules become the same in 
localization, for example 

\begin{Exercise}
For $V$ as above, check that 
\begin{equation}
\bSd V = (-1)^{\rk V}  \det V^\vee \otimes \bSd V^\vee\label{Sddual} 
\end{equation}
in localization of $K_\bT$. 
\end{Exercise}

\noindent 
We will see, however, that this is a feature rather than a bug. 

\begin{Exercise}
Show that $\bSd$ extends to a map 
\begin{equation}
K_\bT' \xrightarrow{\quad \bSd \quad} K_{\bT,\textup{localized}} 
\label{SV2}
\end{equation}
where prime means that there is no zero weight, which satisfies
$$
\bSd (V_1 \oplus V_2) = \bSd V_1 \otimes \bSd V_2 \,, 
$$
and, in particular, 
\begin{equation}
\bSd(-V) = \Ld \, V = \sum_i (-1)^i \Lambda^i V  \,.  \label{defLd} 
\end{equation}
\end{Exercise}

\noindent 
Here and in what follows, 
the symbol $\Ld \, V$ is defined by \eqref{defLd} as the
alternating sum of exterior powers. 

\subsubsection{}

The map \eqref{SV2} may be extended by 
continuity to a completion of $K_\bT$ with respect to a suitable
norm. This gives a compact 
way to write infinite products, for example 
$$
\bSd \frac{a-b}{1-q} = \prod_{n\ge 0} \frac{1-q^n b}{1-q^n a}  \,, 
$$
which converges in any norm $\|\,\cdot\,\|$ such that $\|q\|<1$.

\begin{Exercise}
Check that 
$$
\bSd \frac{a}{(1-q)^{k+1}}= \prod_{n\ge 0} (1-q^n a)^{-\binom{k+n}{n}}\,. 
$$
\end{Exercise}

\subsubsection{}

The map $\bSd$ is also known under many aliases, including \emph{plethystic
exponential}.  Its inverse is known, correspondingly, as the 
plethystic logarithm.

\begin{Exercise}
Prove that the inverse to $\bSd$ is given by the formula 
$$
\bbchi_V(t) = \sum_{n>0} \frac{\mu(n)}{n}  \ln \bbchi_{\bSd V}(t^n)
$$
where $\mu$ is the \emph{M\"obius function} 
$$
\mu(n) = 
\begin{cases}
(-1)^{\textup{\# of prime factors}} \,, &  \textup{$n$ squarefree}\\
0 \,, &  \textup{otherwise} \,. 
\end{cases}
$$
The relevant property of the M\"obius function is that it 
gives the matrix elements of $C^{-1}$ where the matrix 
$C=(C_{ij})_{i,j\in\N}$ is defined by 
$$
C_{ij} =
\begin{cases}
  1 \,, & i|j \\
 0 \,, & \textup{otherwise} \,. 
\end{cases}
$$
In other words, $\mu$ is the M\"obius function of the set $\N$ 
partially ordered by divisibility, see \cite{StEn}. 
\end{Exercise}

\subsubsection{}\label{s_ShV} 
If the determinant of $V$ is a square as a character of $\bT$, we define 
\begin{equation}
 \bSdh V = \left(\det V \right)^{1/2} \,  \bSd  \label{def_bSdh}
\end{equation}
which by \eqref{Sddual} satisfies
$$
 \bSdh V^\vee = (-1)^{\rk V} \,  \bSdh V \,.
$$
Somewhat repetitively,  it may be 
called the symmetrized symmetric algebra.

\subsection{$K_G(X)$ and $K^\circ_G(X)$} 

\subsubsection{}

Let a reductive group $G$ act on a scheme $X$. We denote by 
$K_G(X)$
 the K-group of the category of $G$-equivariant 
coherent sheaves on
$X$. Replacing general coherent sheaves by locally free ones, that is, 
by $G$-equivariant 
vector bundles on $X$, gives another group $K^\circ_G(X)$ with 
a natural homomorphism 
\begin{equation}
K^\circ_G(X) \to K_G(X)\label{e_circ_K} \,. 
\end{equation}
Remarkably and conveniently, \eqref{e_circ_K} is 
 is an isomorphism if $X$ is nonsingular.  
In other words, every coherent sheaf on a nonsingular variety is 
\emph{perfect}, which means it admits a locally free resolution of 
finite length, see for example Section B.8 in \cite{Ful}. 

\begin{Exercise}\label{ExerTor}
 Consider $X=\{x_1 x_2 = 0 \} \subset \C^2$ with the action of the 
maximal torus 
$$
T= 
\left\{
  \begin{pmatrix}
    t_1 \\
 & t_2 
  \end{pmatrix} \right\}
\subset GL(2)\,.
$$
Let $\cF=\cO_0$ be the structure 
sheaf of the origin $0\in X$. Compute the minimal $T$-equivariant 
resolution
$$
\dots \to \cR^{-2} \to \cR^{-1} \to \cR^0 \to \cF \to 0 
$$
of $\cF$ by sheaves of the form 
$$
\cR^{-i} = \cO_X \otimes R_i \,,
$$
where $R_i$ is a finite-dimensional $T$-module. 
Observe from the resolution that the groups 
\begin{equation}
\Tor_i(\cF,\cF) \eqdef  H^{-i}(\cR^\bullet \otimes \cF) = R_i   \label{Tor_i}
\end{equation}
are nonzero for all $i\ge 0$ and conclude that $\cF$ is not in the image of 
$K^\circ_T(X)$. Also observe that
\begin{equation}
  \sum (-1)^i \bbchi_{\Tor_i(\cF,\cF)} = 
\frac{\chi(\cF)^2}{\chi(\cO_X)} = 
\frac{(1-t_1^{-1})(1-t_2^{-1})}{1-t_1^{-1} t_2^{-1}}  
\label{TorFF} 
\end{equation}
expanded in inverse powers of $t_1$ and $t_2$.
\end{Exercise} 

\begin{Exercise}
Generalize \eqref{TorFF} to the case 
\begin{align*}
X &= \Spec \C[x_1,\dots,x_d] / I  \,,\\
\cF &= \C[x_1,\dots,x_d] / I' \,, 
\end{align*}
where $I\subset I'$ are monomial ideals, that is, ideals generated by 
monomials in the variables $x_i$.
\end{Exercise}

\subsubsection{}

The domain and the source of the map \eqref{e_circ_K} have 
different functorial properties with respect to 
$G$-equivariant morphisms
\begin{equation}
f: X \to Y \label{fXY} 
\end{equation}
of schemes. 

The pushforward of a K-theory class $\left[\cG\right]$ 
represented by a coherent sheaf $\cG$ is defined as
\begin{equation}
f_* 
\left[\cG \right] = \sum_i (-1)^i \left[ R^if_* \cG \right] \,. 
\label{K_push} 
\end{equation}
We abbreviate 
$f_*\cG=f_* 
\left[\cG \right]$ 
in what follows.

The length of the 
sum in \eqref{K_push} is bounded, e.g.\ by the dimension of $X$,
 but the terms, 
in general, are only quasicoherent sheaves on $Y$.  If $f$ is proper 
on the support of $\cG$ then this ensures $f_* \cG$ is coherent 
and thus lies in $K_G(Y)$. 
Additional hypotheses are required to conclude $f_* \cG$ 
is perfect. For an example, take $\iota_* \cO_0$ where 
$$
\iota: \{0\} \hookrightarrow \{x_1 x_2 =0 \}
$$
is the inclusion in Exercise \ref{ExerTor}.

\begin{Exercise}\label{ex_chiPn} 
The group $GL(2)$ acts naturally on $\bP^1=\bP(\C^2)$ and 
on line bundles $\cO(k)$ over it.  Push forward
these line bundles under $\bP^1 \to \pt$ using an explicit 
$T$-invariant \v Cech covering of $\bP^1$. Generalize to 
$\bP^n$. 
\end{Exercise}

\subsubsection{}

The pull-back of $\left[\cE\right] \in K_G(Y)$ is defined by 
$$
f^* \left[\cE\right]  = \sum (-1)^i \left[\cTor_i^{\cO_Y}
  \!(\cO_X,\cE) \right]  \,.
$$
Here the terms are coherent, but there may be infinitely many 
of them, as is the case for $\iota^* \cO_0$ in our running example. 
To ensure the sum terminates we need some flatness assumptions, 
such as $\cE$ being locally free. In particular, 
$$
f^*: K^\circ_G(Y) \to K^\circ_G(X)
$$
is defined for arbitrary $f$ by simply pulling back vector bundles. 

\begin{Exercise}
Globalize the computation in Exercise \ref{KozV0} to compute 
$\iota^* \cO_X \in K_G(\cO_X)$ for a $G$-equivariant inclusion
$$
\iota: X \hookrightarrow Y 
$$
of a nonsingular subvariety $X$
 in a nonsingular variety
$Y$. 
\end{Exercise}

\subsubsection{}

Tensor product makes $K^\circ_G(X)$ a ring and 
$K_G(X)$ is a module over it. The projection formula
\begin{equation}
f_*(\cF \otimes f^* \cE) = f_*(\cF) \otimes \cE \label{proj_form} 
\end{equation}
expresses the covariance of this module structure with respect to 
morphisms $f$.

\begin{Exercise}
Write a proof of the projection formula. 
\end{Exercise}

Projection formula can be used to prove that a proper 
pushforward $f_* \cG$ is perfect if $\cG$ is flat over $Y$, 
see Theorem 8.3.8 in \cite{FDA}.

\subsubsection{}
Let $X$ be a scheme and $X'\subset X$ a 
closed $G$-invariant subscheme. Then the sequence 
\begin{equation}
K_G(X') \to K_G(X) \to K_G(X \setminus X') \to 0 
\label{exK}
\end{equation}
where all maps are the natural pushforwards, 
is exact, see e.g.\ Proposition 7 in \cite{BS} 
for a classical discussion. This is the beginning of a long exact 
sequence of higher K-groups.  

\begin{Exercise}\label{Ex_exK}
For $X=\C^n$, $X'=\{0\}$, and $T\subset GL(n)$ the maximal torus, 
fill in the question marks in the following diagram 
$$
\xymatrix{
K_T(X') \ar[r] \ar[d]_{\sim} & K_T(X) \ar[r] \ar[d]_{\sim}
&  K_T(X\setminus X') \ar[d]_{\sim} \ar[r] & 0\\
\Z[t^\mu] \ar[r]^{\textup{?}} & \Z[t^\mu] \ar[r] & \textup{?} \ar[r] &
0
}
$$
in which the vertical arrows send the structure sheaves to $1\in
\Z[t^\mu]$. 
\end{Exercise}

In particular, since $X \setminus X_\textup{red}  = \varnothing$, 
the sequence \eqref{exK} implies 
$$
K_G(X_\textup{red}) \cong K_G(X) 
$$
where $X_\textup{red} \subset X$ is the reduced subscheme, whose  
sheaf of ideals $\cI\subset \cO_X$ is formed by nilpotent elements. 
Concretely, any coherent sheaf $\cF$ has a finite filtration 
$$
\cF \supset \cI \cdot \cF \supset \cI^2 \cdot \cF \supset \dots 
$$
with quotients pushed forward from $X_\textup{red}$. 

\subsubsection{}

One can think about the sequence \eqref{exK} like this. Let 
$\cF_1$ and $\cF_2$ be two coherent 
sheaves on $X$, together with an isomorphism 
$$
s:  \cF_2 \big|_U \xrightarrow{\,\sim\,} \cF_1\big|_U
$$
of their restriction to the open set $U = X\setminus X'$. 
Let $\iota: U \to X$ denote the inclusion  and let 
$$
\whcF \subset \iota_* \iota^* \cF_1 
$$
be the subsheaf generated by the natural maps 
$$
\cF_1 \to \iota_* \iota^* \cF_1 \,, \quad 
\cF_2 \xrightarrow{\,s\,} \iota_* \iota^* \cF_1  \,. 
$$
Of course, $ \iota_* \iota^* \cF_1$ is only a quasicoherent sheaf on $X$, 
which is evident in the simplest example $X=\bbA^1$, 
$X'=\textup{point}$, $\cF_1 = \cO_X$. 
However, the sheaf $\whcF$ is generated
by the generators of $\cF_1$ and $\cF_2$, and hence 
coherent.  

By construction, the kernels and cokernels of the natural maps
$$
f_i : \cF_i \to \whcF 
$$
are supported on $X'$. Thus 
$$
 \cF_1 - \cF_2 = \Coker f_1 - \Ker f_1 + \Ker f_2 - \Coker
 f_2 
$$
is in the image of $K(X')\to K(X)$. 

\subsubsection{}

\begin{Exercise}
Let $G$ be trivial and 
let $\cF$ be a coherent sheaf on $X$ with support $Y\subset X$.
Let $\cE$ be a vector bundle on $X$ of rank $r$. Prove that there
exists $Y'\subset Y$ of codimension $1$ such that 
$$
\cE\otimes \cF - r\, \cF 
$$
is in the image of $K(Y') \to K(X)$. 
\end{Exercise}

This exercise illustrates a very useful finite filtration on $K(X)$ formed 
by the images of $K(Y)\to K(X)$ over all subvarieties of given 
codimension. 

\begin{Exercise}\label{ex_nilp} 
Let $G$ be trivial and 
$\cE$ be a vector bundle on $X$ of rank $r$. Prove that $(\cE - r)
\otimes$  is nilpotent as an operator on $K(X)$. 
\end{Exercise}

\begin{Exercise}
Take $X=\bP^1$ and $G=GL(2)$. Compute the minimal polynomial 
of the operator $\cO(1) \otimes$ and see, in particular, that it is 
not unipotent. 
\end{Exercise}

\begin{Exercise}\label{ex_otimes_cL} 
In general, if $G$ is connected and $\cL$ is a line bundle on $X$, then all eigenvalues of 
$\cL\otimes$ on $K_G(X)$ are the weights of $\cL$ at the fixed 
points of a maximal torus of $G$. 
\end{Exercise}

\subsection{Localization}

\subsubsection{}

Let a torus $T$ act on a scheme $X$ and let $X^\bT$ be subscheme 
of $\bT$-fixed points, that is, let 
$$
\cO_X \to \cO_{X^T} \to 0 
$$
be the largest quotient on which $T$ acts trivially. For what follows,
both $X$ and $X^T$ may be replaced by their reduced subschemes.

Consider the kernel and cokernel of the map 
$$
\iota_* : K_T(X^T) \to K_T(X)\,.
$$
This kernel and cokernel are $K_T(\pt)$-modules and have 
some support in the torus $T$.  A very general 
localization theorem of Thomason \cite{Thomason} states 
\begin{equation}
\supp \Coker \iota_*  \subset \bigcup_{\mu} \, \{ t^\mu =1 \} 
\label{localiz1} 
\end{equation}
where the union over finitely many nontrivial 
characters $\mu$. The same is 
true of $\Ker \iota_*$, but since 
\begin{equation}
K_T(X^T) = K(X) \otimes_\Z K_T(\pt) \label{KTXT} 
\end{equation}
has no such torsion, this forces $\Ker \iota_*=0$. To summarize, 
$\iota_*$ becomes an isomorphism after inverting finitely 
many coefficients of the form $t^\mu -1$. 

This localization theorem is an algebraic analog of the 
classical localization theorems in topological K-theory that go back 
to \cite{AB,Segal}. 

\begin{Exercise}
Compute $\Coker \iota_*$ for $X=\bP^1$ and $T\subset GL(2)$ 
the maximal torus. Compare your answer with what you computed
in Exercise \ref{Ex_exK}. 
\end{Exercise}

\subsubsection{}

For general $X$, it is not so easy to make the localization theorem 
explicit, but a very nice formula exists if $X$ is nonsingular. 
This forces $X^\bT$ to be also nonsingular. 

Let $N=N_{X/X^T}$ 
denote the normal bundle to $X^T$ in $X$. The total 
space of $N$ has a natural action of $s\in \Ct$ by scaling 
the normal directions. Using this scaling action, we may 
define 
\begin{equation}
\cO_{N,\textup{graded}} = \sum_{k=0} s^{-k} \, \bS^k N^\vee = 
\bigotimes_\mu \frac1{\Ld \left(s^{-1} t^{-\mu} N_\mu^\vee \right)}
\in K_T(X^T) [[s^{-1}]] \,, 
\label{cOgrad} 
\end{equation}
where 
$$
N = \bigoplus t^\mu N_\mu
$$
is the decomposition of $N$ into eigenspaces of $T$-action
according to \eqref{KTXT}.

\begin{Exercise}
Using Exercise \ref{ex_nilp} prove that 
$$
\cO_{N} = \cO_{N,\textup{graded}}  
\Big|_{s=1} = \bSd N^\vee \in K^\circ_\bT(X^\bT) \left[\frac1{1-t^\mu}\right] 
$$
where $\mu$ are the weights of $T$ in $N$. 
\end{Exercise}

\begin{Exercise}
Prove that 
$$
\iota^* \iota_* \cG = \cG \otimes \Ld \, N^\vee 
$$
for any $\cG\in K_T(X^T)$ and that the operator 
$\Ld \, N^\vee \otimes$ becomes an isomorphism after 
inverting $1-t^\mu$ for all weights of $N$.  Conclude 
the localization theorem \eqref{localiz1} implies
\begin{equation}
\iota_*^{-1} \cF = \bSd N^\vee \otimes \iota^* \cF 
\label{equiv_localiz} 
\end{equation}
for any $\cF$ in localized equivariant K-theory. 
\end{Exercise}

\subsubsection{}

A $T$-equivariant map $f: X \to Y$ induces a diagram 
\begin{equation}
\label{fT} 
\xymatrix{
X^T \ar[rr]^{\iota_X} \ar[d]_{f^T}  && X \ar[d]^f \\
Y^T \ar[rr]^{\iota_Y} && Y
}
\end{equation}
with 
\begin{equation}
f_* \circ \iota_{X,*} = \iota_{Y,*} \circ f_*^T  \,. \label{iffi} 
\end{equation}
Normally, we don't care much about torsion, or we may know 
ahead of time that there is no torsion in $f_*$, like when $f$ is a 
proper map to a point, or some other trivial $T$-variety. 
Then, we can write 
\begin{equation}
f_* = \iota_{Y,*} \circ f_*^T  \circ \iota_{X,*}^{-1} \,. \label{fifi} 
\end{equation}
This is what it means to compute the pushforward by localization. 

\begin{Exercise}\label{Ex_loc_Pn} 
Redo Exercise \ref{ex_chiPn}, that is, 
compute $\chi(\bP^n,\cO(k))$ by localization. 
\end{Exercise}

\begin{Exercise}\label{Ex_loc_GB} 
Let $G$ be a reductive group and $X=G/B$ the corresponding flag
variety. Every character $\lambda$ of the maximal torus $T$ gives 
a character of $B$ and hence a line bundle 
$$
\cL_\lambda :   \left(G \times \C_\lambda\right)/B \to X 
$$
over $X$. Compute $\chi(X,\cL_\lambda)$ by localization. A theorem 
of Bott, see e.g.\ \cite{Demaz}
 states that at most one cohomology group $H^i(X,\cL_\lambda)$ is 
nonvanishing, in which case it is an irreducible representation of $G$. 
So be prepared to rederive Weyl character formula from your
computation. 
\end{Exercise}

\begin{Exercise}
Explain how Exercise \ref{Ex_loc_Pn}  is a special case of 
Exercise \ref{Ex_loc_GB}. 
\end{Exercise}

\subsubsection{}

Using \eqref{fifi}, one may \emph{define} pushforward 
$f_*\cF$ in localized equivariant 
cohomology as long as $f^T$ is proper on $(\supp \cF)^T$. 
This satisfies all usual properties and leads to meaningful 
results, like 
$$
\chi(\C^n, \cO) = \prod_i \frac1{1-t_i^{-1}} 
$$
as a module over the maximal torus $T\subset GL(n)$.

\subsubsection{}

The statement of the localization theorem goes over with little 
or no change to certain more general $X$, for example, to 
orbifolds. Those are modelled locally on $\tX/\Gamma$, where 
$\tX$ is nonsingular and $\Gamma$ is finite. By definition, 
coherent sheaves on $\tX/\Gamma$ are $\Gamma$-equivariant 
coherent sheaves on $\tX$. 

A torus action on $\tX/\Gamma$ is a $T\times \Gamma$ action 
on $\tX$ and, in particular, the normal bundle $N_{\tX/\tX^T}$ is 
$\Gamma$-equivariant, which means it descends to 
 to an orbifold normal bundle to $\left[\tX/\Gamma\right]^T$. 

\begin{Exercise}
For $a,b>0$, consider the weighted projective line 
$$
X_{a,b} = \C^2 \setminus \{0\} \left/ 
\begin{pmatrix}
z^a \\
& z^b 
\end{pmatrix}\,, \right. \quad z \in \Ct \,. 
$$
Show it can be covered by two orbifold charts. Like any
$\Ct$-quotient, it inherits an orbifold line bundle $\cO(1)$ whose sections 
are functions $\phi$ on the prequotient such that 
$$
\phi (z\cdot x) = z\,  \phi(x) \,.
$$
Show that 
\begin{equation}
\sum_{k\ge 0} \chi(X_{a,b}, \cO(k)) \, s^k = 
\frac{1}{(1-t_1^{-1} s^a) (1-t_2^{-1} s^b)}\label{skcoeff} 
\end{equation}
as a module over diagonal matrices.  Compute 
$\chi(X_{a,b}, \cO(k))$ by localization. 
Compare your answer to the computation 
of the $s^k$-coefficient in \eqref{skcoeff} by residues. 
\end{Exercise}

\subsubsection{}
What we will really need in these lectures is the \emph{virtual 
localization formula} from \cite{GP,CF3}. It will be discussed 
after we get some familiarity with virtual classes. 

In particular, in this greater generality 
 the normal bundle $N$ to the fixed locus is a virtual 
vector bundle, that is, an element of $K^\circ_T(X^T)$ of the form 
$$
N = N_{\Def}- N_{\Obs}\,,
$$
where the $N_{\Def}$ is responsible for first-order deformations,
while $N_{\Obs}$ contains obstructions to extending those. 
Naturally, 
$$
\bSd N^\vee = \bSd N_{\Def}^\vee \otimes \Ld \,N_{\Obs}^\vee \,,
$$
so a virtual localization formula has both denominators and
numerators.

\subsection{Rigidity}\label{s_princ} 

\subsubsection{}

If the support of a $T$-equivariant sheaf $\cF$ is proper then 
$\chi(\cF)$ is an element of $K_T(\pt)$ and so a Laurent 
polynomial in $t\in T$. In general, this polynomial is 
nontrivial which, of course, is precisely
 what makes equivariant K-theory interesting. 

However, for the development of the theory, 
one would like its certain building blocks 
to depend on few or no equivariant variables. 
This phenomenon is known as rigidity. 
A classical \cite{AtHirz,Kr1,Kr2} and surprisingly effective way to show rigidity 
 is to use 
the following elementary observation: 
$$
    \textup{$p(z)$ is bounded as $z^{\pm 1} \to \infty$} 
\quad 
\Leftrightarrow 
\quad 
p = \textup{const} 
$$
for any $p \in \C[z^{\pm 1}]$. The behavior of $\chi(\cF)$ at 
the infinity of the torus $T$ can be often
 read off directly 
from the localization formula. 

\begin{Exercise}\label{Ex_mOmega}
Let $X$ be proper and smooth with an action of a 
connected reductive group $G$. Write a localization formula
for the action of $T \subset G$ on 
$$
\sum_p (-m)^p \chi(X,\Omega^p) 
$$
and conclude that every term in this sum is a trivial $G$-module. 
\end{Exercise}

Of course, Hodge theory gives the triviality of $G$-action on 
each 
$$
H^q(X,\Omega^p)\subset H^{p+q}(X,\C) 
$$
for a compact K\"ahler $X$ and a connected group $G$. 

\subsubsection{}

When the above approach works it also means that the localization 
formula may be simplified by sending the equivariant 
variable to a suitable infinity of the torus. 

\begin{Exercise}\label{exer_BB} 
In Exercise \ref{Ex_mOmega}, pick a generic 
$1$-parameter subgroup 
$$
z\in \Ct \to T
$$
and compute the asymptotics of your localization formula 
as $z\to 0$.
\end{Exercise}

It is instructive to compare the result of Exercise \ref{exer_BB} 
with the Bia\l{}ynicki-Birula decomposition, which goes
as follows. Assume $X\subset \bP(\C^N)$ is smooth and 
invariant under the action of a 1-parameter subgroup 
$\Ct\to GL(N)$. 
Let 
$$
X^{\Ct} = \bigsqcup_i F_i
$$
be the decomposition of the fixed locus into connected components. 
It induces a decomposition of $X$ 
\begin{equation}
  \label{BB}
X= \bigsqcup X_i\,, \quad 
X_i = \left\{ x \left| \lim_{z\to 0} z\cdot x \in F_i \right. \right\}
\end{equation}
into locally closed 
sets. The key property of this decomposition is that the 
natural map 
$$
X_i \xrightarrow{\,\, \lim \,\,} F_i
$$
is a fibration by affine spaces of dimension 
$$
d_i = \rk \big(N_{X/F_i} \big)_{+}
$$
where plus denotes the subbundle spanned by vectors of positive 
$z$-weight, see for example \cite{Bros} for a recent discussion. 
As also explained there, the 
decomposition \eqref{BB} is, in fact, motivic, and in particular 
the Hodge structure of $X$ is that sum of those for $F_i$ 
shifted by $(d_i,d_i)$. 

The same decomposition of $X$ can be obtained from Morse 
theory applied to the Hamiltonian $H$ that generated the 
action of $U(1) \subset \Ct$ on the (real) symplectic manifold $X$. 
Concretely, if $z$ acts by 
$$
[x_1:x_2:\dots:x_n] \xrightarrow{\,\, z \,\,} 
[z^{m_1}\, x_1: z^{m_2}\, x_2:\dots:z^{m_n}\, x_n] 
$$
then
$$
H(x) = \sum m_i |x_i|^2 \left/ \sum |x_i|^2 \right. \,.
$$

\subsubsection{}

In certain instances, the same argument gives more.

\begin{Exercise}
Let $X$ be proper nonsingular variety with a nontrivial action of $T\cong
\Ct$.  Assume that a fractional power $\cK^p$ for $0<p<1$ of 
the canonical bundle $\cK_X$ exists in $\Pic(X)$.  Replacing $T$ by 
a finite cover, we can make it act on $\cK^p$. Show that 
$$
\chi(X,\cK^{p}) = 0 \,. 
$$
What does this say about projective spaces ? Concretely, 
which are the bundles $\cK^p$, $0<p<1$, 
 for $X=\bP^n$ and what do we 
know about their cohomology ? 
\end{Exercise}

\section{The Hilbert scheme of points of 3-folds} 
\label{S_Hilb3} 

\subsection{Our very first DT moduli space}

\subsubsection{}

For a moment, let $X$ be a nonsingular quasiprojective 3-fold; very
soon we will specialize the discussion to the case $X=\C^3$. 
Our interest is in the enumerative geometry of subschemes in $X$, 
and usually we will want these subscheme projective and 
$1$-dimensional. 

A subscheme $Z\subset X$ is 
defined by a sheaf of ideals $\cI_Z \subset \cO_X$ in the sheaf
$\cO_X$ of 
functions on $X$ and, by 
construction, there is an exact sequence 
\begin{equation}
  \label{eq:2}
  0 \to \cI_Z \to \cO_X \to \cO_Z \to 0 
\end{equation}
of coherent sheaves on $X$. 
Either the injection $\cI_Z \hookrightarrow \cO_X$, or the surjection 
$\cO_X \twoheadrightarrow \cO_Z$ determines
$Z$ and can be used to parametrize subschemes of $X$. 
The result, known as the Hilbert scheme, 
 is a countable union of quasiprojective algebraic 
varieties, one for each possible topological $K$-theory class of
$O_Z$. The construction of the Hilbert scheme goes back to 
A.~Grothendieck and is explained, for example, in \cite{FDA,Koll}. 

In particular, for $1$-dimensional $Z$, the class $[O_Z]$ is specified by 
$$
\deg Z = - c_2(\cO_Z) \in H_2(X,\Z)_\textup{effective} 
$$
and by the Euler characteristic $\chi(\cO_Z)\in\Z$. In this section, we
consider the case $\deg Z=0$, that is, the case of the Hilbert scheme
of points.

\subsubsection{}

If $X$ is affine then the Hilbert scheme of points parametrizes
modules $M$ over the ring $\cO_X$ such that
$$
\dim_\C M = n \,, \quad n=\chi(\cO_Z)\,,
$$
together with a surjection from a free module. Such Hilbert schemes 
$\Hilb(R,n)$ may, 
in fact, be defined for an arbitrary finitely-generated algebra
$$
R = \C \lan x_1, \dots, x_k \ran \big/ \textup{relations} 
$$
and consists of $k$-tuples of $n\times n$ matrices 
\begin{equation}
\bX_1\,\dots, \bX_k \in \End(\C^n)  \label{bXk} 
\end{equation}
satisfying the 
relations of $R$, together with a cyclic vector $v \in \C^n$, 
all modulo the action of $GL(n)$ by conjugation. Here, a vector 
is called cyclic if it generates $\C^n$ under the action of $\bX_i$'s. 
Clearly, a surjection $R_1 \twoheadrightarrow R_2$ leads to 
inclusion $\Hilb(R_2,n) \hookrightarrow \Hilb(R_1,n)$ and, in
particular, 
$$
\Hilb(R,n) \subset \Hilb(\Free_k,n) 
$$
if $R$ is generated by $k$ elements. 

\begin{Exercise}
Prove that $\Hilb(\Free_k,n)$ is a smooth algebraic variety of
dimension $(k-1)n^2+n$. Show that $\Hilb(\Free_1,n)$ is isomorphic 
to $S^n \C \cong \C^n$ by the map that takes $x_1$ to its
eigenvalues. 
\end{Exercise}

\noindent 
By contrast, $\Hilb(\C^3,n)$ is a very singular variety of unknown 
dimension. 

\begin{Exercise}
Let $\fm\subset \cO_{\C^d}$ be the ideal of the origin. Following 
Iarrobino, observe, that any linear linear subspace $I$ such that 
$\fm^r \supset I \supset \fm^{r+1}$ for some $r$ is an ideal in
$\cO_{\C^d}$. Conclude that the dimension of $\Hilb(\C^d,n)$ 
grows at least like a constant times $n^{2-2/d}$ as $n\to\infty$. 
This is, of course, consistent with 
$$
\dim \Hilb(\C^1,n) = n \,, \quad \dim \Hilb(\C^2,n) = 2n
$$
but shows that $\Hilb(\C^d,n)$ is not the closure of the locus of 
$n$ distinct points for $d\ge 3$ and large enough $n$. 
\end{Exercise}

\subsubsection{}
Consider the embedding 
$$
\Hilb(\C^3,n) \subset  \Hilb(\Free_3,n) 
$$
as the locus of matrices that commute, that is 
$\bX_i \bX_j = \bX_j \bX_i$.  For 3 matrices, and only for 3 matrices, 
these relations can be written as equations for a critical point: 
$$
d\phi = 0 \,, \quad \phi(\bX) = \tr \left( \bX_1 \bX_2 \bX_3 - 
\bX_1 \bX_3 \bX_2 \right) \,.
$$
Note that $\phi$ is a well-defined function on $\Hilb(\Free_3,n)$ 
which transforms in the 1-dimensional representation $\bk^{-1}$, where
$$
\bk = \Lambda^3 \C^3 = {\det}_{GL(3)} \,, 
$$
under the natural action of $GL(3)$ on $\Hilb(\Free_3,n)$. Here we
have to remind ourselves that the action of a group $G$ on functions 
is \emph{dual} to its action on coordinates. 

This means that our moduli space 
$\mM = \Hilb(\C^3,n)$ is cut out inside an ambient smooth 
space $\mMh= \Hilb(\Free_3,n)$ by a section 
$$
\cO_{\mMh} \, \xrightarrow{\, d\phi \otimes \bk \, }
\bk \otimes T^*_\mMh \,, 
$$
of a vector bundle on $\mMh$. The twist by $\bk$ is necessary to 
make this section $GL(3)$-equivariant. 

This illustrates two important points about moduli problems 
in general, and moduli of coherent sheaves 
on nonsingular threefolds in particular. 
First, locally near any point, deformation theory describes 
many moduli spaces in a similar way: 
\begin{equation}
\mM = s^{-1}(0) \subset \mMh, \quad  s \in \Gamma(\mMh,\cE)
\,,\label{cMs0}
\end{equation}
for a certain \emph{obstruction} bundle $\cE$. Second, for 
coherent sheaves on 3-folds, there is a certain kinship between the 
obstruction bundle $\cE$ and the cotangent bundle of $\mMh$, 
stemming from Serre duality beween the groups $\Ext^1$, which 
control deformations, and the groups $\Ext^2$, which control
obstructions.  The kinship is only approximate, unless the 
canonical class $\cK_X$ is \emph{equivariantly} trivial, which is not 
the case even for $X=\C^3$ and leads to the twist by $\bk$ above. 

\subsection{$\cO^\vir$ and $\tO^\vir$} 

\subsubsection{}

The description \eqref{cMs0} means that $\cO_\mM$ is the 
0th cohomology of the Koszul complex 
\begin{equation}
0 \to \Lambda^{\rk \cE} \cE^\vee \xrightarrow{\,d\,} \dots 
\to \Lambda^2 \cE^\vee \xrightarrow{\,d\,} \cE^\vee \xrightarrow{\,d\,} 
\cO_\mMh \to 0 
\label{KoszulE} 
\end{equation}
in which $\cO_\mMh$ is placed in cohomological degree $0$ and the 
differential is the contraction with the section $s$ of $\cE$.  

The Koszul complex is an example of a sheaf of \emph{differential graded 
algebras}, 
which by definition is a sheaf $\cA^\bullet$ of graded algebras with the 
differential
\begin{equation}
\dots \xrightarrow{\,d\,} \cA^{-2} \xrightarrow{\,d\,} 
\cA^{-1} \xrightarrow{\,d\,} \cA^{0} \to 0 \label{DGA} 
\end{equation}
satisfying the Leibniz rule 
$$
d(a\cdot b)= da \cdot b + (-1)^{\deg a} a \cdot db \,. 
$$
The notion of a DGA has become one of the cornerstone notions in
deformation theory, see for example how it used in 
the papers \cite{BCFHR,CF1,CF2,CF3} for a very incomplete set of
references.

In particular, the structure sheaves $\cO_\mM$ of great many 
moduli spaces are described as $H^0(\cA^\bullet)$ of a certain 
natural DGAs. 

\subsubsection{}

Central to $K$-theoretic enumerative geometry is the concept of 
the \emph{virtual structure sheaf} denoted $\cO_{\mM}^\textup{vir}$. 
While $\cO_\mM$ is the 0th cohomology of a complex \eqref{DGA}, 
 the virtual structure sheaf is its Euler characteristic
 \begin{align}
  \cO_{\mM}^\textup{vir} &= \sum_i (-1)^i \cA^{i}  \notag\\
    & = \sum_i (-1)^i H^i(\cA^\bullet) \label{defOvir} \,,
 \end{align}
see \cite{Krat,CF3}. By
Leibniz rule, each $H^i(\cA^\bullet)$ is acted upon by $\cA^0$ and 
annihilated by 
$d\cA^{-1}$, hence defines a quasicoherent sheaf on 
$$
\mM = \Spec \cA^{0}\big/ d\cA^{-1} \,. 
$$
If cohomology groups are coherent $\cA^0$-modules 
and vanish for $i\ll 0$ then the second line in \eqref{defOvir} 
gives a well-defined element of $K(\mM)$, or of $K_G(\mM)$ 
if all constructions are equivariant with respect to a group $G$. 

The definition of $\cO_{\mM}^\textup{vir}$ is, in several 
respects, simpler than the definition \cite{BF} of the virtual fundamental 
cycle in cohomology. The agreement between the two is explained in 
Section 3 of \cite{CF3}. 

\subsubsection{}

There are many reasons to prefer $\cO_{\mM}^\textup{vir}$ over 
$\cO_\mM$, and one of them is the invariance of virtual 
counts under deformations. 

For instance, in a family $X_t$ of threefolds, 
special fibers may have many more curves
than a generic fiber, and even the dimensions of the Hilbert 
scheme of curves in $X_t$ may be jumping wildly. 
This is reflected by the fact that in a family of complexes
$\cA^\bullet_t$ 
each individual cohomology group is only semicontinuous and 
can jump up for special values of $t$.  However, in a flat family  
of complexes the (equivariant) Euler characteristic is constant, and 
equivariant virtual counts are
invariants of equivariant deformations.

\subsubsection{}

Also, not the actual but rather the virtual counts are 
usually of interest in mathematical physics. 

A supersymmetric physical theory starts out as a 
Hilbert space and an operator of the form \eqref{HDH}, where 
at the beginning the Hilbert space $\cH$ is something 
enormous, as it describes the fluctuations of many fields extended
over many spatial dimensions. However, all those infinitely many 
degrees of freedom that correspond to nonzero eigenvalues of
the operator $\Dir^2$ pair off and make no contribution to 
supertraces \eqref{strA}. What remains, in cases of interest to
us, may be identified with a direct sum (over various topological 
data) of complexes of finite-dimensional $C^\infty$ vector bundles over 
finite-dimensional K\"ahler manifolds. These complexes combine the 
features of
\begin{itemize}
\item[(a)] a Koszul complex for a section $s$ of a certain vector
  bundle, 
\item[(b)] a Lie algebra, or BSRT cohomology complex when a certain
  symmetry needs to be quotiented out, and 
\item[(c)] a Dolbeault cohomology, or more precisely a related 
Dirac cohomology complex, which turns the supertraces into 
\emph{holomorphic} Euler characteristics of K-theory classes 
defined by (a) and (b). 
\end{itemize}
\noindent
If $M$ is a K\"ahler manifold, then spinor bundles of $M$ are 
the bundles 
$$
\mathscr{S}_\pm = \cK_M^{1/2} \otimes 
\bigoplus_{\textup{$n$ even/odd}} \Omega^{0,n}_M 
$$
with the Dirac operator $\Dir = \dbar+\dbar^*$. Here $\cK_M$ 
is the canonical line bundle of $M$, which needs to be a square
in order for $M$ to be spin.

 In item (c) above, the difference 
between Dolbeault and Dirac cohomology is precisely the 
extra factor of $\cK_M^{1/2}$. While this detail may look 
insignificant compared to many other layers of complexity 
in the problem, it will prove to be of fundamental importance 
in what follows and will make 
many computations work. The basic reason for this was 
already discussed in Section \ref{s_princ}, and will be revisited 
shortly: the twist by $\cK_M^{1/2}$ 
makes formulas more self-dual and, thereby, more rigid
than otherwise. 

\subsubsection{}
Back in the Hilbert scheme context, the distinction between 
Dolbeault and Dirac is the distinction between 
$\cO_{\mM}^\textup{vir}$ and 
$$
\tO_{\mM}^\textup{vir}  = (-1)^n \, \cK^{1/2}_\vir \otimes
\cO_{\mM}^\textup{vir}
$$
where the sign will be explained below and 
the virtual canonical bundle $\cK^{1/2}_\vir $ is constructed as
the dual of the determinant of the virtual tangent bundle 
\begin{align}
T^\vir_\mM &= \Def - \Obs \label{Tvir} \\
&= \left(T_\mMh - \bk \otimes T^*_\mMh\right)\Big|_\mM 
\label{Tvir2} \,. 
\end{align}
{}From \eqref{Tvir2} we conclude 
$$
\cK^{1/2}_\vir = \bk^{\frac{\dim}2}  \otimes \cK_{\mMh} 
\Big|_\mM 
$$
where $\dim=\dim \mMh$. This illustrates a general result of \cite{NO} 
that for DT moduli spaces the virtual canonical line bundle $\cK_\vir$
is a square in the equivariant Picard group up to a character of
$\Aut(X)$ and certain additional twists which will be discussed 
presently. 

{} From the canonical isomorphism 
$$
\Lambda^k T_\mMh \otimes \cK_{\mMh} \, \cong \, \Omega^{\dim -k }_\mMh 
$$
and the congruence 
$$
\dim \mMh \equiv n \mod 2 
$$
we conclude that 
\begin{equation}
\tO_{\mM}^\textup{vir} = 
\bk^{-\frac{\dim}2} \sum_i (-\bk)^i \, \Omega^i_{\mMh} \,. 
\label{OOm} 
\end{equation}
We call $\tO_{\mM}^\textup{vir}$ the \emph{modified} or the
\emph{symmetrized} 
virtual structure sheaf. The hat which this K-class is wearing 
 should remind the reader of 
the $\widehat{A}$-genus and hence of the Dirac operator. 

%
% \left[
% \bk^{\frac{\dim}2} \cO_{\mMh} \to 
% \bk^{\frac{\dim}2-1}  T^*_{\mMh} \to \dots \to 
% \bk^{-\frac{\dim}2} \cK_{\mMh} \right]\,,
%

\subsubsection{}

Formula \eqref{OOm} merits a few remarks.
  As discussed before, Serre duality between $\Ext^1$ and 
$\Ext^2$ groups on threefolds yields a certain kinship between the
deformations and the \emph{dual} of the obstructions. This makes 
the terms of the complex \eqref{KoszulE} and hence 
of the virtual structure sheaf 
$\cO^\vir_\mM$ look like polyvector fields $\Lambda^i T_\mMh$ 
on the ambient space
$\mMh$. 

The twist by $\cK_\mMh \approx \cK^{1/2}_\vir$ turns polyvector fields into
differential forms, and differential forms are everybody's favorite
sheaves for cohomological computations. For example, if 
$\mMh$ is a compact K\"ahler manifold then $H^q(\Omega^q_\mMh)$ is 
a piece of the Hodge decomposition of $\Hd(\mMh,\C)$, and 
in particular, rigid for any connected group of automorphisms. But 
even when these cohomology groups are not rigid, or when 
the terms of $\tO^\vir$ are not exactly differential forms, they still
prove to be much more manageable than $\cO^\vir$. 

\subsubsection{}\label{s_taut} 

The general definition of $\tO^\vir$ from \cite{NO} involves, 
in addition to $\cO^\vir$ and $\cK_\vir$, a certain
\emph{tautological} 
class on the DT moduli spaces. 

The vector space $\C^n$ in which the matrices \eqref{bXk} 
operate, descends to a rank $n$ vector bundle over 
the quotient $\Hilb(\Free,n)$.
Its fiber over $Z\in \Hilb(\C^3,n)$ is naturally
identified with global sections of $\cO_Z$, that is, the pushforward
of the \emph{universal sheaf} $\cO_\fZ$ along $X$, where 
$$
\fZ \subset \Hilb(X,n) \times X
$$
is the universal subscheme. 

Analogous universal sheaves exists for Hilbert schemes of 
subschemes of any dimension. In particular, for the 
Hilbert scheme of curves, and other DT moduli spaces $\mM$, there 
exists a universal 1-dimensional sheaf $\fF$ on $\mM\times X$
such that 
$$
\fF\big|_{m \times X} = \cF
$$
where $m=(\cF,\dots)\in \mM$ is the moduli point representing a $1$-dimensional sheaf 
$\cF$ on $X$ with possibly some extra data. 

By construction, the sheaf $\fF$ is \emph{flat} over $\mM$, therefore 
\emph{perfect}, that is, has a finite resolution by vector bundles on 
$\mM \times X$. This is a nontrivial property, since $\mM$ is highly 
singular and for singular schemes $K^\circ \subsetneq K$, where 
$K^\circ$ is the subgroup generated by vector bundles. From elements
of $K^\circ$ one can form tensor polynomials, for example 
\begin{equation}
P = S^2 \fF \otimes \gamma_1 + \Lambda^3 \fF \otimes \gamma_2
 \in K^\circ(\mM \times
X) \label{ex_tens_poly} \,,
\end{equation}
for any $\gamma_i \in K(X)=K^\circ(X)$. The support of any tensor 
polynomial like \eqref{ex_tens_poly} is proper over $\mM$ and 
therefore
$$
\pi_{\mM,*} P \in K^\circ(\mM) \,,
$$
where $\pi_\mM$ is the projection along $X$. This is because 
a proper pushforward of a flat sheaf is perfect, see 
for example Theorem 8.3.8 in \cite{FDA}.
  It makes sense to apply 
further tensor operations to $\pi_{\mM,*} P$,  or to take it 
determinant. 

In this way one can manufacture a large supply of 
 classes in $K^\circ(\mM)$ 
which are, in their totality, called \emph{tautological}. 
One should think of them as K-theoretic 
version of imposing various geometric conditions on curves, like 
meeting a cycle in $X$ specified by $\gamma\in K(X)$.

\subsubsection{}
An example of the tautological class is the determinant term in the 
following definition 
\begin{equation}
\tO^\vir = \textup{prefactor} \,\, \cO^\vir \otimes 
\left(\cK_\vir \otimes \det \pi_{\mM,*} (\fF\otimes (\cL_4 -
  \cL_5))\right)^{1/2}\,,
\label{deftOvir} 
\end{equation}
where $\cL_4 \in \Pic(X)$ is an arbitrary line bundle and $\cL_5$ is a 
line bundle determined from the equation 
$$
\cL_4 \otimes \cL_5 = \cK_X \,.
$$
These are the same $\cL_4$ and $\cL_5$ as in Section \ref{sL4L5}. 
The prefactor in \eqref{deftOvir} contains the $z$-dependence 
\begin{equation}
\textup{prefactor} = (-1)^{(\cL_4,\beta)+n}
z^{n-(\cK_X,\beta)/2}\,, \label{prefactor} 
\end{equation}
where 
$$
\beta = \deg \cF \in H_2(X,\Z) \,, \quad n = \chi(\cF) \in \Z
$$
are locally constant functions of a 1-dimensional sheaf $\cF$ on 
$X$. 

The existence of square root in \eqref{deftOvir} is shown in
\cite{NO}. With this definition, the general problem of computing 
the K-theoretic DT invariants may be phrased as 
\begin{equation}
\chi\left(\mM, \tO^\vir \otimes \textup{tautological} \right) =
\textup{\large ?}\,, \label{?}
\end{equation}
where $\mM$ is one of the many possible DT moduli spaces. 
In a relative situation, one can put further insertions in \eqref{?}. 

\subsection{Nekrasov's formula}

\subsubsection{}
Let $X$ be a nonsingular quasiprojective threefold and 
consider its Hilbert scheme of points 
$$
\mM = \bigsqcup_{n\ge 0} \Hilb(X,n) \,. 
$$
In the prefactor \eqref{prefactor} we have $\beta=0$ and therefore 
\begin{equation}
\textup{prefactor} = (-z)^n  \,. \label{pref_points}
\end{equation}
In \cite{Zth}, Nekrasov conjectured a formula for 
\begin{equation}
\bZ_{X,\points} = \chi\left(\mM, \tO^\vir\right) \,. 
\label{defcZ}
\end{equation}
Because of the prefactor, this is well-defined as 
a formal power series in $z$, as long as $\chi\!\left(\Hilb(n),
  \tO^\vir\right)$ is well-defined for each $n$. For that we need 
to assume that either $X$ is proper or, more generally, that there exists 
$g\in \Aut(X)_0$ such that it fixed point set $X^g$ is proper (if $X$
is already proper we can always take $g=1$). Then 
$$
\bZ_{X,\points}  \in 
\begin{cases}
  K_{\Aut(X)}(\pt) [[z]]\,, & g=1\\
K_{\Aut(X)}(\pt)_\textup{loc} [[z]]\,, & \textup{otherwise} \,. 
\end{cases}
$$
To be precise, \cite{Zth} considers the 
case $X=\C^3$, but the generalization to arbitrary $X$ is immediate.

\subsubsection{}

Nekrasov's formula computes $\bZ_{X,\points}$ in the 
form 
\begin{equation}
  \bZ_{X,\points} = \bSd \chi(X,\star) \,, \quad \star \in K_\bT(X)[[z]]
\end{equation}
where $\bSd$ is the symmetric algebra from Section \ref{s_bSd}. 

Here, and this is 
very important, the boxcounting variable $z$ is viewed as a part of the 
torus $\bT$, that is, it is also raised to the power $n$ in the
formula \eqref{SV}. This is very natural from the 5-dimensional 
perspective, since the $z$ really acts on the 5-fold \eqref{ZLL}, with 
the fixed point set $X$, as discussed in Section \ref{sL4L5}. 

Now we are ready to state the following result, 
conjectured by Nekrasov in \cite{Zth}

\begin{Theorem}\label{t_Hilb3}
We have 
\begin{equation}
\bZ_{X,\points} = \bSd \chi\left(X, \frac{z\cL_4 \,(T_X + \cK_X - T_X^\vee
    - \cK_X^{-1})}{(1-z \cL_4)(1-z \cL^{-1}_5)} \right)\,. 
\label{f_Hilb3} 
\end{equation}
\end{Theorem}

\subsubsection{}
In \cite{Zth}, this conjecture appeared together with an important 
physical interpretation, as the supertrace of $\Aut(Z,\Omega^5)$-action on the 
fields of \emph{M-theory} on \eqref{ZLL}. M-theory is a supergravity 
theory in 10+1 real spacetime dimensions.  Its fields are: 
\begin{itemize}
\item[---] the metric, also known as the graviton, 
\item[---] its fermionic 
superpartner, gravitino, which is a field of spin $3/2$, and
\item[---] one more 
bosonic field, a $3$-form
analogous to a connection, or a gauge boson, in gauge 
theories such as electromagnetism.
\end{itemize}
These are considered up to a gauge 
equivalence that includes: diffeomorphisms, changing the 3-form by 
an exact form, and changing gravitino by a derivative of a spinor. 

In addition to these fields, 
$M$-theory has extended objects, namely:
\begin{itemize}
\item[---] membranes, which have a
3-dimensional worldvolume and hence are naturally charged under 
the $3$-form, and also 
\item[---]magnetically dual M5-branes.
\end{itemize}
While membranes naturally 
appear in connection with DT invariants of positive degree
\cite{NO}, see for example Section \ref{s_MemSn} below, 
possible algebro-geometric interpretations of
M5-branes are still very much in the initial exploration stage. 

In Hamiltonian description, the field sector of the Hilbert space of 
M-theory is, formally, the $L^2$ space of functions on 
a very infinite-dimensional configuration supermanifold, namely 
\begin{equation}
\cH \eqf
L^2 \left(\textup{bosons} \oplus \frac12 \, \textup{fermions on $Z$}
\Big/ \textup{gauge}\right) \,, 
\label{cHL2} 
\end{equation}
where 
\begin{itemize}
\item[---] $Z$ is a fixed time slice of the 11-dimensional 
spacetime, 
\item[---] $\frac12$ of the fermions denotes a choice of polarization, that is, 
a splitting of fermionic operators into operators of creation and 
annihilation, 
\item[---] the $\eqf$ sign denotes an equality which is formal,
  ignores key analytic and dynamical questions, but may be suitable for 
equivariant K-theory which is often insensitive to 
finer details.  
\end{itemize}
As a time slice, one is allowed to take a 
Calabi-Yau 5-fold, for example the one in \eqref{ZLL}. 
Automorphisms of $Z$ preserving the $5$-form $\Omega^5$ 
are symmetries
of the theory and hence act on its Hilbert space. 

Since the configuration space is a linear representation of 
$\Aut(Z,\Omega^5)$, we have 
$$
\cH \eqf \bSd \textup{Configuration space} 
$$
in $K$-theory of $\Aut(Z,\Omega^5)$. The character of 
the latter may be computed, see \cite{Zth}, the exposition of 
the results of \cite{Zth} in  
Section 2.4 of \cite{NO}, and below, with the result that 
\begin{equation}
\cH \eqf \Ld \chi(Z, T_Z) \,,
\label{conf_space_trace} 
\end{equation}
or its dual 
$$
\overline{ \Ld \chi(Z, T_Z)} = \bSd \chi(Z, T^*_Z) \,, 
$$
depending on the choice of the polarization 
in \eqref{cHL2}. 

Note that 
\begin{equation*}
\chi(Z, T^*_Z-T_Z)  = \chi\left(X, \cK_X - \cO_X + 
\frac{z\cL_4 \,(T_X + \cK_X - T_X^*
    - \cK_X^{-1})}{(1-z \cL_4)(1-z \cL^{-1}_5)} \right) \,. 
\end{equation*}
This implies 
\begin{equation}
  \bSd \chi\left(X, \cK_X - \cO_X\right) 
\otimes \bZ_{X,\points} \eqf \cH \otimes \overline{\cH} \,,
\label{ZHH} 
\end{equation}
which is a formula with at least two issues, one minor and the other 
more interesting. The minor issue is the prefactor in 
the LHS, which is ill-defined as written. We will see below how 
this prefactor appears in DT computations and why in the 
natural regularization it is simply removed. 

The interesting issue in \eqref{ZHH} is the doubling the 
contribution of $\cH$. As already pointed out by Nekrasov, it should
 be reexamined with a more careful analysis of the
physical Hilbert 
space of M-theory.  

\subsubsection{}

The following exercises will guide the reader through the proof of
\eqref{conf_space_trace}.  Since this computation is about
finite-dimensional representations of reductive groups, we may 
complexify problem and consider 
$$
V=Z \otimes_\R \C \cong \C^{10} \,.
$$
This is a vector space with a nondegenerate 
quadratic form. We may assume the form is given by 
$$
\| (x_1,\dots,x_{10}) \|^2 = \sum_{i=1}^k x_i x_{11-i} 
$$
in which case 
$$
B = SO(V) \cap \{ \textup{upper triangular matrices} \}
$$
is a Borel subgroup with a maximal torus 
$$
T=\left\{\diag(t_1,\dots,t_5,t_5^{-1},\dots,t_1^{-1})\right\} \cong 
\left(\Ct\right)^5\,. 
$$
By definition, the spinor representations 
$S_\pm$ of $SO(V)$ have the highest weight 
$$
\left(\frac12,\frac12,\frac12,\frac12,\pm \frac12\right) \,. 
$$

\begin{Exercise}
Show that the character of $S_{\pm}$ is given by 
$$
\tr_{S_+} t - \tr_{S_-} t = (t_1 \cdots t_5)^{1/2} \prod_{i=1}^5
(1-t_i^{-1}) \, .
$$
\end{Exercise} 

\begin{Exercise} Prove that the dual 
of one spinor representation is the other. 
\end{Exercise} 

\begin{Exercise}
Show there exists a nonzero $SO(V)$ intertwiner 
$$
V \otimes S_{\pm} \to S_{\mp} \,. 
$$
\end{Exercise} 

\begin{Exercise}
Using e.g.\ Weyl dimension formula, show 
$$
V \otimes S_{\pm} = S_{\mp} \oplus RS_{\pm}\,, 
$$
where $RS_{\pm}$ is an irreducible $SO(V)$ module with highest weight 
$$
\left(\frac32,\frac12,\frac12,\frac12,\pm \frac12\right) \,. 
$$
Fields transforming in this representation of the orthogonal 
group are called \emph{Rarita-Schwinger} fields. 
\end{Exercise} 

Viewed as $SO(V)$ module, the configuration superspace 
in \eqref{conf_space_trace} is the space of functions on $V$ with 
values in the following virtual representation of $SO(V)$ 
\begin{alignat}{2}
  \mathbf{M}_\pm=&S^2 V - 1&\qquad&\textup{traceless
    metric}  \label{fieldsM} \\
&-V &\qquad&\textup{modulo diffeomorphism} \notag  \\
&\Lambda^3V &\qquad&\textup{a $3$-form} \notag  \\
&-\Lambda^2V+V-1 &\qquad&\textup{modulo exact forms} \notag  \\
&-RS_{\pm}   &\qquad&\textup{a Rarita-Schwinger field}  \notag  \\
&+S_\pm  &\qquad&\textup{modulo derivative of a spinor}  \notag  \,. 
\end{alignat}
Note the change of sign in the last two lines in \eqref{fieldsM}. This
is because fields of half-integer spin are fermions, and thus count
with the minus sign in the supertrace. 

Also note that the dual of $\mathbf{M}_\pm$ is
$\mathbf{M}_\mp$.  This is the place where the 
choice of polarization 
in \eqref{cHL2} makes a difference. 

\subsubsection{}
The span $W \subset V$ of the first 5 coordinate lines is isotropic,
and we can write 
$$
V = W \oplus W^* \cong T^{1,0}Z \oplus T^{0,1}Z 
$$
where 
$$
TZ \otimes_\R \C = T^{1,0}Z \oplus T^{0,1}Z
$$
is the decomposition of tangent vectors by their Hodge type. 
As a result, we have an embedding 
$$
GL(W) \subset SO(V) \,.
$$
The subgroup $SL(W) \subset GL(W)$ corresponds to those 
operators that act complex-linearly on $Z\cong \C^5$ and preserve 
a nondegenerate $5$-form. 

\begin{Exercise}
Prove that 
 \begin{align}
\mathbf{M}_+ &= - W \otimes \Ld W \,, \label{MSLW} \\
\mathbf{M}_- &=  W^* \otimes \Ld W \,, \notag 
 \end{align}
as $SL(W)$ modules. By duality, it is enough to prove either of
those. 
\end{Exercise}

Since $W \cong T^{1,0} Z \cong \Omega^{0,1} Z$ as $SL(W)$-modules, we 
have 
\begin{align*}
\textup{Configuration (super)space} &= \chi(Z,C^{\infty}(\mathbf{M}_+)) \\
&= - \chi(\textup{Dolbeault}(TZ)) \\
&= -\chi(Z,TZ)
\end{align*}
where $C^{\infty}(\mathbf{M}_+)$ is the sheaf of smooth functions with 
values in $\mathbf{M}_+$, and in the second line we have the 
Dolbeault complex of the holomorphic tangent bundle of $Z$. 
 This completes the proof of 
\eqref{conf_space_trace}.

\subsubsection{}

Below we will see how \eqref{f_Hilb3} reproduces an earlier result in 
cohomology proven in
\cite{mnop2} for $X$ toric varieties and in \cite{LevPand} for general 
3-folds. The cohomological version of \eqref{defcZ} is 
\begin{equation}
\bZ_{X,\points,\textup{coh}} = \sum_{n\ge 0} 
(-z)^n \int_{[\Hilb(X,n)]_\vir} 1 \,, 
\label{defcZcoh}
\end{equation}
where $[\Hilb(X,n)]_\vir$ is the virtual fundamental cycle \cite{BF}. 
It may be defined as the cycle corresponding to the K-theory class
$\cO^\vir$ which, 
as \cite{CF3} prove, has the expected dimension of support.

In general, for DT moduli spaces, the expected dimension is 
$$
\vir \, \dim = - (\deg \cF, \cK_X) \,,
$$
where $\cF$ is the $1$-dimensional sheaf on $X$. For Hilbert schemes
of points this vanishes and $[\Hilb(X,n)]_\vir$ is an equivariant $0$-cycle. 
For $\Hilb(\C^3)$ this cycle is the Euler class of the obstruction 
bundle. 

The following result was conjectured in \cite{mnop2} and proven there
for toric $3$-folds. The general algebraic cobordism approach of 
Levine and Pandharipande reduces the case of a general $3$-fold 
to the special case of toric varieties. 

\begin{Theorem}[\cite{mnop2,LevPand}] \label{t_Z_coh} 
We have 
$$
\bZ_{X,\points,\textup{coh}}  = M(z)^{\int_X c_3(T_X \otimes
  \cK_X)}\,, 
$$
where 
$$
M(z) = \bSd \frac{z}{(1-z)^2} = \prod_{n>0} (1-z^n)^{-n} \,,
$$
is McMahon's generating function for 3-dimensional partitions. 
\end{Theorem}

\noindent
The appearance of 3-dimensional partitions here is very natural --- 
they index the torus fixed points in $\Hilb(\C^3,n)$. These appear 
naturally in equivariant virtual localization, which is the subject to 
which we turn next.

\subsection{Tangent bundle and localization}

\subsubsection{}
Consider the action of the maximal torus 
$$
T=\left\{
  \begin{pmatrix}
    t_1 \\
& t_2 \\
& & \ddots \\
& & & t_d
  \end{pmatrix}
\right\} 
\subset GL(d) 
$$
on $\Hilb(\C^d,\textup{points})$, that is, on ideals of 
finite codimension in the ring $\C[x_1,\dots,x_d]$. 

\begin{Exercise}
Prove that the fixed points set $\Hilb(\C^d,\textup{points})^T$ is 
$0$-dimensional and formed by \emph{monomial ideals}, that is, 
ideals generated by monomials in the variables $x_i$. 
In particular, the points of $\Hilb(\C^d,n)^T$ are in natural 
bijection with $d$-dimensional partitions $\pi$ of the number $n$.  
\end{Exercise}

\subsubsection{}
For $d=1,2$, Hilbert schemes are smooth and our next goal is to 
compute the character of the torus action on $T_\pi \Hilb$. This will 
be later generalized to the computation of the torus character of 
the \emph{virtual} tangent space for $d=3$. 

By construction or by the functorial definition of the Hilbert scheme,
its Zariski tangent space at any subscheme $Z\subset X$ 
is given by 
$$
T_Z \Hilb = \Hom(\cI_Z,\cO_Z) 
$$
where $\cI_Z$ is the sheaf of ideals of $Z$ and $\cO_Z =\cO_X/\cI_Z$ 
is its structure sheaf. Indeed, the functorial description of Hilbert 
scheme says that 
$$
\textup{Maps}(B,\Hilb(X)) = \left\{
  \begin{matrix}
    \textup{subschemes of $B\times X$} \\
\textup{flat and proper over $B$} 
  \end{matrix}
\right\} 
$$
for any scheme $B$. In particular, a map from 
$$
B = \Spec \C[\varepsilon]/\varepsilon^2
$$
is a point of $\Hilb(X)$ together with a tangent vector, and this
leads to the formula above. 

\begin{Exercise}\label{ex_tan_hilb} 
   Check this. 
\end{Exercise}

\begin{Exercise}\label{ex_Hilb_curve} 
Let $X$ be a smooth curve and $Z\subset X$ a $0$-dimensional subscheme. 
Prove that 
$$
T_Z \Hilb = H^0( T^*X \otimes \cO_Z)^* \,. 
$$
In particular, for $X=\C^1$ and 
the torus-fixed ideal $I=(x_1^n)$ we have 
$$
T_{(x_1^n)} = t_1 + \dots + t_1^n \,, 
$$
in agreement with global coordinates on
$\Hilb(\C^1,n)\cong \C^n$ 
given by 
$$
I = (f), \quad f = x^n + a_1 \, x^{n-1} + \dots + a_n \,. 
$$
\end{Exercise}

\subsubsection{}

Now let $X$ be a nonsingular surface and $Z\subset X$ a
$0$-dimensional subscheme. By construction, we have a 
short exact sequence 
$$
0 \to \cI_Z \to \cO_X \to \cO_Z \to 0 
$$
which we can apply in the first argument to get the following 
long exact sequence of $\Ext$-groups 
\begin{equation}
\scalebox{0.85}{\xymatrix{
0 \ar[r] & \Hom(\cO_Z,\cO_Z) \ar[r]^\sim & \Hom(\cO_X,\cO_Z) 
\ar[r]^0 & \Hom(\cI_Z,\cO_Z) \ar[r] &  \\
\ar[r] & \Ext^1(\cO_Z,\cO_Z) \ar[r] & \bcancel{\Ext^1(\cO_X,\cO_Z)}
\ar[r] & \Ext^1(\cI_Z,\cO_Z) \ar[r] & \\
\ar[r] & \Ext^2(\cO_Z,\cO_Z) \ar[r] & \bcancel{\Ext^2(\cO_X,\cO_Z)}
\ar[r] & \bcancel{\Ext^2(\cI_Z,\cO_Z)} \ar[r] & 0 \,, 
}}
\label{longExt}
\end{equation}
all vanishing and isomorphisms in which follow from 
$$
\Ext^i(\cO_X,\cO_Z) = H^i(\cO_Z) = 
\begin{cases}
  H^0(\cO_Z) = \Hom(\cO_Z,\cO_Z)\,, & i =0 \,,\\
0 \,, & i>0 \,,
\end{cases} 
$$
because $Z$ is affine. 

Since the support of $Z$ is proper, we can use Serre duality,  which gives
\begin{multline}
\Ext^2(\cO_Z,\cO_Z) = \Hom(\cO_Z,\cO_Z\otimes \cK_X)^* =  \\ 
= H^0(\cO_Z\otimes \cK_X)^* = \chi(\cO_Z\otimes \cK_X)^* =
\chi(\cO_Z,\cO_X) \,,
\label{Ext2S} 
\end{multline}
where 
\begin{equation}
\chi(\cA,\cB) = \sum (-1)^i \Ext^i(\cA,\cB) \,. \label{chiAB}
\end{equation}
This is an sesquilinear pairing on the equivariant K-theory of $X$
satisfying 
$$
\chi(\cA,\cB)^* = (-1)^{\dim X} \chi(\cB,\cA\otimes \cK_X) \,.
$$
Putting \eqref{longExt} and \eqref{Ext2S} together, we obtain 
the following 

\begin{Proposition}\label{p_TanH}
If $X$ is a nonsingular surface then 
\begin{align}
  T_Z \Hilb(X,\textup{points}) &= \chi(\cO_Z) + \chi(\cO_Z,\cO_X) - \chi(\cO_Z,\cO_Z) 
\notag \\
& = \chi(\cO_X) - \chi(\cI_Z,\cI_Z)  \label{TchiI} \,. 
\end{align}
\end{Proposition}

\subsubsection{}
Proposition \ref{p_TanH} lets us easily compute the characters of the 
tangent spaces to the Hilbert scheme at monomial ideals. To 
any $\cF \in K_T(\C^2)$ we can naturally associate two $T$-modules: 
$\chi(\cF)$ and the K-theoretic stalk of $\cF$ at $0\in \C^2$, which we 
can write as $\chi(\cF\otimes \cO_0)$.  They are related by 
\begin{equation}
\chi(\cF) = \chi(\cF\otimes \cO_0) \, \chi(\cO_{\C^2}) 
\label{chiF1} 
\end{equation}
where, keeping in mind that linear functions on $\C^2$ form a $GL(2)$-module 
\emph{dual} to $\C^2$ itself, 
$$
\chi(\cO_{\C^2})  = \bSd \left(\C^2 \right)^* = 
\frac{1}{(1-t_1^{-1}) (1-t_2^{-1})} \,. 
$$
The formula 
\begin{equation}
\chi(\cF,\cG) = \chi(\cF\otimes \cO_0)^*  \chi(\cG\otimes \cO_0)\, 
\chi(\cO_{\C^2})  = \frac{\chi(\cF)^* \,\chi(\cG)} {\chi(\cO)^*} 
\label{chiF2} 
\end{equation}
generalizes \eqref{chiF1}. 

Let $I_\lambda$ be a monomial ideal and let $\lambda$ be 
the corresponding partition with diagram 
$$
\textup{diagram}(\lambda) = \left\{(i,j) \, \Big | \,x_1^i x_2^j \notin I_\lambda
\right\} \subset \Z_{\ge 0}^2 \,,
$$
see Figure \ref{f_partition_diagram}. 
More traditionally, the boxes or the dots in the diagram of $\lambda$ 
are indexed by pairs $(i,j) \in \Z_{> 0}^2$, but this, certainly, a
minor detail. In what follows, we don't make a distinction between 
a partition and its diagram. 
\begin{figure}[!htbp]
  \centering
   \includegraphics[scale=0.75]{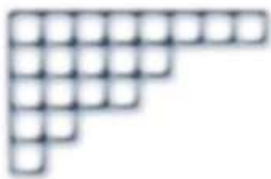}
 \caption{Diagram of the partition $\lambda=(8,5,4,2,1)$.}
  \label{f_partition_diagram}
\end{figure}

Let $\cV = \chi(\cO_Z)$ denote the tautological rank $n$ 
vector bundle over
$\Hilb(\C^2,n)$ as in Section \ref{s_taut}, and let 
\begin{align}
V_\lambda
&= \chi(\cO_{Z_\lambda}) \notag \\
&= \sum_{(i,j)\in \lambda} t_1^{-i} t_2^{-j} \label{Glambda}
\end{align}
be the character of its stalk at $I_\lambda$. Clearly, this 
is nothing but the generating function for the diagram $\lambda$. 
{} From \eqref{TchiI} and \eqref{chiF2}  we deduce the 
following 

\begin{Proposition}
 Let $I_\lambda\subset\C[x_1,x_2]$ be a monomial ideal and let $V$ denote the 
generating function \eqref{Glambda} for the diagram of $\lambda$. 
We have
\begin{equation}
T_{I_\lambda} \Hilb = 
V + \overline{V} \, t_1 t_2 - V  \overline{V}  \, (1-t_1)(1-t_2) \,. 
\label{Tlam} 
\end{equation}
as a $T$-module, where $\overline{V} = V^*$ denotes the dual. 
\end{Proposition}

For example, take 
$$
\lambda = \square\,, \quad I_\lambda = \fm = (x_1,x_2)\,,  \quad 
V = 1 
$$
then 
$$
T_{\square} \Hilb = 1 + t_1 t_2 -  (1-t_1)(1-t_2) = 
t_1 +t_2 = \C^2 
$$
in agreement with $\Hilb(X,1) \cong X$. 

\subsubsection{}
Formula \eqref{Tlam} may be given the following combinatorial 
polish. For a square $\square=(i,j)$ in the diagram of $\lambda$ 
define its arm-length and leg-length by 
\begin{alignat}{2}
  a(\square) &= &\# \{j'>j \, &\big| \, (i,j') \in \lambda \} \,, \notag
  \\ 
l(\square) &= &\# \{i'>i \, &\big| \, (i',j) \in \lambda \} \,, 
\end{alignat}
see Figure \ref{f_arms_legs}. 
\begin{figure}[!htbp]
  \centering
   \includegraphics[scale=0.75]{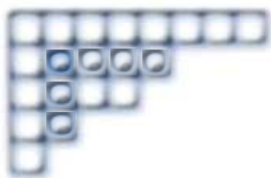}
 \caption{For the box $\square=(1,1)$ in this diagram, we have 
$a(\square)=3$ and $l(\square)=2$.}
  \label{f_arms_legs}
\end{figure}

\begin{Exercise}
Prove that 
\begin{equation}
T_{I_\lambda} \Hilb = \sum_{\square\in \lambda} 
t_1^{-l(\square)} t_2^{a(\square)+1} + t_1^{l(\square)+1}
t_2^{-a(\square)} 
\,. \label{Tarmleg} 
\end{equation}
\end{Exercise}

\begin{Exercise}
Prove a generalization of \eqref{Tlam} and \eqref{Tarmleg} for 
the character of $\chi(\cO_X)- \chi(I_\lambda,I_\mu)$. If in need 
of a hint, open \cite{CO}. 
\end{Exercise}

\begin{Exercise}\label{ex_Z_surface}
Using the formulas for $T_{I_\lambda} \Hilb$ and equivariant 
localization, write a code for the computation of the series 
$$
\bZ_{\Hilb(\C^2)} = \sum_{n,i\ge 0} z^n (-m)^i \, \chi(\Hilb(\C^2,n), 
\Omega^i)
$$
and check experimentally that is $\bSd$ of a nice rational function of the 
variables $z,m,t_1,t_2$.  We will compute this function 
theoretically in Section \ref{s_Z_surface}. 
\end{Exercise}

\subsubsection{}\label{def_Hilb3} 
We now advance to the discussion of the case when $X$ is a
nonsingular threefold and $Z\subset X$ is a $1$-dimensional 
subscheme and $\cI_Z$ is its sheaf of ideals. The sheaf $\cI_Z$ 
is clearly torsion-free, as a subsheaf of $\cO_X$, and 
$$
\det \cI_Z = \cO_X
$$
because the two sheaves differ in codimension $\ge 2$.

Donaldson-Thomas
theory views $\Hilb(X,\textup{curves})$ as the moduli space 
of torsion-free sheaves rank 1 sheaves $\cI$ with trivial
determinant. For any such sheaf we have 
$$
\cI \hookrightarrow \cI^{\vee\vee} = \det \cI = \cO_X
$$
and so $\cI$ is the ideal sheaf of a subscheme $Z\subset X$. 
We have 
$$
0 = c_1(\det \cI) = c_1(\cI) = [Z] \in H_4(X,\Z)
$$
and therefore $\dim Z=1$.  The deformation theory of sheaves 
gives 
\begin{align}
T^\vir_\cI\Hilb = \Def(\cI) - \Obs(\cI) &= \chi(\cO_X) -
                              \chi(\cI,\cI)\,, \notag \\
&=\chi(\cO_Z) + \chi(\cO_Z,\cO_X) - \chi(\cO_Z,\cO_Z) 
\label{DefObscI} 
\end{align}
just like in Proposition \ref{p_TanH}. 

The group $\Ext^1(\cI,\cI)$ 
which enters \eqref{DefObscI} parametrizes sheaves on 
$B \times X$ flat over $B$ for $B=\C[\varepsilon]/\varepsilon^2$, 
just like in Exercise \ref{ex_tan_hilb}. It describes
the deformations of the sheaf $\cI$. The obstructions to 
these deformations lie in $\Ext^2(\cI,\cI)$. 

We now examine how 
this works for the Hilbert scheme of points in $\C^3$. 

\subsubsection{}
Let $\pi$ be a 3-dimensional partition and let $I_\pi \subset 
\C[x_1,x_2,x_3]$ be the corresponding monomial ideal. The 
passage from \eqref{TchiI} to  \eqref{Tlam} is exactly the same 
as before, with the correction for 
$$
\chi(\cO_{\C^3})  = 
\frac{1}{(1-t_1^{-1}) (1-t_2^{-1})  (1-t_3^{-1})} \,. 
$$
We obtain the following 

\begin{Proposition}
 Let $I_\pi\subset\C[x_1,x_2,x_3]$ be a monomial ideal and let $V$
 denote the
generating function for $\pi$, that is, the character of
$\cO_{Z_\pi}$. 
We have
\begin{equation}
T_{I_\pi}^\vir \Hilb = 
V - \overline{V} \, t_1 t_2 t_3 - V  \overline{V}  \, \prod_{1}^3 (1-t_i) \,. 
\label{Tpi} 
\end{equation}
as a $T$-module. 
\end{Proposition}

As an example, take 
$$
\pi = \bx\,, \quad I_\pi = \fm = (x_1,x_2,x_2)\,,  \quad 
V = 1 
$$
then 
$$
T^\vir_{\bx} \Hilb = t_1 + t_2 + t_3 - t_1 t_2 - t_1 t_3 - t_2 t_3 = 
\C^3 - \det \C^3 \otimes \left(\C^3\right)^*
$$
in agreement with the identification 
$$
\mM(1)=\Hilb(X,1) \cong X  \cong \Hilb(\Free_3,1) = \mMh(1) 
$$
for the Hilbert schemes of $1$ point in $X=\C^3$ and 
the description of the obstruction bundle as 
$$
\Obs = \bk \otimes T^*_{\mMh} =  \bk \otimes \left(\C^3\right)^*
$$
where $\bk = \det \C^3$. 

In general, for $\mMh = \Hilb(\Free_3,n)$ we have 
$$
T_\mMh = (\C^3 - 1) \otimes \overline{V} \otimes V + V
$$
by the construction of $\mMh$ as the space of 3 operators and 
a vector in $V$ modulo the action of $GL(V)$.  Therefore
\begin{align}
T^\vir &= T_\mMh  - \bk^{-1} \otimes T^*_{\mMh}  \notag \\
 &=  V - \det \C^3 \otimes \overline V - \overline{V}\otimes {V} 
\otimes \sum (-1)^i
   \Lambda^i \C^3 \,, \label{Tpi2}
\end{align}
which gives a different and more direct proof of \eqref{Tpi}. 

\begin{Exercise}
$\dots$ or, rather, a combinatorial challenge. Is there a combinatorial 
formula for the character of $T_\mMh$ at torus-fixed points 
and can one find some systematic cancellations in the first line 
of \eqref{Tpi2} ? 
\end{Exercise}

\subsubsection{}
Let 
$$
\tilde\iota:  \mM \hookrightarrow \mMh 
$$
be the inclusion of $\mM = \Hilb(\C^3,\textup{points})$ into the 
Hilbert scheme of a free algebra. By our earlier discussion, 
$$
\tilde\iota_* \cO^\vir = \Ld \Obs^*
$$
and therefore we can use equivariant localization on the 
smooth ambient space $\mMh$ to compute $\chi(\mM,\cO^\vir)$. 

Let $\pi$ be 3-dimensional partition and $I_\pi \in \Hilb\subset \mMh$ 
the corresponding fixed point.  Since 
$$
T^\vir_\pi = T_\pi \mMh - \bk \otimes 
\left(  T_\pi \mMh\right)^*\,,
$$
we have 
$$
T^\vir_\pi  = \sum_{w} \left(w - \frac{\bk}{w} \right)
$$
where the sum is over the weights $w$ of $T_\pi \mMh$. Therefore 
\begin{equation}
\cO^\vir_{\Hilb(\C^3,n)} = 
\sum_{|\pi|=n} \cO_{I_\pi} \prod \frac{1-w/\bk}{1-w^{-1}} 
\label{OvirHilbloc} 
\end{equation}
in localized equivariant K-theory. 
This formula illustrates a very important notion of \emph{virtual
localization}, see in particular \cite{GP,FG,CF3}, which we now 
discuss. 

\subsubsection{}

Let a torus $T$ act on a scheme $\mM$ with a $T$-equivariant 
perfect obstruction theory. For example, $\mM$ be DT moduli 
space for a nonsingular threefold $X$ on which a torus $T$ act. 
Let $\mM^T\subset \mM$ be the subscheme of fixed points. 
We can decompose 
\begin{align}
  \left(\Def-\Obs\right)|_{\mM^T} =&\, 
  \Def^\textup{fixed}-\Obs^\textup{fixed}  \notag \\
+&\,  \Def^\textup{moving}-\Obs^\textup{moving}
\end{align}
in which the fixed part is formed by trivial $T$-modules and 
the moving part by nontrivial ones. 

\begin{Exercise}
Check that for $\Hilb(\C^3,\textup{points})$ and the maximal torus 
$T\subset GL(3)$  the fixed part of the obstruction theory vanishes. 
\end{Exercise}

By a result of \cite{GP}, the fixed part of the obstruction theory 
is perfect obstruction theory for $\mM^T$ and defines
$\cO^\vir_{\mM^T}$. The virtual localization theorem of \cite{CF3}, 
see also \cite{GP} for a cohomological version, states that 
\begin{equation}
\cO^\vir_\mM = \iota_*  \left(\cO^\vir_{\mM^T} 
\otimes \bSd \left(\Def^\textup{moving}-\Obs^\textup{moving}
\right)^*\right) \label{virloc} 
\end{equation}
where 
$$
\iota: \mM^T \hookrightarrow \mM 
$$
is the inclusion. Since, by construction, the moving part of the 
obstruction theory contains only nonzero $T$-weights, its 
symmetric algebra $\bSd$ is well defined. 

\begin{Exercise}
 Let $f(\cV)$ be a Schur functor of the tautological bundle $\cV$ 
over $\Hilb(\C^3,n)$, for example 
$$
f(\cV) = S^2 \cV, \Lambda^3 \cV, \dots \,. 
$$
Write a localization formula for $\chi(\Hilb(\C^3,n), \cO^\vir \otimes
f(\cV))$. 
\end{Exercise}

\subsubsection{}\label{s_aroof} 
It remains to twist \eqref{OvirHilbloc} by 
$$
\cK_\vir^{1/2}  = {\det}^{-1/2}\,  T^\vir 
$$
to deduce a virtual localization formula for $\tO^\vir$.  It is 
convenient to define the transformation $\aroof(\dots)$, a 
version of the $\widehat{A}$-genus, by 
$$
\aroof(x+y) = \aroof(x)\,  \aroof(y)\,, 
\quad \aroof(w) = \frac{1}{w^{1/2}-w^{-1/2}}
$$
where $w$ is a monomial, that is, a weight of $T$. For 
example 
\begin{equation}
\aroof\left(T^\vir_\pi\right) = \prod_w 
\frac
{(\bk/w)^{1/2}- (w/\bk)^{1/2}}
{w^{1/2}-w^{-1/2}} \,, \label{aroofTvir} 
\end{equation}
where the product is over the same weights $w$ as in
\eqref{OvirHilbloc}. 

With this notation, we can state the following 

\begin{Proposition} We have
\begin{equation}
\tO^\vir_{\Hilb(\C^3,n)} = 
(-1)^n \, \sum_{|\pi|=n} \aroof\left(T^\vir_\pi\right) \, \cO_{I_\pi} 
\label{tOvirHilbloc} 
\end{equation}
in localized equivariant K-theory. 
\end{Proposition}

\begin{Exercise}
Write a code to check a few first terms in $z$ of Nekrasov's formula
\eqref{f_Hilb3}. 
\end{Exercise}

\begin{Exercise}
Take the limit $t_1,t_2,t_3 \to 1$ in Nekrasov's formula for $\C^3$ and 
show that it correctly reproduces the formula for
$\bZ_{\C^3,\points,\textup{coh}}$ from Theorem \ref{t_Z_coh}. 
\end{Exercise}

%\texttt{Prefactor in \eqref{ZHH}} 

\subsection{Proof of Nekrasov's formula}\label{s_proofN} 

\subsubsection{}
Our next goal is to prove Nekrasov's formula for $X=\C^3$. By
localization, this immediately generalizes to nonsingular toric 
threefolds. Later, when we discuss relative invariants, we will 
see the generalization to the relative setting. The path from 
there to general threefolds is the same as in \cite{LevPand}. 

Our proof of Theorem \eqref{t_Hilb3} will have two parts. 
In the first step, we prove 
\begin{equation}
\bZ_{\C^3,\points} = \bSd  \frac{\star}{(1-t_1^{-1}) (1-t_2^{-1})
  (1-t_3^{-1})}
\label{bZS} 
\end{equation}
where 
$$
\star \in 
\Z\left[t_1^{\pm1},t_2^{\pm1},t_3^{\pm1},(t_1 t_2 t_3)^{1/2}\right] [[z]] 
$$
is a formal power series in $z$, without a constant term, with 
coefficients in Laurent polynomials. In the second step, we 
identify the series $\star$ by a combinatorial argument 
involving equivariant localization. 

The first step is geometric and, in fact, we prove that 
\begin{equation}
\bZ_{\C^3,\points} = \bSd \chi(\C^3,\cG) \,, \quad 
\cG=\sum_{i=1}^\infty z^i \, \cG_i \,,  \quad \cG_i \in
K_{GL(3)}(\C^3) \,,  \label{bZSR} 
\end{equation}
where each $\cG_i$ is constructed iteratively starting from 
$$
\cG_1 = - \tO^\vir_{\Hilb(\C,1)} \, .
$$
The same argument applies to many similar sheaves and moduli 
spaces, including, for example, the generating function 
\begin{equation}
\sum_n z^n \chi(\Hilb(\C^3,n), \cO^\vir) \label{_Zovir} 
\end{equation}
in which we have $z^n$ instead of $(-z)^n$ and plain $\cO^\vir$ instead
of the symmetrized virtual structure sheaf. The second part of 
the proof, however, doesn't work for $\cO^\vir$ and I don't know 
a reasonable formula for the series \eqref{_Zovir}. 

It is natural to prove \eqref{bZSR} in that greater generality and
this will be done in Proposition \ref{P_tOvir_factors}  in 
Section \ref{s_tOvir_factors} below. For now, we assume \eqref{bZSR} 
and proceed to the second part of the proof, that is, to the 
identification of the polynomial $\star$ in \eqref{bZS}. 

\subsubsection{}

For $\Hilb(\C^3,n)$, the line bundles $\cL_4$ and $\cL_5$ in 
\eqref{deftOvir} are necessarily trivial with, perhaps, a nontrivial 
equivariant weight. Hence the determinant term in
\eqref{deftOvir} is a trivial bundle with weight 
\begin{equation}
{\det}^{1/2} \pi_{\mM,*} (\fF\otimes (\cL_4 -
  \cL_5))= \left(
\frac{\textup{weight}(\cL_4)}
{\textup{weight}(\cL_5)} \right)^{n/2} \,,  \label{detHHC} 
\end{equation}
which can be absorbed in the variable $z$. Therefore, without 
loss of generality, we can assume that 
$$
\cL_4  = \cL_5 = \bk^{-1/2} = \frac{1}{(t_1 t_2 t_3)^{1/2}}  \,,
$$
where $t_i$'s are the weights of the $GL(3)$ action on $\C^3$, in 
which case the term \eqref{detHHC} is absent. 

We define 
$$
t_4 = \frac{z}{\bk^{1/2}} \,, \quad t_5 = \frac{1}{z \, \bk^{1/2}} \,, 
$$
so that $t_1,\dots,t_5$ may be interpreted as the weights of the 
action of
$SL(5)\supset \Ct_z$ on $Z\cong \C^5$. With this notation, what 
needs to be proven is
\begin{align}
\bZ_{\C^3,\points} &= \bSd \, \, \aroof\left(\sum_{i=1}^5 t_i - 
\sum_{i<j\le 3} t_i t_j \right) \notag \\
& = \bSd \, \frac{\prod_{i<j\le 3} ((t_i t_j)^{1/2} -  (t_i
  t_j)^{-1/2})}
{\prod_{i\le 5} (t_i^{1/2}- t_i^{-1/2})} \,. 
\end{align}

\subsubsection{}\label{s_strr} 

\begin{Exercise}
Prove that the localization weight \eqref{aroofTvir} of any 
nonempty 3-dimensional partition is divisible by $t_1 t_2  - 1$. 
In fact, the order of vanishing of this weight at $t_1 t_2 =1$ 
is computed in Section 4.5 of \cite{OP2}. 
\end{Exercise}

By symmetry, the same is clearly true for all $t_i t_j-1$ with 
$i<j\le 3$. Note that plethystic substitutions $\{t_i\} \mapsto \{t_i^k\}$
preserve vanishing at $t_i t_j=1$. Therefore, using \eqref{bZS}, 
we may define 
$$
\strr \in \Z[t_1^{\pm 1/2},t_2^{\pm 1/2}, t_3^{\pm 1/2}] [[z]]
$$
so that 
\begin{equation}
 \bZ_{\C^3,\points}=  \bSd \, \left(\strr \, 
\prod_{i=1}^3 
\frac{(\bk/t_i)^{1/2}- (t_i/\bk)^{1/2}}
{t_i^{1/2}- t_i^{-1/2}}
\right)\,. 
\label{def_strr} 
\end{equation}

\begin{Proposition} We have 
$$
\strr \in \Z[\bk^{\pm 1}][[z]]\,, 
$$
that is,
this polynomial depends only on the product $t_1 t_2 t_3$, and not on the 
individual $t_i$'s. 
\end{Proposition}

\begin{proof}
A fraction of the form 
$$
\frac{(\bk/w)^{1/2}- (w/\bk)^{1/2}}
{w^{1/2}-w^{-1/2}}
$$
remains bounded and nonzero as $w^{\pm 1} \to \infty$ with $\bk$ fixed. 
Therefore, both the localization contributions \eqref{aroofTvir}
and the fraction in \eqref{def_strr} remain bounded as $t_i^{\pm 1} 
\to \infty$ in such a way that $\bk$ remains fixed. We conclude 
that the Laurent polynomial $\strr$ is bounded at all such 
infinities and this means it depends on $\bk$ only. 
\end{proof}

\noindent 
This is our first real example of rigidity. 

\subsubsection{}\label{s_strr_next} 

The proof of Nekrasov's formula for $\C^3$ will be complete 
if we show the following 

\begin{Proposition}\label{p_strr}
$$
\strr =  \frac1{(t_4^{1/2}-t_4^{-1/2}) (t_5^{1/2}-t_5^{-1/2})} 
= - \frac{z}{(1-\bk^{1/2} \,z)(1-\bk^{-1/2} \,z)}
$$
\end{Proposition}

To compute $\strr$, we may let the variables 
$t_i$ go to infinity with $\bk$ fixed. Since 
$$
\frac{(\bk/w)^{1/2}- (w/\bk)^{1/2}}
{w^{1/2}-w^{-1/2}} 
\to 
\begin{cases}
- \bk^{-1/2} \,, & w\to \infty \,,\\ 
- \bk^{1/2} \,, & w \to 0 \,, 
\end{cases}
$$
we conclude that 
$$
\aroof \left( \sum w_i  - \bk \sum w_i^{-1} \right) \to 
(-\bk^{1/2})^{\indx}
$$
where 
$$
\indx = \# \left\{ i \,\big|\, w_i \to 0 \right\}  - 
\# \left\{ i \,\big|\, w_i \to \infty \right\} \,. 
$$

For the computation of $\strr$ we are free to send $t_i$ to 
infinity in any way we like, as long as their product stays fixed. 
As we will see, a  particularly nice choice is
\begin{equation}
t_1,t_3 \to 0\,, \quad |t_1| \ll |t_3| \,, \quad \bk = \textup{fixed} \,.
\label{tiinf} 
\end{equation}
For the fraction in \eqref{def_strr} we have 
$$
\prod_{i=1}^3 
\frac{(\bk/t_i)^{1/2}- (t_i/\bk)^{1/2}}
{t_i^{1/2}- t_i^{-1/2}} 
\to (-\bk^{1/2})^{\indx(\C^3)} = - \bk^{1/2} \,.
$$
Thus Proposition \ref{p_strr} becomes a corollary of the 
following 

\begin{Proposition}\label{p_index_Z} 
Let the variables $t_i$ go to infinity of the torus as in
\eqref{tiinf}. Then 
\begin{equation}
\bZ_{X,\points} \to \bSd \, 
\frac{\bk^{1/2}\, z}{(1-\bk^{1/2} \,z)(1-\bk^{-1/2} \,z)} \,. 
\label{f_p_indexZ} 
\end{equation}
\end{Proposition}

This is a special case of the computations of \emph{index vertices} from 
Section 7 of \cite{NO} and the special case of the limit \eqref{tiinf} 
corresponds to the \emph{refined vertex} of 
 Iqbal, Kozcaz, and Vafa in \cite{IKV}. We will now show how 
this works in the example at hand.

\subsubsection{}

Recall the formula 
$$
T_{I_\pi} \mMh =  (\C^3 - 1) \otimes \overline{V} \otimes V + V 
$$
for tangent space to $\mMh$ at a point corresponding to 
a 3-dimensional partition $\pi$ with the generating function 
$V$.  For a partition of size $n$, this 
is a sum of $2n^2+n$ terms and, in principle, we need to 
compute the index of this very large torus module to know the asymptotics
of $\aroof(T^\vir_\pi)$ as $t\to \infty$. 

A special feature of the limit \eqref{tiinf} is that one can see
a cancellation of $2n^2$ terms in the index, with the following
result 

\begin{Lemma}
In the limit \eqref{tiinf}, 
$$
\indx \left(T_{I_\pi} \mMh \right)= \indx \left( t_3^k \, V \right)\,,
$$
for any $k$ is such that $k > |\pi|$ but $|t_1| \ll |t_3|^k$. 
\end{Lemma}

\noindent 
This Lemma is proven in the Appendix to \cite{NO}.  Clearly 
$$
 \indx \left( t_3^k \, V \right) = \sum_{\bx=(i_1,i_2,i_3)\in \pi}
\textup{sgn}\,(i_2-i_1+0) 
$$
and therefore Proposition \ref{p_index_Z} becomes 
the 
\begin{align}
  z \bk^{1/2} & = q_0 = q_1 = q_2 = \dots \notag \\
  z \bk^{-1/2} & = q_{-1} = q_{-2} = \dots \label{qspeczk} 
\end{align}
case of the following generalization of McMahon's enumeration 

\begin{Theorem}[\cite{OR1,OR2}] \label{t_McM} 
  \begin{equation}
\sum_\pi \prod_{\bx=(i_1,i_2,i_3)\in \pi}  q_{i_2-i_1} = 
\bSd \sum_{a\le 0 \le b} q_a q_{a+1} \cdots q_{b-1} q_b 
\label{fOR12} 
\end{equation}
\end{Theorem}

With the specialization \eqref{qspeczk} the sum under $\bSd$ in 
\eqref{fOR12} become the series expansion of the 
fraction in \eqref{f_p_indexZ}.

\subsubsection{}

We conclude with a sequence of exercises that will lead the 
reader through the proof of Theorem \ref{t_McM}. It is 
based on introducing a \emph{transfer matrix}, 
a very much tried-and-true tool in statistical mechanics 
and combinatorics. 

A textbook example of transfer matrix arises in 2-dimensional 
Ising model on a torus. The configuration space in that model 
are assignments of $\pm 1$ spins to sites $x$ of a rectangular 
grid, that is, functions 
$$
\sigma: \{1,\dots,M\} \times \{1,\dots,N\} \to \{\pm 1\} 
$$
Each is weighted with 
$$
\Prob(\sigma) \propto W(\sigma)=\exp(-\beta E(\sigma)) \,,
$$
where $\beta$ is the inverse temperature and the energy $E$ is defined by 
$$
E(\sigma)  = - \sum_{\textup{nearest neighbors $x$ and $x'$}} \sigma(x) \sigma(x') \,,
$$
with periodic boundary conditions. 

Now introduce a vector space 
$$
V \cong (\C^2)^{N}
$$
with a basis formed by all possible 
spin configurations $\vec \sigma$ in one column of our grid. 
Define a diagonal matrix $\bW^{\textup{v}}$ and a dense matrix 
$\bW^{\textup{h}}$ by 
\begin{align}
\bW^{\textup{v}}_{\vec\sigma,\vec\sigma} 
&= \exp\left(-\beta \sum_{\substack{\textup{vertical 
neighbors}\\ \textup{ in column 
$\vec\sigma$}}} \sigma(x) \, \sigma(x')\right) \\
\bW^{\textup{h}}_{\vec\sigma_1,\vec\sigma_2} 
&= \exp\left(-\beta \sum_{\substack{\textup{horizontal 
neighbors}\\ \textup{ in columns 
$\vec\sigma_1$ and $\vec\sigma_2$
}}}
\sigma_1(x) \, \sigma_2(x')\right)\,,
\end{align}
see Figure \ref{f_Ising}. 

\begin{Exercise}
  Prove that 
$$
\sum_{\{\sigma\}} W(\sigma) = \tr \, (\bW^{\textup{v}} \bW^{\textup{h}})^M 
$$
where the summation is over all $2^{MN}$ spin configurations
$\{\sigma\}$. 
\end{Exercise}

\begin{figure}[!htbp]
  \centering
   \includegraphics[scale=0.64]{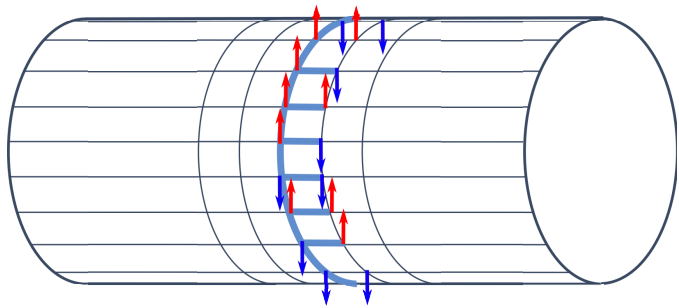}
 \caption{The transfer matrix $\bW^{\textup{v}} \bW^{\textup{h}}$ 
of the Ising model captures the
   interactions of the spins in two adjacent columns along the 
highlighted edges. Iterations of the transfer matrix cover all 
spins and all edges.}
  \label{f_Ising}
\end{figure}

The matrix $\bW^{\textup{v}} \bW^{\textup{h}}$ is an example of 
a transfer matrix. 
Onsager's great discovery was the diagonalization of this matrix. 
Essentially, he has shown that the transfer matrix 
 is in the image of a certain 
matrix $g\in O(2N)$ in the spinor representation of this group. 

\subsubsection{}

Now in place of the spinor representation of $O(2N)$ we will have
the action of the Heisenberg algebra $\glh(1)$ on 
\begin{multline*}
  V = \textup{Fock space} = \textup{symmetric functions} =\\
= \textup{polynomial representations of
    $GL(\infty)$} \,. 
\end{multline*}
This space has an 
orthonormal basis $\{s_\lambda\}$ of Schur 
functions, that is, the characters of the Schur functors $S^\lambda \C^\infty$
of the defining representation of $GL(\infty)$. 
It is  indexed by 2-dimensional partitions $\lambda$ which will 
arise as slices of $3$-dimensional
partitions $\pi$ by planes $i_2-i_1 = \textup{const}$. 

The diagonal operator $\bW^{\textup{v}}$ will be replaced by 
the operator 
$$
q^{|\,\cdot \,|} \, s_\lambda = q^{|\lambda|} s_\lambda \,. 
$$
In place of the nondiagonal operator $\bW^{\textup{h}}$, we will 
use the operators 
$$
\Gamma_-(z) = \left(\sum_{k\ge 0} z^k \, \bS^k \C^\infty \right) \otimes 
$$
and its transpose $\Gamma_+$. The character of $\bS^k \C^\infty $ is 
the Schur function $s_k$. It is well known that
$$
\Gamma_-(z) \, s_\lambda = \sum_\mu z^{|\mu|-|\lambda|} \, s_\mu 
$$
where the summation is over all partitions $\mu$ such that 
$\mu$ and $\lambda$ interlace, which means that 
$$
\mu_1 \ge \lambda_1 \ge \mu_2 \ge \lambda_2 \ge \dots \,. 
$$

\begin{Exercise}
Prove that 
\begin{multline*}
  \sum_\pi \prod_{\bx=(i_1,i_2,i_3)\in \pi} q_{i_2-i_1} = \\
=\left(
    \cdots \Gamma_+(1) \, q_{-1}^{|\,\cdot \,|} \,\Gamma_+ (1)\,
    q_{0}^{|\,\cdot \,|} \,\Gamma_- (1) \, q_{1}^{|\,\cdot \,|} \,
    \Gamma_-(1) \, q_{2}^{|\,\cdot \,|} \cdots s_{\varnothing},
    s_{\varnothing}\right) \,.
\end{multline*}
\end{Exercise}

\begin{Exercise}
Prove that $q^{|\,\cdot \,|} \, \Gamma_-(z) = 
 \Gamma_-(qz) \, q^{|\,\cdot \,|}$ and that 
$$
 \Gamma_+(z)  \, \Gamma_-(w) = 
\frac1{1-zw} \Gamma_-(w) \,\Gamma_+(z) 
$$
if $|zw|<1$. Deduce Theorem \ref{t_McM}. 

\end{Exercise}

\subsubsection{}\label{s_transfer_1} 

Ultimately, the idea of transfer matrix rests on certain 
fundamental principles of mathematical physics which may be 
applied both in the discrete and continuous situations. 

The first such principle is \emph{gluing} local field theories, which 
in the discrete situation means the following. Suppose we 
want to compute a partition function of the form 
$$
Z = \sum_{\phi\in \Phi^\Omega} e^{-E(\phi)} 
$$
where $\Omega$ is a graph, $\Phi^\Omega$ is the space of functions 
$$
\phi: \Omega \to \Phi 
$$
and 
$$
E(\phi) = \sum_{x,y\in \Omega} V_2(x,y,\phi(x),\phi(y))  + 
\sum_{x\in\Omega} V_1(x,\phi(x)) \,. 
$$
For example, for the Ising model we have $\Phi=\{\pm 1\}$,
$$
V_2(x,y,\phi(x),\phi(y)) = 
\begin{cases}
- \beta \phi(x) \phi(y) \,, \quad & \textup{$x$ and $y$ are
  neighbors}\,, \\
0  & \textup{otherwise} \,,
\end{cases}
$$
and $V_1=0$ if there is no external magnetic field. 

Suppose the pair potential $V_2$ has finite range, that is, suppose 
there exists a constant $R$ such that 
$$
\dist(x,y) > R  \Rightarrow V_2(x,y,\dots) = 0  \,. 
$$
For example, for the Ising model $R=1$. Suppose 
$$
\Omega = \Omega_1 \sqcup B \sqcup \Omega_2 
$$
where the boundary region $B$ disconnects $\Omega_1$ from $\Omega_2$ 
in the sense that 
$$
\dist(\Omega_1,\Omega_2) > R 
$$
as in Figure \ref{f_domain1}. 
Then 
\begin{equation}
Z 
=\sum_{\phi_B \in \Phi^{B}} e^{-E(\phi_B)} \, Z(\Omega_1 \big| \phi_B) \, 
Z(\Omega_2 \big| \phi_B) \,, \label{ZZB} 
\end{equation}
where the partition functions  $Z(\Omega_1 \big| \phi_B)$ with 
boundary conditions imposed on $B$ are defined by 
$$
Z(\Omega_i \big| \phi_B) = \sum_{\phi_{\Omega_i} \in \Phi^{\Omega_i}} 
e^{-E(\phi_{\Omega_i}) - E(\phi_{\Omega_i}\big| \phi_B)} 
$$
with 
$$
E(\phi_{\Omega_i}\big| \phi_B) = \sum_{x\in \Omega_i, y \in B} 
V_2(x,y,\phi(x),\phi(y)) \,. 
$$
Effectively, fixing the boundary conditions on $B$ modifies the 
1-particle potential $V_1$ in an tube of radius $R$ around $B$. 

If $V_1$ and $V_2$ are real, then clearly \eqref{ZZB} may be written as
$$
Z  = \Big( Z(\Omega_1\big| \, \cdot \,), Z(\Omega_2\big| \, \cdot \,)
\Big)_{L^2(\Phi^B,e^{-E(\phi_B)})}  \,.
$$
In any event, it may be more prudent to interpret \eqref{ZZB} as 
a pairing between a pair of dual vector spaces, as in the continuous 
limit the two partition functions paired in \eqref{ZZB} 
may turn out to be objects of rather different nature. In this very 
abstract form, gluing says that if $B$ disconnects $\Omega$, then 
partition functions $Z(\Omega_1\big| \, \cdot\,)$ give a pair of vectors
in dual vector spaces so that 
\begin{equation}
Z  = \left\langle Z(\Omega_1\big| \, \cdot \,), Z(\Omega_2\big| \, \cdot \,)
\right\rangle  \,. 
\end{equation}
This is probably very familiar to most readers from TQFT context, but
is not restricted to topological theories. 
\begin{figure}[!htbp]
  \centering
   \includegraphics[scale=0.64]{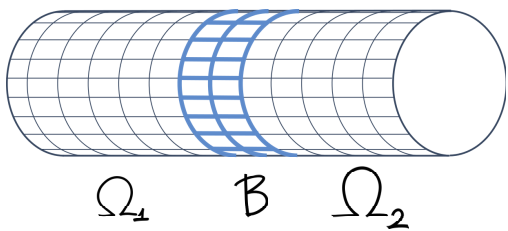}
 \caption{If the domain $\Omega$ is glued from two pieces 
   then the partition function is a pairing of a vector and a covector.}
  \label{f_domain1}
\end{figure}

The second principle on which transfer matrices rest is that if $B$
itself is disconnected 
$$
B = B_1 \sqcup B_2\,, \quad \dist(B_1,B_2) > R \,, 
$$
then 
$$
L^2(\Phi^B) = L^2(\Phi^{B_1}) \otimes L^2(\Phi^{B_2})\,, 
$$
as inner product spaces. 
Again, in continuous limit, it may be better to talk about operators 
from one functional space to another --- the transfer matrix, or 
the evolution operator. Multiplying such operators, we can 
build larger domains from small pieces (e.g. a long cylinder 
from a short one) like we did in the two 
examples above.  
\begin{figure}[!htbp]
  \centering
   \includegraphics[scale=0.64]{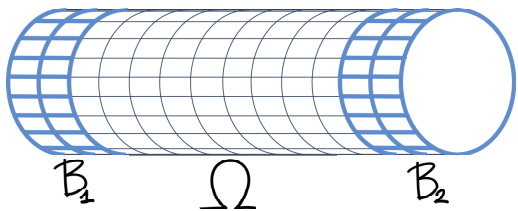}
 \caption{If the boundary has two components then 
the partition function is an operator. }
  \label{f_domain2}
\end{figure}

A larger supply of operators may be obtained
by modifying the partition function $Z(\Omega \big| B_1 \cup B_2)$ 
by, for example, inserting a local observable. 

\subsubsection{}\label{s_transfer_2} 
The purpose of this abstract discussion of transfer matrices is that 
our goals in these notes may be directly compared to understanding 
the transfer matrix of the Ising model or 3-dimensional partitions 
in terms of representation theory of a certain algebra. 

For a mathematical physicist, the study of quasimaps to Nakajima
varieties appears as a low energy limit in the study of supersymmetric gauge 
theories. A gauge theory is specified by a choice of a gauge group $G$
and 
a representation in which the matter fields of theory
transform. 
If the choices are made like in the next
section, then a Nakajima variety $X$ appears as the one of the components 
(the Higgs branch) of the moduli spaces of vacua.
In the low energy
limit (also known as  the thermodynamic, or infrared limit), the state
of the system is described as a modulated vacuum state, or as a map 
from the spacetime to the moduli spaces of vacua. 

K-theory of
quasimaps $f: C \dasharrow X$, where $C$ is a Riemann surface, 
 is relevant for the study of 
3-dimensional supersymmetric gauge theories on manifold of the form 
$C \times S^1$. General ideas about cutting and gluing apply
 to the surface $C$. In 
particular, $K(X)$ will be the vector space associated to a puncture in
$C$ and various vectors and operators in this space that will 
occupy us below may be interpreted as partition functions with 
given boundary conditions. The eventual goal will be to understand 
these operators in terms of representation theory of a certain 
quantum loop algebra $\cU_\hbar(\fgh)$ acting on $K(X)$, see 
Sections \ref{s_Stab_Q}, \ref{s_KZ}, and \cite{OS}.  

This quantum loop algebra action 
extends the actions constructed by Nakajima \cite{Nak3}. 
Nekrasov and Shatashvili \cite{NS1,NS2} were the first to realize this 
connection between supersymmetric gauge theories, 
quantum integrable systems, and Nakajima theory.

\section{Nakajima varieties} \label{s_Nak} 

\subsection{Algebraic symplectic reduction} 

\subsubsection{} 

Symplectic reduction was invented in classical mechanics \cite{ArnMech}
to deal with the following situation. Let $M$ be the 
configuration space of a mechanical system and $T^*M$ ---
the corresponding phase space. A function $H$, called 
Hamiltonian, generates dynamics by 
$$
\frac{d}{dt} f = \{H, f\}
$$
where $f$ is an arbitrary function on $T^*M$ and $\{\,\cdot\,, 
\,\cdot\,\}$ is the Poisson bracket. If this dynamics commutes
with a Hamiltonian action of a Lie group $G$, 
it descends to a certain reduced phase space $T^*M\rd G$. 
The reduced space could be a more complicated variety but of smaller 
dimension, namely 
$$
\dim T^*M\rd G = 2 \dim M  - 2 \dim G \,. 
$$

\subsubsection{} 

In the algebraic context, 
let a reductive group $G$ act on a smooth algebraic 
variety $M$. The induced action on $T^*M$, which is 
a algebraic symplectic variety, is Hamiltonian: 
the function 
\begin{equation}
  \label{Xrd}
  \mu_\xi (p,q) = \lan p, \xi \cdot q \ran \,, \quad q\in M\,, 
p\in T^*_q M\,, 
\end{equation}
generates the vector field $\xi\in \Lie(G) $.  This gives a map 
$$
\mu: T^*M \to \Lie(G)^*\,,
$$
known as the \emph{moment} map, because in its mechanical origins 
$G$ typically included translational or rotational symmetry. 

\subsubsection{}

We can form the algebraic symplectic reduction 
\begin{equation}
X=T^*M \rd G  = \mu^{-1}(0) \rdd G = \mu^{-1}(0)_\textup{semistable}
/G \,,
\label{alg_sym}
\end{equation}
where a certain choice of stability is understood.

\subsubsection{} 

The zero section $M \subset T^*M$ is automatically inside
$\mu^{-1}(0)$ and 
$$
T^*(M_\textup{free}/G) \subset X
$$
is an open, but possibly empty, subset. Here $M_\textup{free} \subset
M$ is the set of semistable points with trivial stabilizer. Thus 
algebraic symplectic reduction is an improved version of the 
cotangent bundle to a $G$-quotient.

\subsubsection{}
The Poisson bracket on $T^*M$ induces a Poisson bracket on 
$X$, which is symplectic on the open set of points with trivial 
stabilizer. In general, however, there will be other, singular, points 
in $X$. 

Finite stabilizers are particularly hard to avoid. Algebraic
symplectic reduction is a source of great many Poisson orbifolds, 
but it very rarely outputs an algebraic symplectic 
variety. 

\subsection{Nakajima quiver varieties \cite{Nak1,Nak2,GinzNak}} 

\subsubsection{}

Nakajima varieties are a remarkable class of symplectic reductions 
for which finite stabilizers can be avoided. For them, 
$$
G = \prod GL(V_i) 
$$
and $M$ is a linear representation of the form 
\begin{equation}
M = \bigoplus_{i,j} \Hom(V_i,V_j) \otimes Q_{ij}  \oplus 
\bigoplus_{i} \Hom(W_i,V_i)  \,. \label{repM}
\end{equation}
Here $Q_{ij}$ and $W_i$ are multiplicity spaces, so that 
\begin{equation}
\prod GL(Q_{ij}) \times  \prod GL(W_i)  \times
\Ct_\hbar  \to \Aut(X) \,, \label{AutNak}
\end{equation}
where the $\Ct_\hbar$-factor scales the cotangent directions with weight 
$\hbar^{-1}$, and hence scales the symplectic form on $X$ with 
weight $\hbar$. 

\subsubsection{}

In English, the only representations allowed in $M$ are 
\begin{itemize}
\item[---] the 
defining representations $V_i$ of the $GL(V_i)$-factors, 
\item[---] the 
adjoint representations of the same factors,
\item[---] representations of the form 
$\Hom(V_i,V_j)$ with $i\ne j$. 
\end{itemize}
Latter are customary called
bifundamental representations in gauge theory context. 
Note that $T^*M$ will also include the duals $V_i^*$ of the 
defining representations. 

What is special about these representation is that the 
stabilizer $G_y$ of any point $y\in T^*M$ is the set of 
invertible elements in a certain associative algebra
$$
E_y \subset \bigoplus \End(V_i) 
$$
over the base field $\C$, and hence cannot be a nontrivial 
finite group. 

\subsubsection{}
The data of the representation \eqref{repM} is conveniently 
encoded by a graph, also called a quiver, 
in which we join the $i$th vertex with 
the $j$th vertex by $\dim Q_{ij}$ arrows. After passing 
to $T^*M$, the orientation of these arrows doesn't matter because 
$$
\Hom(V_i,V_j)^* = \Hom(V_j,V_i)  \,. 
$$
Therefore, it is convenient to assume that only one of the spaces
$Q_{ij}$ and $Q_{ji}$ is nonzero for $i\ne j$. 

To the vertices of the quiver, one associates two dimension vectors 
$$
\bv = (\dim V_i) \,, \,\, \bw = (\dim W_i) \in \Z_{\ge 0}^I \,, 
$$
where $I=\{i\}$ is the set of vertices. 

\subsubsection{} \label{stheta} 

The quotient in \eqref{Xrd} is a GIT quotient and choice of 
stability condition is a choice of a character of $G$, that is, a choice 
of vector $\theta = \Z^I$, up to positive 
proportionality. For $\theta$ away 
from certain hyperplanes, there are no strictly semistable points 
and Nakajima varieties are holomorphic symplectic 
varieties. 

\begin{Exercise}
Let $Q$ be a quiver with one vertex and no arrows. Show that, for 
either choice of the stability condition, 
the corresponding Nakajima varieties are the cotangent bundles 
of Grassmannians. 
\end{Exercise}

\begin{Exercise}\label{Ex_Hilb_Quiv} 
Let $Q$ be a quiver with one vertex and one loop. For $\bw=1$, 
identify the Nakajima varieties with $\Hilb(\C^2,n)$ where $n=\bv$. 
\end{Exercise}

\subsection{Quasimaps to Nakajima varieties}

\subsubsection{}\label{s_sigma} 
The general notion of a quasimap to a GIT quotient is discussed in
detail in \cite{CKM}, here we specialize it to the case of Nakajima 
varieties $X$. 
Twisted quasimaps, which will play an important technical 
role below, are a slight variation on the theme. 

Let $\bT$ be the maximal torus of the group \eqref{AutNak} and 
$\bA = \Ker \hbar$ the subtorus preserving the symplectic form. 
Let 
\begin{equation}
  \label{def_sigma}
  \sigma: \Ct \to \bA 
\end{equation}
be a cocharacter of $\bA$, it will determine how the quasimap
to $X$ is twisted. In principle, nothing prevents one from similarly
twisting  by a cocharacter of $\bT$, but this will not be done in 
what follows. 

\subsubsection{}\label{s_O(p)}

Let $C \cong \bP^1$ denote the domain of the quasimap. 
Since $\bA$ acts on multiplicity spaces $W_i$ and $Q_{ij}$, a choice 
of $\sigma$ determines bundles $\cW_i$ and $\cQ_{ij}$ over 
$C$ as bundles associated to $\cO(1)$.  To fix 
equivariant structure we need to linearize $\cO(1)$ and the 
natural choice is $\cO(p_1)$, where $p_1\in C$ is a fixed point 
of the torus action.

\subsubsection{}

By definition, a twisted quasimap 
$$
f: C \dasharrow X
$$
is a collection of 
vector bundles $\cV_i$ on $C$ of ranks $\bv$ and a section 
$$
f  \in H^0(C,\cM \oplus \cM^* \otimes \hbar^{-1}) 
$$
satisfying $\mu=0$, where
\begin{equation}
\cM = \bigoplus_{i,j} \cHom(\cV_i,\cV_j) \otimes \cQ_{ij}  \oplus 
\bigoplus_{i} \cHom(\cW_i,\cV_i)  \,. \label{repcM}
\end{equation}
Here $\hbar^{-1}$ is a trivial line bundle with weight $\hbar^{-1}$, inserted 
to record the $\bT$-action on quasimaps (in general, the centralizer
of $\sigma$ in $\Aut(X)$ acts on twisted quasimaps).  One can 
replace $\hbar^{-1}$ by an arbitrary line bundle and that will
correspond to $\bT$-twisted quasimaps. 

\subsubsection{}
Let $p\in C$ be a point in the domain of $f$ and fix a local 
trivialization of $\cQ_{ij}$ and $\cW_i$ at $p$. 

The value $f(p)$ of a quasimap at a point $p\in C$ gives a 
well-defined $G$-orbit in $\mu^{-1}(0) \in T^*M$ or, in a more 
precise language, it defines a map 
\begin{equation}
\ev_p(f) = f(p) \in \big[ \mu^{-1}(0)/G \big] \supset X \label{evp}
\end{equation}
to the quotient stack, which contains $X =
\mu^{-1}(0)_\textup{stable}/G$ as an open set. 
By definition, a quasimap is \emph{stable} if 
$$
f(p) \in X
$$
for all but finitely many points of $C$. These exceptional 
points are called the
\emph{singularities} of the quasimap.

\subsubsection{}
We consider twisted quasimaps up to isomorphism that is 
required to be an identity on $C$ and on the 
multiplicity bundles $\cQ_{ij}$ and $\cW_i$.  
In other words, we consider quasimaps from parametrized domains, 
and twisted in a fixed way. We define 
\begin{equation}
  \label{QM}
  \QM(X) = 
\{ \textup{stable twisted quasimaps to $X$} \} \big/ \cong \ 
\end{equation}
with the understanding that  it is the data
of the bundles $\cV_i$ and of the section $f$ that varies in these
moduli 
spaces, while the 
curve $C$ and the twisting bundles $\cQ_{ij}$ and $\cW_i$ are 
fixed\footnote{
More precisely, for relative quasimaps, to be discussed below, 
the curve $C$ is allowed to change to $C'$, where 
$$
\pi: C' \to C 
$$
collapses some chains of $\bP^1$s. The bundles $\cQ_{ij}$ and $\cW_i$
are then pulled back by $\pi$.}. 

As defined, $\QM(X)$ is a union of countably many moduli spaces
of quasimaps of given degree, see below.

\subsubsection{}
The degree of a quasimap is the vector 
\begin{equation}
\deg f = (\deg \cV_i) \in \Z^I \,. \label{degf}
\end{equation}
For nonsingular quasimaps, this agrees with the usual notion of degree
modulo the expected generation of $H^2(X,\Z)$ by first Chern classes
of the tautological bundles. 

The graph of a nonsingular twisted quasimap is a curve in a nontrivial 
$X$ bundle over $C$, the cycles of effective curves in which lie in an 
extension of $H_2(C,\Z)$ by $H_2(X,\Z)$. 
Formula \eqref{degf} is a particular way to split this 
extension\footnote{
There is no truly canonical notion of a degree zero twisted quasimap
and the prescription \eqref{degf} depends on previously made choices. 
Concretely, if $\sigma$ and $\sigma'$ differ by something in 
the kernel of \eqref{AutNak} then the corresponding twisted quasimap 
moduli spaces are naturally 
isomorphic. This isomorphism, however, may not 
preserve degree.}. 

\subsubsection{}\label{s_const_map} 
Every fixed point $x \in X^\sigma$ 
defines a ``constant'' twisted quasimap with $f(c)=x$ for 
all $c\in C$. The degree of this constant map, which  is nontrivial, 
is computed as follows. 

A fixed point of $a\in \bA$ in a quiver variety means $a$ acts 
in the vector spaces $V_i$ so that all arrow maps are
$a$-equivariant. This gives the $i$-tautological line bundle 
\begin{equation}
  \label{cLi}
  \cL_i = \det V_i 
\end{equation}
an action of $a$, producing a locally constant map 
\begin{equation}
  \label{def_mu}
\bmu:  X^\bA \to \Pic(X)^\vee \otimes \bA^\vee\,,
\end{equation}
compatible with restriction to subgroups of $\bA$.  We have
\begin{equation}
\deg \big(f\equiv x\big)  = \lan \bmu(x), \textup{---} \otimes \sigma
\ran 
\label{deg_const}
\end{equation}
where we used the pairing of characters with cocharacters.

The map 
\eqref{def_mu} can be seen as the universal \emph{real} moment 
map: the moment maps for different K\"ahler forms 
$$
\omega_\R \in H^{1,1}(X) = \Pic(X) \otimes_{\Z} \R 
$$
map fixed points to different points of 
$$
\Lie(\bA_\textup{compact})^\vee \cong \bA^\vee \otimes_{\Z} \R  \,.
$$

\subsubsection{}
Moduli spaces of stable quasimaps have a perfect obstruction theory 
with 
\begin{equation}
T_\vir = \Hd(\cM \oplus \hbar^{-1} \cM^*) - 
(1+\hbar^{-1}) \sum \Extd(\cV_i,\cV_i) \,.\label{TvirQM}
\end{equation}
The second term accounts for the moment map equations as well 
as for
\begin{align*}
  - \Hom(\cV_i,\cV_i) &= - \Lie \Aut(\cV_i)  \\
    \Ext^1(\cV_i,\cV_i) & = \textup{deformations of $\cV_i$} \,. 
\end{align*}
With our assumptions on the twist, the degree terms vanish in 
the Riemann-Roch formula, and we get 
$$
\rk T_\vir = \dim X \,, 
$$
as the virtual dimension. 

\subsubsection{}

In general, \cite{CKM} consider quasimaps to quotients 
$X=Y\rdd G$ where $Y$ is an affine algebraic variety with an 
action of a reductive group $G$ such that 
$$
\varnothing \ne Y_\textup{$G$-stable}    = Y_\textup{$G$-semistable} 
\subset Y_\textup{nonsingular} \,. 
$$
Deformation theory of a quasimap $f: C\dasharrow X$ may be 
studied by
\begin{itemize}
\item[---] first, deforming the quasimap in affine charts of $C$, and then
\item[---] patching these deformations together. 
\end{itemize}
Global deformations and obstructions arise in the second step 
as cohomology of a
certain complex constructed in first step. For the obstruction theory 
of quasimaps to be perfect, it is thus crucial that the local 
deformation theory is perfect, and this is achieved by 
requiring that $Y$ is a \emph{local complete intersection}. Further,
to insure that $H^1(C,\textup{local obstructions})=0$, we need
$$
\dim \supp \textup{local obstructions} = 0 
$$
and this follows from stability of $f$, by which the generic point of
$C$ is mapped to the smooth orbifold $Y_\textup{$G$-stable}\rdd G$.

\subsubsection{}

Let $C$ be a smooth curve of arbitrary genus and 
\begin{equation}
Y = 
\begin{matrix}
\cL_1 \oplus \cL_2  \\
\downarrow\\
C
\end{matrix} \,, 
\label{YLL} 
\end{equation}
the total space of two line bundles over $C$. We can use
$\cL_1$ and $\cL_2$ to define $\bT$-twisted quasimaps 
from $C$ to $\Hilb(\C^2,n)$ as follows. By definition 
$$
f: C \dasharrow \Hilb(\C^2,n)
$$
is a vector bundle  $\cV$ on $C$ of rank $n$ together with 
a section 
$$
f = (\bv, \bv^\vee, \bX_1,\bX_2) 
$$
of the bundles
\begin{align*}
\bv &\in H^0(\cV) \,, \\
\bv^\vee &\in  H^0(\cL^{-1}_1 \otimes \cL^{-1}_2 \otimes \cV^\vee)\,,
  \\
\bX_i &\in H^0(\End(\cV)\otimes 
\cL^{-1}_i) \,,
\end{align*}
satisfying the equation 
$$
\left[\bX_1,\bX_2\right] + \bv \, \bv^\vee = 0 
$$
of the Hilbert scheme and the stability condition. 
The stability 
condition forces 
$$
\left[\bX_1,\bX_2\right] = 0 \,, 
$$
as in Exercise \ref{Ex_Hilb_Quiv}. 

\begin{Exercise}
Show this data defines a coherent sheaf $\cF$ on the threefold $Y$, 
together with a section 
$$
s: \cO_Y \to \cF \,.
$$
\end{Exercise}

\begin{Exercise}
Show the quasimap data is in bijection with 
complexes 
$$
\cO_Y \xrightarrow{\quad s\quad} \cF \,, 
$$
of sheaves on $Y$ such that: 
\begin{itemize}
\item[---] the sheaf $\cF$ is 1-dimensional and \emph{pure}, that is,
  has
no $0$-dimensional subsheaves, and 
\item[---] the cokernel of the section $s$ is $0$-dimensional. 
\end{itemize}
By definition, such complexes are parametrized by the
Pandharipande-Thomas (PT) moduli spaces for $Y$. 
\end{Exercise}

\begin{Exercise}
Compute the virtual dimension of the quasimap/PT moduli spaces. 
\end{Exercise}

\section{Symmetric powers}\label{s_PT1}

\subsection{PT theory for smooth curves}\label{sub_PT1} 

\subsubsection{}

Let $X$ be a nonsingular threefold. By definition, a point in the 
Pandharipande-Thomas moduli space is a pure $1$-dimensional 
sheaf with a section
\begin{equation}
\cO_X \xrightarrow{\,\, s \,\,} \cF \label{PTmod} 
\end{equation}
such that $\dim \Coker s = 0$. Here pure means that $\cF$ has 
no $0$-dimensional subsheaves. 

\begin{Exercise}\label{Ex_C_pure} 
Consider the the structure sheaf 
$$
\cO_{C} = \cO_{X} / \Ann(\cF) 
$$
of the scheme-theoretic support of $\cF$. Show it is also pure
$1$-dimensional. 
\end{Exercise}

Another way of saying the conclusion of Exercise \ref{Ex_C_pure}   is 
that $C$ is a $1$-dimensional 
Cohen-Macaulay subscheme of $X$, a scheme that for any 
point $c\in C$ has a function that vanishes at $c$ but is not 
a zero divisor. 

\subsubsection{}

In this section, we consider the simplest case when $C$ is a
reduced smooth curve in $X$.  This forces \eqref{PTmod} to have
the form 
$$
\cO_X \to \cO_C \to \cO_C(D) = \cF
$$
where $D\subset C$ is a divisor, or equivalently, a zero-dimensional 
subscheme, and the maps are the canonical ones. 
The moduli space, for fixed $C$, is thus
\begin{equation}
\mM = \bigsqcup_{n\ge 0} S^n C = \bigsqcup_{n\ge 0} \Hilb(C,n)  \,. 
\label{mMSnC} 
\end{equation}
The full PT moduli space also has directions that correspond to 
deforming the curve $C$ inside $X$, with deformation theory given 
by 
$$
\Def(C) - \Obs(C) = \Hd(C,N_{X/C}) \,, 
$$
where $N_{X/C}$ is the normal bundle to $C$ in $X$.

 Here we fix $C$ and focus on \eqref{mMSnC}. Our goal is to 
relate
$$
\bZ_C = \chi(\mM,\tO^\vir)
$$
to deformations of the curve $C$ in the 4th and 5th directions in 
\eqref{ZLL}, that is, to $\Hd(C,z \cL_4 \oplus z^{-1} \cL_5)$. 

\subsubsection{}\label{s_MemSn}
Remarkably, we will see that all 4 directions of 
\begin{equation}
N_{Z/C} = N_{X/C} \oplus z \cL_4 \oplus z^{-1} \cL_5
\label{NZC} 
\end{equation}
enter the full PT computation completely symmetrically. 

This finds a natural explanation in the conjectural correspondence
\cite{NO} between K-theoretic DT counts and K-theoretic counting 
of membranes in M-theory. From the perspective of M-theory, 
the curve $C\subset Z$ is a supersymmetric membrane and its bosonic degrees
 of freedom are simply motions in the transverse directions. 
All directions of $N_{Z/C}$ contribute equally to those. 

While Theorem \ref{t_Sn} below  is a very basic check of the conjectures
made in \cite{NO}, it does count as a nontrivial evidence in their
favor. 

\subsubsection{}

As already discussed in Exercise \ref{ex_Hilb_curve}, the 
moduli space $\mM$ is smooth with the cotangent bundle 
$$
\Omega^1 \mM = H^0(\cO_D \otimes \cK_C) \,. 
$$
Pandharipande-Thomas moduli spaces have a perfect obstruction theory which 
is essentially the same as the deformation theory from 
Section \ref{def_Hilb3}, that is, the deformation theory of
complexes 
$$
\cO_X \twoheadrightarrow \cF
$$
with a \emph{surjective} map to $\cF\cong\cO_Z$. 
In particular, 
\begin{equation}
\Def-\Obs = \chi(\cF) + \chi(\cF,\cO_X) - \chi(\cF,\cF) \,. 
\label{DObsPT} 
\end{equation}

\begin{Exercise}
Prove that the part of the obstruction in 
\eqref{DObsPT} that corresponds to 
keeping the curve $C$ fixed is given by 
\begin{align}
\Obs \mM &= H^0(C,\cO_D \otimes \det N_{X/C}) \label{ObsPT}  \\
& = H^0(C, \cO_D \otimes \cK_C \, \cL_4^{-1} \, \cL_5^{-1} ) \notag 
\end{align}
and, in particular, is the cotangent bundle of $\mM$ if 
$\cK_X \cong \cO_X$. 
\end{Exercise}

\noindent 
Since $\mM$ is smooth and the obstruction bundle has 
constant rank, we conclude 
$$
\cO^\vir_\mM = \Ld \Obs^\vee\,. 
$$
Define the integers 
$$
h_i = \dim \Hd(C,\cL_i) = \deg \cL_i + 1 - g(C) \,, 
$$
as the numbers from the Riemann-Roch theorem for $\cL_i$. 

\begin{Lemma}\label{l_tOvirSn} We have 
  \begin{multline}
\tO^\vir\Big|_{S^n C} = (-1)^{h_4} z^{\frac{h_4+h_5}{2} + n}  
\left(\det \Hd(C,\cL_4-\cL_5)\right)^{1/2}   \!
\otimes \!\Ld \Obs \, \otimes \, \cL_4^{\boxtimes n} 
\,. 
\label{tOvirSn}
  \end{multline}
where 
$\cL_4^{\boxtimes n}$ is an $S(n)$-invariant line bundle on $C^n$ 
which descends to a line bundle on $S^n C$. 
\end{Lemma}

\begin{proof}
We start the discussion of
$$
\tO^\vir = \textup{prefactor} \,\, \cO^\vir \otimes 
\left(\cK_\vir \otimes \det \Hd(\cO_C(D) \otimes (\cL_4 -
  \cL_5))\right)^{1/2}
$$
with the prefactor. Since 
$$
\dim \chi(\cF) = 1-g(C)+\deg D
$$
formula \eqref{prefactor} specializes to 
\begin{equation}
  \textup{prefactor}  = 
(-1)^{h_4+n} z^{\frac{h_4+h_5}{2} + n} \,. 
\end{equation}
Since $\rk \Obs = n$, we have 
$$
\cO^\vir = (-1)^n \Ld \Obs \otimes \left( \det \Obs\right)^{-1} \,.  
$$
This proves \eqref{tOvirSn} modulo 
$$
\cL_4^{\boxtimes n} \overset{?}= 
 \det \Hd\left(C,\cO_D \otimes \cG \right)^{1/2}
$$
where 
$$
\cG = \cK_C- \cK_C \, \cL_4^{-1} \, 
\cL_5^{-1} + \cL_4 - \cL_5  \, .
$$
We observe
\footnote{It suffices to check this for $\cB_1=\cB_2(p)$ where 
$p\in C$. In this case $(\cB_1 /\cB_2)^{\boxtimes n}$ is the 
line bundles corresponding to the divisor $\{D\owns p\} 
\subset  S^n C$.}
that for any two line bundles $\cB_1,\cB_2$ on $C$ 
$$
\det \Hd\left(C,\cO_D \otimes (\cB_1 - \cB_2) \right) = 
(\cB_1 /\cB_2)^{\boxtimes n} \,,
$$
whence the conclusion. 
\end{proof}

\subsubsection{}

In these notes, we focus on equivariant K-theory, that is,
we compute equivariant Euler characteristics of coherent sheaves. 
This can be already quite challenging, but still much, much easier than 
computing individual cohomology groups of the same sheaves. 

Our next result is a rare exception to this rule. Here we get the 
individual cohomology groups of $\tO^\vir$ on $S^nC$ as 
symmetric powers of $\Hd(C,\dots)$. When computing 
$\bSd \Hd$, one should keep in mind the sign rule --- 
a product of two odd cohomology classes picks up a sign 
when transposed. More generally, 
in a symmetric power of a complex
$$
\cC^\bullet = \dots \xrightarrow{\,\,d\,\,}  \cC^{i} 
 \xrightarrow{\,\,d\,\,}  \cC^{i+1}  \xrightarrow{\,\,d\,\,}  \dots 
$$
the odd terms are antisymmetric with respect to permutations. 

The virtual structure sheaf $\cO^\vir$ 
is defined on the level of derived 
category of coherent sheaves as the complex \eqref{DGA} itself. 
Lemma \ref{l_tOvirSn} shows that up to a shift and 
tensoring with a certain 1-dimensional vector space, the
symmetrized virtual structure sheaf $\tO^\vir$ is 
represented by 
the complex 
\begin{align*}
\bbO^\bullet_n &=  \Ld \Obs \, \otimes  \cL_4^{\boxtimes n}\\
&= \cL_4^{\boxtimes n} \xrightarrow{\,\,0\,\,} 
\cL_4^{\boxtimes n} \otimes \Obs  \xrightarrow{\,\,0\,\,} 
\cL_4^{\boxtimes n} \otimes \Lambda^2 \Obs \xrightarrow{\,\,0\,\,}  
\dots 
\end{align*}
with zero differential. The differential is zero because 
our moduli space is cut out by the zero section of the 
obstruction bundle over $S^n C$. 

In particular, 
$$
\bbO^\bullet_1 = \cL_4 \xrightarrow{\,\,0\,\,}  \cK_C \otimes
\cL_5^{-1} \,, 
$$
and hence by Serre duality 
$$
H^i (\bbO^\bullet_1) = 
\begin{cases}
H^0(\cL_4) \,, & i=0 \,,\\
H^1(\cL_4)\oplus H^1(\cL_5)^\vee \,, & i=1 \,,\\
H^0(\cL_5)^\vee\,, & i=2 \,, \\
0\,, &  \textup{otherwise}\,. 
\end{cases}
$$
In other words
$$
\Hd(\bbO^\bullet_1) = \Hd(\cL_4) \oplus \Hd(\cL_5)^\vee[-2] 
$$
where 
$$
\cC[k]^i = \cC^{k+i} 
$$
denotes the shift of a complex $\cC^\bullet$ by $k$ steps to the left.

\begin{Theorem}\label{t_Sn} 
  \begin{equation}
    \sum_n z^n \Hd\!\left(\bbO^\bullet_n\right) 
=
\bSd z \, \Hd\!\left(\bbO^\bullet_1\right) \,. 
\label{f_t_Sn} 
  \end{equation}
\end{Theorem}

\noindent 
It would be naturally
very interesting to know to what extent 
our other formulas can be upgraded to the level of the derived 
category of coherent sheaves.  

\subsubsection{}

Recall the definition of the symmetrized symmetric 
algebra from Section \ref{s_ShV} . Working again in 
K-theory, we have 
$$
(-1)^{h_4} z^{\frac{h_4}2} 
\left(\det \Hd(\cL_4)\right)^{1/2} \, \bSd z \Hd(C,\cL_4) = 
\bSdh \Hd(C, z \cL_4)^\vee \,.
$$
With this notation, Theorem \ref{t_Sn} gives the following 

\begin{Corollary}
  \begin{equation}
    \chi(\mM,\tO^\vir) = \bSdh \Hd(C, z \cL_4 \oplus z^{-1}
    \cL_5)^\vee \,. 
  \end{equation}
\end{Corollary}

\noindent 
Since deformations of the curve $C$ inside $X$ contribute 
$\bSdh \Hd(C,N_{X/C})^\vee$, we see that, indeed, all directions \eqref{NZC}
normal to $C$ in $Z$ contribute equally to the K-theoretic 
PT count. 

\subsubsection{}
As an example, let us take $C=\C^1$ and work equivariantly. We have
$S^n C \cong \C^n$, so there is no higher cohomology anywhere. 

\begin{Exercise}\label{e_qbinom}
Show Theorem \ref{t_Sn} for $C=\C^1$ is equivalent to 
the following identity, known as the $q$-binomial theorem 
\begin{equation}
\label{qbinom} 
\sum_{n \ge 0} z^n \prod_{i=1}^n \frac{1-m \, t^i}{1-t^i} = 
\bSd z \, \frac{1-m\, t}{1-t} = 
\prod_{k=0}^\infty \frac{1-z \, m \, t^{k+1}}{1-z \, t^k}  \,.
\end{equation}
\end{Exercise}

\begin{Exercise}
Prove \eqref{qbinom} by proving a 1st order difference equation 
with respect to $z\mapsto tz$ for both sides. This is a baby 
version of some quite a bit more involved difference equations to come. 
\end{Exercise}

\subsection{Proof of Theorem \ref{t_Sn}}

\subsubsection{}

As a warm-up, let us start with the case 
$$
\cL_4 = \cL_5 = \cO, 
$$
in which case the theorem reduces to a classical formula, going back
to Macdonald, for $\Hd(\Omd_{S^n C})$ and, in particular, 
for Hodge numbers of $S^n C$. It says that 
\begin{equation}
 \sum_n z^n \Hd \left(S^n C, \Omd_{S^n C} \right) 
= \bSd z \, \Hd(C, \Omd_{C}) \,, \label{e_Macd}
\end{equation}
and this equality is canonical, in particular gives an isomorphism of 
orbifold vector bundles over moduli of $C$. 

Here $\Omd$ is not the de Rham complex, but rather the complex 
$$
\Omd = \cO \xrightarrow{\,\,0\,\,}  \Omega^1 
\xrightarrow{\,\,0\,\,}  \Omega^2 \xrightarrow{\,\,0\,\,}   \dots 
$$
with zero differential, as above. 

\subsubsection{}

Let $M$ be a manifold of some dimension and consider the 
orbifold 
$$
e^M = \bigsqcup_{n\ge 0}\,  S^n M   \,. 
$$
It has a natural sheaf of orbifold differential forms $\Omd_\orb$, 
which are the differential forms on $M^n$ invariant under the action of
$S(n)$. 

By definition, this means that 
\begin{equation}
\Omd_\orb = \pi_{*,\orb} \big(\Omd\big)^{\boxtimes n} \label{pi_orb}
\end{equation}
where 
$$
\pi: M^n \to S^n M 
$$
is the natural projection and $\pi_{*,\orb}$ is the usual direct image 
of an $S(n)$-equivariant coherent 
sheaf followed by taking the $S(n)$-invariants. 
Since the map $\pi$ is finite, there are no higher 
direct images and from the triangle 
$$
\xymatrix{
M^n \ar[rr]^\pi \ar[rd]_p && S^n M \ar[ld] \\
& \textup{point} 
}
$$
we conclude that 
$$
\Hd(S^n M, \Omd_\orb) = p_{*,\orb}  \big(\Omd\big)^{\boxtimes n}  =
S^n \Hd(M,\Omd) \,.
$$
If $\dim M = 1$ then $\Omd_\orb= \Omd_{S^n M}$ and we 
obtain \eqref{e_Macd}.

% Differential forms on $M^n$ can be decomposed according to their 
% type, that is, according to how many $dx_i$ they use from each 
% factor. For example, an $S(n)$-invariant $1$-form is necessarily of 
% the form 
% $$
% \alpha = \sum f_i(m_1,m_2,\dots,m_n) \, dx_i^{(1)} + \dots 
% $$
% where $dx_i^{(1)}$ is a coordinate $1$-form on the first factor, $f_i$ 
% is symmetric in the variables $m_2,\dots,m_n$, and the dots stand
% for terms obtained by symmetrization with respect to 
% $S(n)/S(n-1)$. This means that 
% $$
% \Omega^1_{S^n M,\orb} = \pi_{1,n-1,*} 
% \left(\Omega^1_M \boxtimes 
% \bS^{n-1} \Omega^0_M\right)
% $$
% where 
% $$
% \pi_{1,n-1} :  M \times S^{n-1} M \to M 
% $$
% is the symmetrization map. 

% Similarly, we have 
% %
% \begin{equation}
% \Omega^2_{S^n M,\orb} = \pi_{1,n-1,*} 
% \left(\Omega^2_M \boxtimes 
% \bS^{n-1} \Omega^0_M\right) + 
% \pi_{2,n-2,*} 
% \left(\bS^2 \Omega^1_M \boxtimes 
% \bS^{n-2} \Omega^0_M\right) 
% \label{Om2Sn} 
% \end{equation}
% %
% and so on, where
% $$
% \pi_{2,n-2}: S^2 M \times S^{n-2} M \to S^n M 
% $$
% and, by the sign rule, $\bS^2 \Omega^1_M$ is 
% really the second exterior power, spanned by form 
% $$
% f_{ij}(m_1,m_2,\dots,m_n) \, dx_i^{(1)} \wedge dx_j^{(2)} 
% $$
% in which the coefficient $f_{ij}$ is antisymmetric in $m_1$ and
% $m_2$. 

% The maps $\pi$ are finite, hence there are no higher direct 
% images and we obtain 
% $$
% \Hd(e^M,\Omd_\orb) = \bSd \Hd(M,\Omd_M) \,. 
% $$

\subsubsection{}
Now let $\cE$ be an arbitrary line bundle on a curve $C$ and
define rank $n$ vector bundles $\cE_n$ on $S^n C$ by 
$$
\cE_n = \Hd (\cO_D \otimes \cE) \,. 
$$
This is precisely our obstruction bundle, with the substitution 
$$
\cE= \cK_C \, \cL_4^{-1} \, \cL_5^{-1}\,.
$$
As before, we form a complex $\Ld \, \cE_n$ with zero differential 
and claim that 
\begin{equation}
\Ld \, \cE_n = \pi_{*,\orb}  \left(\Ld \, \cE_1\right)^{\boxtimes n} 
\,, \label{LamcEn} 
\end{equation}
in similarity to \eqref{pi_orb}. In fact, locally on $C$ there is
no difference between $\cE$ and $\cK_C$, so these are really 
the same statements. 
See for example \cite{EinLaz,Vois} for places in the literature 
where a much more powerful 
calculus of this kind is explained and used. 

By the projection formula
\begin{align*}
\Hd \!\left(S^n C, \Ld\, \cE_n  \otimes \cL_4^{\boxtimes n}\right) &= 
\Hd \!\left(C^n, 
\left(\Ld \, \cE_1\right)^{\boxtimes n} \otimes \cL_4^{\boxtimes n}\right)^{S(n)}  \\
&= 
\bS^n \Hd \!\left(C,\bbO^\bullet_1\right) \,, 
\end{align*}
as was to be shown.

\subsection{Hilbert schemes of surfaces and threefolds} 

\subsubsection{}

Now let $Y$ be a nonsingular surface. In this case $S^n Y$ is singular 
and 
$$
\pi_{\Hilb} : \Hilb(Y,n) \to S^n Y 
$$
is a resolution of singularities. The sheaves $\pi_{\Hilb,*} \Omd$ and 
$\Omd_\orb$ on $Y$ are not equal, but they share one important 
property known as \emph{factorization}. It is a very important 
property, much discussed in the literature,
with slightly different definitions in different contexts, see 
for example \cite{BFS,Gaits}.  Here we 
will need only a very weak version of factorization, which may 
be described as follows. 

A point in $S^n Y$ is an unordered $n$-tuple $\{y_1,\dots,y_n\}$ 
of points from $Y$. Imagine we partition the points $\{y_i\}$ into 
groups and let $m_k$ be the number of groups of size $k$. 
In other words, consider the natural 
map 
$$
\prod_k S^{m_k} S^k Y \xrightarrow{\quad f\quad}  S^n Y \,, 
\quad n = \sum m_k k \,. 
$$
Let $U$ be the open set of in the domain of $f$ formed by 
$$
y_i \ne y_j
$$
for all $y_i$ and $y_j$ which belong to \emph{different} groups. 

Fix a sheaf, of a K-theory class $\cF_n$ on each $S^n Y$. A
factorization of this family of sheaves is a collection of
isomorphisms
\begin{equation}
\cF_n \big|_U  \cong \bigboxtimes \, \bS^{m_k} \cF_k 
\label{factoriz}
\end{equation}
for all $U$ as above. Here $\bS^{m_k} \cF_k$ is the orbifold 
pushforward of $\left(\cF_k \right)^{\boxtimes m_k}$ and the
isomorphisms \eqref{factoriz} must be compatible 
with subdivision into smaller groups. 

For example, $\Omd_\orb$ has a factorization by construction and 
it is easy to see $\pi_{\Hilb,*} \Omd$ similarly factors. 

\subsubsection{}
The following lemma is a geometric version of the well-known 
combinatorial principle 
of inclusion-exclusion. 

\begin{Lemma}\label{l_factor}
For any scheme $Y$ and any 
factorizable sequence $\cF_n\in K_G(S^n Y)$ there exists 
\begin{equation}
\cG = z \, \cG_1 + z^2 \, \cG_2 + \dots \in 
K_G(Y)[[z]] \label{cGz} 
\end{equation}
such that 
\begin{equation}
1+\sum_{n>0} z^n\, \chi(\cF_n) = \bSd \chi(\cG) \,. 
\label{bsdcg} 
\end{equation}
\end{Lemma}

\noindent
Concretely, formula \eqref{bsdcg} means that 
\begin{equation}
\chi(\cF_n) = \sum_{\sum k \, m_k = n} \bigotimes \bS^{m_k} \chi(\cG_k) \,,
\label{cFpart} 
\end{equation}
where the summations here over all solutions $(m_1,m_2,\dots)$ 
of the equation $\sum k \,m_k = n$ or, equivalently, over 
all partitions 
$$
\mu = (\dots 3^{m_3}\, 2^{m_2} \, 1^{m_1}) 
$$
of the number $n$. For example
\begin{align*}
\chi(\cF_2)&= \bS^2 \chi(\cG_1) + \chi(\cG_2) \,, \\
\chi(\cF_3)&= \bS^3 \chi(\cG_1) + \chi(\cG_2)\chi(\cG_1)
 + \chi(\cG_3)
             \,. 
\end{align*}

\subsubsection{}

\begin{proof}[Proof of Lemma \ref{l_factor}]
The sheaves $\cG_i$ in \eqref{cGz} are constructed inductively, 
starting with 
$$
\cG_1 = \cF_1  \,. 
$$
and using the exact sequence \eqref{exK}. 
Consider $X=S^2Y$ and let 
$$
Y \cong X' \subset X 
$$
be the diagonal. Factorization gives 
$$
\cF_2 \big|_{U} \cong \bS^2 \cG_1 \,, \quad  U = X\setminus X'\,, 
$$
and so from \eqref{exK} we obtain
$$
\cG_2 = \cF_2 - \bS^2 \cG_1  \in K(X')=K(Y) 
$$
which solves \eqref{cGz} modulo $O(z^3)$. 

Now take $X=S^3 Y$ and let $X'=p(Y^2)$ where 
$$
p(y_1,y_2)=2 y_1+ y_2 \in X \,.
$$
Consider 
$$
\cF_3' = \cF_3 - \bS^3 \cG_1 \in K(X') \,, 
$$
and denote 
$$
X''= \{y_1=y_2=y_3\} = p( \textup{diagonal}_{Y^2}) \cong Y\,. 
$$
By compatibility of factorization with respect to further 
refinements 
$$
\cF_3' \big|_{X'\setminus X''} = p_*(\cG_2 \boxtimes \cG_1)  \,.
$$
Using the exact sequence \eqref{exK} again, we construct 
$$
\cG_3 = \cF''_3 = \cF'_3 - p_*(\cG_2 \boxtimes \cG_1) \in K(Y)
$$
which solves \eqref{cGz} modulo $O(z^4)$. 

For general $n$,  we consider closed subvarieties 
$$
S^n Y = X_n \supset X_{n-1} \supset \dots \supset X_1 = Y 
$$
where $X_k$ is the locus of $n$-tuples $\{y_1,\dots,y_n\}$ 
among which at most $k$ are distinct. In formula \eqref{cFpart}, 
$X_k$ will correspond to partitions $\mu$ with 
$$
\ell(\mu)=\textup{length}(\mu) = \sum m_i = k \,. 
$$
We construct 
$$
\cF_n' \in K(X_{n-1})\,, \quad \cF_n'' \in K(X_{n-2})\,, \quad \dots 
$$
inductively, starting with $\cF_n$ on $X_n$. 
For each $k<n-1$, the set 
$$
U_{n-k} = X_{n-k} \setminus X_{n-k-1}
$$
is a union of sets to which factorization applies, and this gives
\begin{equation}
\cF^{(k)} \Big|_{X_{n-k} \setminus X_{n-k-1}} = 
\sum_{\ell(\mu) =n-k} p_{\mu,*} \left(
\bigboxtimes \bS^{m_k} \cG_k \right)   \Big|_{X_{n-k} \setminus
X_{n-k-1}}
\label{cFk} 
\end{equation}
where
$$
p_{\mu} (y_1,y_2,\dots,y_\ell) =  \sum \mu_i \, y_i \in S^n Y  \,.
$$
We let $\cF^{(k+1)}$ be the difference between two sheaves in
\eqref{cFk}, which is thus a sheaf supported on $X_{n-k-1}$. 
Once we get to $X_1 \cong Y$, this gives 
$$
\cG_n = \cF_n^{(n-1)} \,. 
$$
\end{proof}

\subsubsection{}\label{s_Z_surface} 
Now we go back to $Y$ being a 
nonsingular suface.  Recall that the Hilbert scheme of 
points in $Y$ is nonsingular and that Proposition \ref{p_TanH} 
expresses its tangent bundle in terms of the 
universal ideal sheaf. 

For symmetric powers of the curves in 
Section \ref{sub_PT1},  the obstruction bundle was a certain 
twisted version of the cotangent bundle. One can 
similarly twist the tangent bundle of the Hilbert scheme 
of a surface, namely we define 
\begin{align}
  T_{\Hilb,\cL} = \chi(\cL) - \chi(\cI_Z,\cI_Z\otimes \cL)  \label{TchiIL} \,. 
\end{align}
for a line bundle $\cL$ on $Y$. 
Let $\Omd_{\Hilb,\cL}$ be the exterior 
algebra of the dual vector bundle.  It is clear that its pushforward to 
$S^n Y$ factors just like the pushforward of $\Omd_{\Hilb}$ and 
therefore 
\begin{equation}
\sum_n z^n \chi(\Hilb(Y,n),\Omd_\cL) = \bSd \chi(Y,\cG)
\label{ZsurfS} 
\end{equation}
for a certain $\cG$ as in \eqref{cGz}.  The analog of Nekrasov's 
formula in this case is the following 

\begin{Theorem}\label{t_Z_surface} 
  \begin{equation}
\sum_n z^n \chi(\Hilb(Y,n),\Omd_\cL)  = 
\bSd \chi\left(Y, \Omd_{\cL}\, \, \frac{z}{1-z \cL^{-1}} \right) \,. 
\label{f_Z_surface} 
\end{equation}
\end{Theorem}

\noindent 
See \cite{CNO} for how to place this formula in a much more 
general mathematical and physical context. As with Nekrasov's 
formula, it is in fact enough to prove \eqref{f_Z_surface} 
for a toric surface, and hence for $Y=\C^2$, in which case it 
becomes a corollary of the main result of \cite{CNO}.  

Here we discuss an alternative approach, based on \eqref{ZsurfS}, 
which we format as a sequence of exercises. 

\begin{Exercise}
Check that for $Y=\C^2$, the LHS in \eqref{f_Z_surface} 
becomes the function 
$$
\bZ_{\Hilb(\C^2)} = \sum_{n,i\ge 0} z^n (-m)^i \, \chi(\Hilb(\C^2,n), 
\Omega^i)
$$
investigated in Exercise \ref{ex_Z_surface},
where $\cL^{-1}$ is a trivial bundle with weight $m$. 
\end{Exercise}

\begin{Exercise}
Arguing as in Section \ref{s_strr},  prove that 
$$
\bZ_{\Hilb(\C^2)}  =  \bSd \, \left(\strr \, 
\frac{(1-m \, t_1^{-1}) (1-m \, t_2^{-1})} {(1- t_1^{-1}) (1- t_2^{-1})} 
\right)
$$
for a certain series 
$$
\strr \in \Z[m][[z]] \,. 
$$
\end{Exercise}

\begin{Exercise}
Arguing as in Section \ref{s_strr_next},  prove that 
$$
\strr = \frac{z}{1-mz} \,. 
$$
What is the best limit to consider for the parameters $t_1$ and $t_2$
? 
\end{Exercise}

\subsubsection{} \label{s_tOvir_factors} 
Now let $X$ be a nonsingular threefold and let
$$
\pi:  \Hilb(X,n)\to S^n
$$
be the Hilbert-Chow map. 
To complete the proof of 
Nekrasov's formula given in Section \ref{s_proofN}, we 
need to show \eqref{bZSR}, which follows from the following 

\begin{Proposition} \label{P_tOvir_factors} 
The sequence 
$$
\cF_n = \pi_* \, \tO^\vir \in K(S^nX) 
$$
factors. 
\end{Proposition}

\noindent  
There is, clearly, something to check here, because, 
for example, this sequence would not factor without the 
minus sign in \eqref{pref_points}. 

\begin{proof}
Recall from Section \ref{S_Hilb3} that 
$$
\tO^\vir = \dots \xrightarrow{\,\, d\,\,} 
\bk^{-\frac{\dim}{2} +i} \, \Omega^i_{\mMh} 
\xrightarrow{\,\, d\,\,} \, \bk^{-\frac{\dim}{2} +i+1}
\Omega^{i+1}_{\mMh} \xrightarrow{\,\, d\,\,}  
$$
where $i$ is also the cohomological dimension and 
$$
d \omega = \bk \, d\phi \wedge \omega \,. 
$$
Here 
$$
\phi(\bX) = \tr \left( \bX_1 \bX_2 \bX_3 - 
\bX_1 \bX_3 \bX_2 \right)
$$
is the function whose critical locus in $\mMh$ is the Hilbert scheme. 

Let $U$ be the locus where 
the spectrum of $\bX_1$ can be decomposed into 
two mutually disjoint blocks of sizes $n'$ and $n''$, respectively. 
This means $\bX_1$ can be put it in the form 
$$
\bX_1 = 
\begin{pmatrix}
  \bX_1'  & 0  \\
0 & \bX_1'' 
\end{pmatrix}
$$
up-to conjugation by 
$GL(n') \times GL(n'')$ or 
$S(2) \ltimes GL(n')^2$ if $n'=n''$. Thus block off-diagonal elements 
of $\bX_1$ and of the gauge group are eliminated simultaneously. 

If $\lambda_i'$ and $\lambda''_j$ 
are the eigenvalues of $\bX_1'$ and $\bX_1''$ respectively, then 
as a function of the off-diagonal elements of $\bX_2$ and $\bX_3$ 
the function $\phi$ can be brought to the form 
$$
\phi = \sum_{ij} (\lambda'_i - \lambda''_j) \, \bX_{2,ij} \bX_{3,ji}
+ 
\dots \,,
$$
and thus has many Morse terms of the form 
$$
\phi_2(u,v) = u v
$$
where the weights of $u$ and $v$ multiply to $\bk^{-1}$. 
For the Morse critical point on $\C^2$, the complex 
$$
\bk^{-1} \cO \xrightarrow{\,\, \bk \, d\phi_2 \wedge \,\,}  
\Omega^1 \xrightarrow{\,\, \bk \, d\phi_2\wedge\,\,}  \bk \, \Omega^2
$$
is exact, except at the last term, where its cohomology is
$\cO_{u=v=0}$. This is also true if we replace $\C^2$ by 
a vector bundle over some base because the existence of a Morse 
function forces the determinant of this bundle to be 
trivial, up to a twist 
by $\bk$. 

Therefore, all off-diagonal matrix elements of $\bX_2$ and
$\bX_3$ are eliminated and $\tO^\vir$, restricted to $U$, 
is, up to an even shift,
 the tensor product of the corresponding complexes for  
$\Hilb(X,n')$ and $\Hilb(X,n'')$.

\end{proof}

\section{More on quasimaps}\label{s_more} 

\subsection{Balanced classes and square roots}

\subsubsection{}

Let $\cK_C$ be the canonical bundle of the domain $C$. For our 
specific domain $C\cong\bP^1$, we have, equivariantly, 
$$
\cK_C - \cO_C = -\cO_{p_1} - \cO_{p_2} 
$$
where $p_1,p_2 \in C$ are the fixed points of a torus in $\Aut(C)$. 
It is customary to choose $\{p_1,p_2\} = \{0,\infty\}$.

\subsubsection{}\label{s_polar} 

The choice of a Lagrangian subspace $M$, namely \eqref{repM}, inside
the symplectic representation $T^*M$ of $G$ determines
a \emph{polarization} 
\begin{equation}
  \label{T12}
  T^{1/2} X = M - \sum_i \End(V_i) 
\end{equation}
which, by definition, is an equivariant K-theory class such that
$$
T X = T^{1/2} X + \hbar^{-1} \, \left(T^{1/2} X\right)^\vee
$$
in $K(X)$.  The chosen polarization induces a virtual bundle  
$$
\cT^{1/2} = \cM - \sum_i \cHom(\cV_i,\cV_i) 
$$
over $C\times \QM(X)$. 

\begin{Lemma}\label{lsquare} We have
  \begin{equation}
 T_\vir = \sum_{i=1}^2 \cT^{1/2} \big|_{p_i}+
H- \hbar^{-1} H^\vee \,, \label{TvirP}
\end{equation}
where $H=\Hd(\cT^{1/2} \otimes \cK_C)$, equivariantly with 
respect all automorphisms of $C$ and $X$ that preserve 
$p_1$, $p_2$, and the symplectic form. 
\end{Lemma}
\begin{proof}
  Serre duality gives
$$
\Hd(\cM \otimes  \cK_C)^* = - \Hd(\cM^*) \,,
$$
whence the conclusion. 
\end{proof}

Naturally, a formula similar to \eqref{TvirP} 
 exists for an arbitrary curve $C$ with 
the canonical divisor replacing $-p_1-p_2$. 

\subsubsection{}
\begin{Definition}\label{def_balanced}
  We say that a K-theory class $\cF$ is \emph{balanced} if 
$$
\cF =\cG - \hbar^{-1} \cG^\vee  
$$
for some K-theory class $\cG$. 
\end{Definition}

\noindent 
Lemma \ref{lsquare} may be rephrased to say that $T_\vir$ 
equals the polarization at the marked points modulo balanced 
classes.

\subsubsection{}

Lemma \ref{lsquare} implies 
$$
\frac{\det T_\vir}{ \det \cT^{1/2} \big|_{p_1}  \det \cT^{1/2} \big|_{p_2}}  = 
\hbar^{\rk H} (\det \Hd(\cT^{1/2} \otimes \cK_C) )^2 \,,
$$ 
which is a square if we replace $\Ct_\hbar$ by its double cover. We
define 
\begin{equation}
\tO_\vir = \cO_\vir \otimes \left( \cK_\vir \, 
\frac{\det \cT^{1/2} \big|_{p_2} }{ \det \cT^{1/2} \big|_{p_1}  }
\right)^{1/2} \,.\label{def_tO}
\end{equation}
where $\cK_\vir = \det^{-1} T_\vir$. 

\subsubsection{}

{}From \ref{s_aroof}, recall the function 
$$
\aroof(x) = \frac{1}{x^{1/2} - x^{-1/2}}\,,
$$
which we extended to $\cG\in K(X)$ by the rule
$$
\aroof(\cG) = \prod_{\textup{Chern roots $x_i$ of $\cG$}} \aroof(x_i) \,,
$$
assuming a square root of $\det\cG$ exists and has been fixed. 
As we already saw e.g.\ in \eqref{tOvirHilbloc}, the contribution of
$\tO_\vir$ to localization formulas is, up to small details, 
$\aroof(T_\vir)$. 

The key technical point about Lemma \ref{lsquare} is that for any 
balanced $K$-theory class we have 
$$
\aroof\left(H-\hbar^{-1} H^\vee\right) = ( - \hbar^{-1/2})^{\rk H} 
\prod_{x_i} \frac{1-\hbar x_i}{1-x_i} \,,
$$
where $x_i$ are the Chern roots of $H$. This rational function 
remains \emph{bounded} as $x_i^{\pm} \to \infty$, which will be 
the essential step of several rigidity arguments below.  

\subsubsection{}

This technical point is the reason we work with $\tO_\vir$ and not 
$\cO_\vir$. It is also the reason we work with quasimaps and not with 
other moduli spaces of rational curves in $X$.

For $\cO_\vir$, without 
the square root of the virtual canonical, the rigidity arguments used
e.g.\ in the proof of Nekrasov's formula breaks down, and it is a
challenge to derive a reasonable formula for the series 
\eqref{_Zovir}. In fact, even the analog of Exercise
\ref{e_qbinom} for $\cO_\vir$ in place of $\tO_\vir$ appears
problematic. 

\subsubsection{}
In particular, it is  important in Lemma \ref{lsquare} that the
duality holds not just with respect to the automorphisms 
of the target $X$ of the quasimaps, but also with respect to the 
automorphisms of the 
domain $C$ that preserve the two marked points. 
 This is because later, when 
we talk about relative quasimaps, we will need to 
take quotients with respect to $\Aut(C,p_1,p_2)$. For duality 
to descend to the quotient, it needs to hold
$\Aut(C,p_1,p_2)$-equivariantly for ordinary quasimaps.

\subsubsection{}\label{s_M0n}

Now suppose that instead of quasimaps to a Nakajima variety $X$ we
take the moduli space $\Mbar_{0,2}(X)$ of stable  
maps 
$$
f: C \to X \,,
$$
where $C$ is a 2-pointed rational curve. 
The essential component of its deformation 
theory is again $\Hd(C,f^* TX)$ and following the argument of 
Lemma \ref{lsquare} we can write 
$$
\Hd(C,f^* TX)  = \textup{like in \eqref{TvirP}} + \Delta
$$
where 
$$
\Delta=\Hd(C,f^*\cT^{1/2}\otimes (\cO(-p_1-p_2)-\cK_C)) \,.
$$
This discrepancy class $\Delta$ is nontrivial K-theory class on 
$\Mbar_{0,2}(X)$ supported on the divisor $D$ where $\cK_C \not\cong
\cO(-p_1-p_2)$. This divisor is formed by curves of the form 
$$
C = C' \cup C''
$$
where $C'$ is the minimal chain of rational curves containing $p_1$ and
$p_2$ and $C''\ne \varnothing$. For example, we have the following 

\begin{Lemma}\label{lsquare2} The multiplicity of $\det \Delta$ along 
a component of $D$ equals minus the 
degree of the polarization $\cT^{1/2}$ 
on the component $C''$. 
\end{Lemma} 

\begin{proof}
We may compute this multiplicity on the stack of 2-pointed rational 
curves with a vector bundle $V$. The vector bundle $V$ will 
generalize the pull-back of
polarization $\cT^{1/2}$ to $C$. 

To write a curve in this stack that
intersects $D$ transversally, we take a trivial 
family $C \times B$ with base $B$ and blow up a point in the fiber $C \times \{b\}$ 
away from $p_1$ and $p_2$, see Figure \ref{f_component_no_p}. 
\begin{figure}[!htbp]
  \centering
   \includegraphics[scale=0.64]{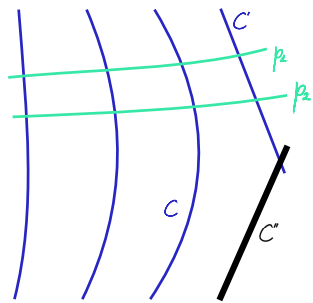}
 \caption{A curve $C$ sprouts off a tail $C''$.}
  \label{f_component_no_p}
\end{figure}
We denote by 
$$
\pi: \bfC \to B 
$$
the corresponding family of $2$-pointed rational curves. The dualizing 
sheaves of the fibers are obtained by restricting 
$$
\cK_{\bfC} \otimes \pi^* \cK_B^{-1} \cong \cO_{\bfC}(-p_1-p_2+C'')
$$
to the fibers, where $p_i$ denote the sections of $\pi$ given by
marked points and $C''\subset \bfC$ is the exceptional divisor.  We have 
$$
\cO_{\bfC}(-p_1-p_2) - \cO_{\bfC}(-p_1-p_2+C'')=-\cO_{C''}(-1) 
$$
and 
$$
\pi_* \left(V \otimes \cO_{C''}(-1)\right) = \left(\deg V\big|_{C''} \right) \cdot 
\cO_b
$$
by Riemann-Roch. This equals the degree of the bundle 
$\det \pi_* \left(V \otimes \cO_{C''}(-1)\right)$ on $B$, and hence
minus the multiplicity of the corresponding component. 
\end{proof}

\begin{Exercise}
What will happen in the above computation if we blow up one the 
points $p_1$ or $p_2$ ? 
\end{Exercise}

While $\det \cT^{1/2}$ may be a square in $\Pic(X)$ which will make
$\det \Delta$ also a square by Lemma \ref{lsquare2}, there is nothing 
selfdual about the class $\Delta$ which would allow the rigidity
arguments to go through. A further modification of the enumerative 
problem is required for that.

\subsection{Relative quasimaps in an example}\label{s_qm_example}

\subsubsection{}

The evaluation at $p\in C$ in \eqref{evp} is a rational map
$$
\ev_p: \QM
\dashrightarrow X 
$$
defined on the open set $\QM_\textup{nonsing $p$}$ of quasimaps nonsingular at $p$. 
Moduli spaces of relative quasimaps is a resolution of this map, 
that is,  they fit into a diagram 
\begin{equation}
  \label{eq9}
  \xymatrix{
& \QM_\textup{relative $p$}\ar[dr]^{\ev}\\
\QM_\textup{nonsing $p$} \ar[rr]\ar@{^{(}->}[ru]&& X 
}
\end{equation}
with a \emph{proper} evaluation map to $X$.  Their
construction follows an established path in Gromov-Witten 
theory \cite{Li1,Li2} and, later, DT theory \cite{LiWu}. 

\subsubsection{}

The basic idea behind relative quasimaps may be explained in the 
most basic example of quasimaps to 
$$
X = \C/ \Ct \,.
$$
{}From definitions
$$
\QM(C \to \C/ \Ct) = \{(\cL,s)\}\,,
$$
where $\cL$ is a line bundle on $C$ and 
$$
s: \cO_C \to \cL 
$$
is a section which is not identically zero. This gives 
$$
\cL = \cO(D) 
$$
where $D=(s) \in S^d C$ is the divisor of $s$ and $d=\deg \cL$. Thus
$$
\QM(C \to \C/ \Ct) = \bigsqcup_{d\ge 0} S^d C \,.
$$

\subsubsection{}
The set $\QM_\textup{nonsing $p$}$ is formed by divisors $D$ disjoint from
$p\in C$. Relative quasimaps compactify this set by allowing the curve
$C$ to break when the support of $D$ approaches $p$. In the language 
of algebraic geometry this means the following. 

Let $D_t$ be a flat $1$-parameter family of divisors parametrized by 
$t\in B\cong \C$, such that 
$$
p \in D_0\,, \quad p\notin D_t, \quad \textup{for $t$ generic} \,.
$$
We may assume that $p=0 \in C\cong \C$, in which 
case 
$$
D_t = \{f(x,t)=0\}
$$
with 
\begin{equation}
f(0,0)=0\,, \quad f(0,t) \ne 0 \,, \quad f(x,0)\ne 0
 \label{cond_fxt} \,. 
\end{equation}
We will now replace the surface $\bfC=B \times C$ by a blow-up 
$$
\bfC^\diamond \to \bfC 
$$
such that 
the proper transform of $D_t$ is disjoint from the proper transform 
of $\{x=0\}$ and from the nodes of the exceptional divisor. 

\subsubsection{}\label{s_toric_blowup}
Let 
$$
f=\sum_{n,m} f_{n,m} x^n t^m 
$$
be the (finite) Taylor expansion of $f$ at $(x,t)=0$. 
Consider the Newton diagram 
$$
\textup{Newton}(f) = \textup{Conv}\left(\{(n,m),f_{n,m} \ne 0\} \cup 
\{(\infty,0),(0,\infty)\}\right)
$$
of the polynomial $f$, see Figure \ref{f_Newton_diagram}.
\begin{figure}[!htbp]
  \centering
   \includegraphics[scale=0.5]{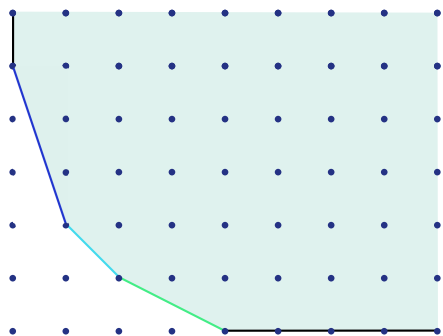}
 \caption{The Newton diagram of a polynomial of the form $
x^5+\star \, 
t x^2+\star\,  t^2 x+\star\, t^4+\dots$, where dots stand for
monomials in the shaded area and stars for nonzero numbers. The
bounded edges have (minus) slopes $\{\frac12,1,3\}$.}
  \label{f_Newton_diagram}
\end{figure}
Let $I$ be the monomial 
ideal 
$$
I = (x^a t^b)_{(a,b)\in \textup{Newton}(f)}\,,
$$
which by \eqref{cond_fxt} is nontrivial zero-dimensional ideal. 
As $\bfC^\diamond$, we take the blowup 
$$
\bfC^\diamond = \Bl_I \bfC  
$$
of $\bfC$ in the ideal $I$, that is, the closure of the graph 
of the map 
$$
\bfC \dashrightarrow \bP^{\textup{\# of generators}-1}
$$
given by the generators of
$I$. 
This is a toric surface whose toric diagram is the Newton diagram. 
The  bounded edges in this 
diagram correspond to components of the exceptional divisor. 

\subsubsection{}
It is easy to see that, by construction, that the 
proper transform of $D_t$ in $\bfC^\diamond$ satisfies the following 
properties: 
\begin{itemize}
\item[---] it is disjoint from the proper transform 
of $\{x=0\}$, 
\item[---] it intersects every component of the exceptional divisor, 
\item[---] it is disjoint from the nodes of the exceptional
  divisor. 
\end{itemize}

\begin{Exercise}
Check this. 
\end{Exercise}

\noindent 
The only 
shortcoming of $\bfC^\diamond$ is that fiber $C^\diamond$ over $0$ of the induced map 
$$
\xymatrix{
C^\diamond \ar[r] \ar[dd] & \bfC^\diamond \ar[dd] \ar[dr] \\
&& \bfC \ar[dl]^t \\
0 \ar[r] & B 
}
$$
may have multiple components and hence cannot serve as 
a degeneration of $C$.  Indeed, we have the following 

\begin{Exercise}
If $I=(x^a,t^b)$
then the multiplicity of the exceptional divisor in $C^\diamond$ is
$a/\gcd(a,b)$. 
More generally, if an compact edge in Figure \ref{f_Newton_diagram} has 
slope $a/b$ with $\gcd(a,b)=1$ then the corresponding 
component of the exceptional divisor has multiplicity $a$ in $C^\diamond$. 
\end{Exercise}

This is remedied by a passing to a degree $m$ branched cover, 
also known as \emph{base change}, 
$$
\xymatrix{
\bfC'\ar[r] \ar[d] & \bfC^\diamond \ar[d] \\
B' \ar[r] & B 
}
$$
where $t=(t')^m$ and 
$$
m = \textup{lcm}\,  \left\{ \frac{a_i}{\gcd(a_i,b_i)} \right\}_{\textup{slopes
    $\frac{a_i}{b_i}$}} \,.
$$
The central fiber $C'$ in the resulting family 
$$
\xymatrix{
C'\ar[r] \ar[d] & \bfC' \ar[d] \\
0 \ar[r] & B'
}
$$
is the curve $C$ with a chain of rational curves (the exceptional
divisor) attached, all with multiplicity one. 

\subsubsection{}

We denote by $p'\in C'$ the intersection of $C'$ with the proper 
transform of $\{x=0\}$ and by $D'_0$ the intersection of $C'$ 
with the proper transform of $D_t$. 

\begin{Exercise} 
Show $D'_0$ is disjoint from $p'$ and the nodes of $C'$. 
What is the degree of this divisor on each component of $C'$ ? 
\end{Exercise} 

\subsubsection{}\label{s_points_approach}
In simple English, the blowup $\bfC^\diamond$ and its
branched cover $\bfC'$ serve the following purpose. 

One can 
write the solutions $x_i(t)$ of the equation $f(x,t)=0$ as Puiseux series
\begin{equation}
x_i(t) = c_i t^{r_i} + o(t^{r_i})\,, \quad r_i = \frac{b_i}{a_i} \,,
\quad c_i \ne 0 \,,  
\label{x(t)} 
\end{equation}
where $i=1,\dots,\deg_x f$. We can assume them ordered 
in the decreasing order of the rates $r_i$ at which they approach 
$0$ as $t\to 0$. 

These rates are the reciprocals of the slopes in Figure \ref{f_Newton_diagram} and 
the solutions $x_i(t)$ that go to zero at the same rate will end 
up on the same component of the exceptional divisor, see
Figure \ref{f_roots}\,. 
\begin{figure}[!htbp]
  \centering
   \includegraphics[scale=0.3]{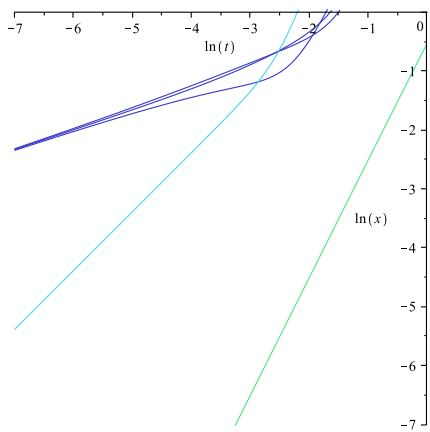}
 \caption{Logarithmic plot of $|x_i(t)|$ for a polynomial 
from Figure \ref{f_Newton_diagram} as $(x,t)\to (0,0)$. The slopes in this figure are
dual, or perpendicular, to the slopes in Figure
\ref{f_Newton_diagram}}
  \label{f_roots}
\end{figure}
Suppose, for example, that 
$$
r_1 = r_2 = r_3 > r_4 \ge \dots \ge 0 
$$
Then the component $C'_{p'}$ of $C'$ that contains $p'$ will contain 
$3$ points of $D'_0$, with coordinates $c_1,c_2,c_3$, up 
to proportionality. The up-to-proportionality business appears 
here because there is no canonical identification $C'_{p'} \cong \bP$ 
as the only distinguished points of this component are $p'$ and 
the node of $C'$.  

Put it in a different way, the rescaling
$$
t \mapsto \lambda t 
$$
acts on the leading coefficients in \eqref{x(t)} by 
$$
(c_1,c_2,c_3) \mapsto \lambda^{-r_1} (c_1,c_2,c_3), 
$$
and so acts nontrivially even on the central fiber. 

\subsubsection{}\label{s_central_fiber}

This is an important point, and so worth repeating one more 
time. In the central fiber, we have 
\begin{enumerate}
\item[---] a chain of rational curves $C'$, one end of which is 
identified with $C$ while the other contains the point $p'$,
\item[---] a divisor disjoint from $p'$ and the nodes of $C'$, and such
  that its degree on any components other than $C$ is positive, 
\item[---] this data is considered up to isomorphism which must
be identity on $C$. 
\end{enumerate}

A pointed nodal curve $C'$ is called stable if it automorphism group is 
finite. If the automorphism group is reductive, the curve is 
called \emph{semistable}. 
Semistability means some components of $C'$
contain only 2 special (that is, marked or nodal) points and the 
sequence 
$$
1 \to \left(\Ct\right)^{\textup{\# such components}} \to 
\Aut(C') \to \textup{finite group} \to 1 
$$
is exact. 

We extend 
this terminology to cover cases when one or more components of $C'$ 
are considered as rigid, equivalently, parametrized. In our example, 
the component $C\subset C'$ is rigid. This may be
reduced to the usual case by adding enough marked points to rigid 
components. 

There is a stabilization map which collapses nonrigid components with only 2 
special points. In the case at hand, $(C,p)$ is the stabilization of
$(C',p')$. 

\subsubsection{}\label{s_stages} 
Instead of constructing $\bfC^\diamond$ and then $\bfC'$ all at once, 
we could have constructed them in stages as follows. Let 
$r_1$ be the maximal rate at which $x_i(t)\to 0$, which means 
that $1/r_1$ is the minimal slope in Figure \ref{f_Newton_diagram}. Write 
$r_1 = b_1/a_1$, $\gcd(a_1,b_1)=1$, 
blow up the ideal $(x^{a_1}, t^{b_1})$, and 
pass to a branched cover of degree $a_1$. This takes care 
of the points that were approaching $p$ the fastest. Now 
they all land on the component of $C'$ that contains $p'$, 
while all other points of $D_t$ approach some point of 
$C$ as $t\to 0$, including the node of the central fiber. 

Obviously, we can deal with points approaching the node inductively, 
by looking at those that approach the fastest, doing the 
corresponding blowup, etc. 

\subsubsection{}

Looking at the list in Section \ref{s_central_fiber}, one can make 
the following important observation. The data of the divisor in 
central fiber is taken modulo the action of 
$$
\Aut(C',\{C,p\}) = \left(\Ct\right)^\textup{\# of bubbles}
$$
where bubbles refer to components of $C'$ other than $C$. The 
action of this group on divisor can and will have finite stabilizers, 
so moduli space of objects in Section \ref{s_central_fiber} is 
an orbifold. 

In fact, the need to go to branched covers precisely correlates
with these orbifold singularities. 

\begin{Exercise}
Let $a/b$ with $\gcd(a,b)=1$ be the slope of one of the 
edges in Figure \ref{f_Newton_diagram}. Show that the restriction of $D'_0$ to 
the corresponding component of $C'$ is invariant under the group 
$$
\mu_a = \{\zeta \,\big| \, \zeta^a =1\} \subset \Ct \,. 
$$
Conversely, if a divisor $D'_0$ has a nontrivial group of automorphisms, 
then there exists a family a $1$-parameter family $D_t$ producing it 
for which base change is necessary. 
\end{Exercise}

\subsection{Stable reduction for relative quasimaps}
\label{s_qm_reduction}

\subsubsection{}\label{s_def_stab_rel_QM} 
We now generalize the discussion of Section \ref{s_qm_example} 
to quasimaps to a Nakajima variety $X$. A stable quasimap 
$$
f: C \dasharrow X
$$
relative $p\in C$ is defined as a diagram 
\begin{equation}
\xymatrix{
p'\ar@{|->}[r] \ar[d] & C' \ar[d]^\pi 
\ar@{-->}^{f'}[rr]&& X\\
p\ar@{|->}[r] & C
}\label{def_rel}
\end{equation}
in which
\begin{enumerate}
\item[---] $\pi$ is the stabilization of 
a semistable curve $(C',p')$, 
%\item[---] $f'$ is a quasimap which takes the generic point of each 
%component of $C'$ to X, 
\item[---] $f'$ is nonsingular at $p$ and the
  nodes of $C'$, 
\item[---] the automorphism group of $f'$ is finite. 
\end{enumerate}
Figure \ref{f_relative_quasimap} is a pictorial representation of this data.  
\begin{figure}[!htbp]
  \centering
   \includegraphics[scale=0.75]{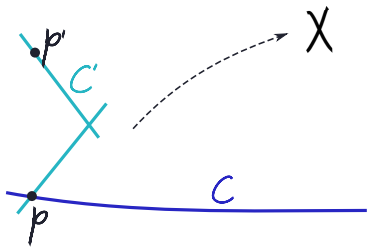}
 \caption{A quasimap relative a point $p\in C$ is a quasimap from a
   semistable curve $C'$ whose stabilization collapses a chain of 
rational curves to $p$.}
  \label{f_relative_quasimap}
\end{figure}
Note that 
nonsingularity at $p$ and the nodes implies $f'$ takes the generic 
point of each component of $C'$ to $X$. 

Two
quasimaps are isomorphic, if they fit into a diagram of the 
form 
$$
\xymatrix{
& C'_1 \ar[dd]^{\phi} \ar@{-->}^{f_1}[dr] \ar[dl]_{\pi_1}\\
C && X \\
& C'_2 \ar@{-->}_{f_2}[ur] \ar[ul]^{\pi_2}}
$$
where $\phi$ is an isomorphism which preserves the marked point.
Since 
$$
\Aut(C',\pi,p) = \left(\Ct\right)^{\textup{\# of new components}}
$$
a quasimap has a finite group of automorphism if and only if each 
component of $C'$ is either mapped nonconstantly to $X$ or 
has at least one singular point. 

\subsubsection{}
A quasimap to $X$ may be composed pointwise with the
 projection to the affine 
quotient 
$$
X_0 = \mu^{-1}(0)/G = \Spec \left(\textup{$G$-invariants} \right)\,,
$$
which has to be
constant since $X_0$ is affine. This gives a map $\QM(X) \to X_0$, and 
similarly for relative quasimaps. 

By a general result of \cite{CKM}, the moduli space of both ordinary and
relative quasimaps is proper over $X_0$. Our goal 
in this section is to get a feeling for how this works and, in 
particular, to explain the logic behind the definition of 
a relative quasimap.

\subsubsection{}

The key issue here is that of completeness, which means that
 we should be 
able to fill in central fibers for maps 
$$
g: B^\times \to \QM(X)_\textup{relative $p$}\,,
$$
where $B^\times = B \setminus \{0\}$ and $B$ is a smooth affine curve 
with a point $0$, under the assumption that the 
corresponding map to $X_0$ 
extends to $B$. 

\subsubsection{}\label{s_exten_bundle} 

By definition, a map 
$$
g: B^\times \to \QM(f: C \to X) 
$$
is given by the bundles $\{g^*\cV_i\}$ and the section $f\circ g$ 
defined on 
$$
\bfC^\times = C \times B^\times \,. 
$$
The first step to find \emph{some} extension of the bundles and the
section to $\bfC$. This preliminary extension will satisfy stability
on the central fiber, but may fail nonsingularity at $p$. 

For simplicity, assume that $C$ is irreducible (otherwise, consider
each component of $C$ separately and use nonsingularity of quasimaps 
at the nodes). The map $g$ yields a rational map in the following 
diagram 
$$
\xymatrix{& X\ar[d] \\
 \bfC \ar@{-->}[ur]^{g_\textup{rational}} \ar[r] & X_0 
} 
$$
Since $X$ is projective over $X_0$, the locus of indeterminacy of 
$g_\textup{rational}$ 
has codimension $\ge 2$, which means it consists of finitely 
many points. After shrinking the curve $B$, we may assume all 
points of indeterminacy lie in the central fiber. This extends the 
bundles and the section to the generic point of the central fiber. 

On a smooth surface $\bfC$, one can extend a vector bundle from a 
complement $U$ of a finite set of points by just pushing forward the 
sheaf of sections under $U\to \bfC$ (e.g.\ because such 
push-forward is reflexive and thus has to be locally free). This 
procedure automatically extends sections of any 
associated vector bundle. One usually talks about Hartogs-type
theorems when discussing such extensions. 

In summary, we have extended the quasimap to the central fiber, except 
it may be singular at $p$.

\subsubsection{}

One may consider quasimaps to general quotients of the form 
$$
X = W \rdd G
$$
and for properness \cite{CKM} need to assume that $W$ is affine and $G$ 
is reductive. Both assumptions are essential for the ability to 
extend a $G$-bundle and the section of the associated $W$-bundle
to $\bfC$. 

Indeed, suppose $W$ is not affine. Then it may not 
be possible to extend a map to $W$ in codimension $2$, as examples
$W=\C^2\setminus \{0\}$ or $W=\bP^1$ clearly show. 
The geometric 
reason the extension will fail for any $W$ that contains a rational 
curve is that rational curves in $W$, parametrized or not, will break
in families like the curves in Figures \ref{f_component_no_p} or 
\ref{f_degeneration_4pts}. Indeed,
take a parametrized curve, precompose with an automorphism, and 
send the automorphism to infinity. Or, take 
a double cover of a rational curve and make the branch points collide,
etc. Special fibers of the corresponding map $\bfC\to W$ are thus
forced to be reducible.

On the other hand, suppose $G$ is not reductive. 
After choosing a faithful linear representation, 
 a $G$-bundle may be seen as rank $n$ vector bundle 
together with its reduction to $G \subset GL(n)$. Such reduction 
is a section of the associated $GL(n)/G$-bundle and by Matsushima and
Onishchik 
$$
\textup{$G$ is reductive} \, \Leftrightarrow 
\, 
\textup{$GL(n)/G$ is affine} \,,
$$
 see e.g.\ Section 4.7 in
\cite{VinPop}.
And so if $G$ were not reductive, there would be $G$ bundles that cannot be extended 
in codimension 2. 

Stable reduction for moduli spaces $\Mbar_{0,n}(X)$ of stable pointed
maps deals with these issues by doing enough blowups, in the course
of which the central fiber may become an arbitrary tree of rational 
curves.  As 
explained in Section \ref{s_M0n}, having an arbitrary tree of
rational curves is precisely what we are trying to avoid. We need 
to keep the central fiber a chain of rational curves. 

\subsubsection{}\label{s_ext_bun_sing}

Also note that for a normal but singular surface $S$ there may be 
no way to extend a vector bundle on nonsingular locus to a vector
bundle on all of $S$. 

For instance, the singularities of the surface 
from Figure \ref{f_Newton_diagram} are the $2$-dimensional 
\emph{toric} singularities, which may be described as follows 
\begin{align}
  \label{toric sing1}
  S &=\Spec \C C \,, \quad C = \textup{cone in $\Lambda\cong\Z^2$} \,, \\
    & = \C^2/\textup{finite abelian subgroup $\Gamma 
\subset GL(2)$} \,,  \notag \\
    & = \C^2/ \diag(\zeta,\zeta^a) \,, \quad \zeta^n=1, \gcd(n,a)=1\,, 
\label{toric sing2}
 \end{align}
where $\C C$ in \eqref{toric sing1} denotes the semigroup algebra of 
a cone $C$ and
\eqref{toric sing2} is the explicit description of $\Gamma/\Gamma'$ 
where $\Gamma'\subset \Gamma$ is the subgroup generated by complex
reflections, i.e. elements with codimension 1 fixed loci. The
invariant ring in \eqref{toric sing2} is of the form $\C C$ where 
\begin{equation}
 \Lambda = \langle (n,0),(-a,1) \rangle \subset \Z^2 \,, \quad 
 C=\Lambda \cap \Z_{\ge 0}^2 
\end{equation}
and any cone in a rank 2 lattice can be written like that.

\begin{Exercise}
 Check the claim about the singularities and find Weil divisors on these surfaces which 
are not Cartier divisors, that is, cannot be defined by one equation 
near the singularity. Conclude that there are line bundles on the 
nonsingular locus such that their pushforward to all of $S$ is not 
locally free. 
\end{Exercise}

Another way to state the content of the above exercise is the
following. Any bundle on $\C^2 \setminus \{0\}$ is trivial as 
the restriction of its extension to $\C^2$. For the singularity 
\eqref{toric sing2}, 
we have 
$$
S \setminus \{0\} = \left(\C^2 \setminus \{0\}\right) \big/ \mu_{n}
\,,
\quad 
\mu_n =\{ \diag (\zeta,\zeta^a) \} \,. 
$$
Therefore, a vector bundle on $S\setminus\{0\}$ may
be seen as a representation of $\mu_{n}$, and it extends to a vector
bundle on the singular surface if and only if this representation is
trivial. 

\subsubsection{}

After these remarks, we return to our task of making the 
quasimap in the central fiber nonsingular at the marked point $p$
by pulling it back under a well-chosen map 
$$
\bfC_1 \to \bfC \,, 
$$
which will be a combination of a (weighted) blowup, that is, 
the blowup of an ideal of the form 
$$
I=(x^a,t^b)\,, \quad \gcd(a,b)=1 \,, 
$$
and a degree $a$ branched cover $t=(t')^a$, as in Section
\ref{s_qm_example}.  The central fiber 
$$
C_1 = E_1 \cup_{p_1} C 
$$
is the union of the exceptional divisor and the component isomorphic to
$C$. We denote by $p_1$ the node of the central fiber. The marked
point $p\in C_1$ lies on $E_1$, it is the intersection of $E_1$ with
the proper transform of $\{x=0\}$, see Figure
\ref{f_weighted_blowup}. 
\begin{figure}[!hbtp]
  \centering
   \includegraphics[scale=0.5]{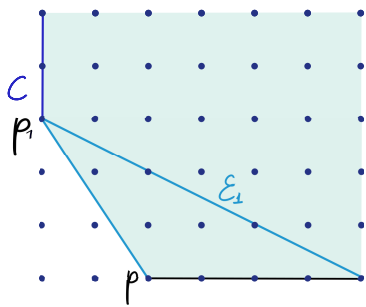}
 \caption{A weighted blowup at slope $3/2$ followed by a base change:
   we blow up of the 
ideal $I=(x^3,t^2)$ and set $t=(t')^3$.}
  \label{f_weighted_blowup}
\end{figure}

Since the vertical map in 
$$
\xymatrix{& &X\ar[d] \\
\bfC_1\ar[r] \ar@{-->}[urr]^{g_{1,\textup{rational}}} & \bfC \ar@{-->}[ur] \ar[r] & X_0 
} 
$$
is projective, the rational map $g_{1,\textup{rational}}$ is defined
at the generic point of $E_1$. In English, the restriction $g_{1,\textup{rational}}$ to
$E_1$ is obtained by taking the lowest degree terms with respect to 
the grading defined by 
$$
\deg t' =1\,, \quad \deg x=b \,. 
$$
These lowest degree terms are then a function of the coordinate $x/(t')^b$
on $E_1$. 

If the slope $a/b$ is too small, then the map
\begin{equation}
E_1 \dasharrow X \label{E1toX}
\end{equation}
is a constant map which contracts $E_1$ to 
$\lim_{t\to 0} g(0,t)$. (Why does this limit exist ?) Moreover, this 
map is nonsingular away from the node $p_1$. We choose $s$
to be the minimal slope for which this doesn't happen, that is, 
the map \eqref{E1toX} is either nonconstant or singular at some
point of $E_1 \setminus \{p_1\}$. It can be seen that
\begin{enumerate}
\item[---] a minimal slope exists, 
\item[---] the 
minimality of the slope implies $g_1$ is nonsingular at $p$.
\end{enumerate}
Indeed, if $g_1$ were singular at $p$, we could get something nontrivial with 
a further blowup, which would then have a smaller slope, etc. 

Picking a minimal slope means focusing on singularities that approach
the marked point the fastest as $t\to 0$, like in Section
\ref{s_points_approach}.

\subsubsection{}

With the singularities of $g_1$ at $E_1 \setminus \{p_1\}$ we deal 
as in Section \ref{s_exten_bundle} to get a quasimap defined on 
all of $E_1 \setminus \{p_1\}$. 

What remains is the singularity at $p_1$. We deal with it inductively,
by blowing up this point at the smallest nontrivial slope, and so on. As the 
degree of the map is consumed at each step, the process will terminate 
in finitely many steps. 

The inductive step in this process is the data of a quasimap on 
$\bfC_{n} \setminus \{p_n\}$, where $p_n$ is the $n$th node of the 
central fiber. The point $p_n$ is a singularity of $\bfC_{n}$ and, in 
light of the discussion in Section  \ref{s_ext_bun_sing}, we don't 
attempt to extend the bundles $\cV_i$ to $p_n$. Eventually, the 
quasimap will be nonsingular at all nodes and  the bundles $\cV_i$
will be extended by pull-back from $X$. 

\subsubsection{}

Note that unnecessary blowups, like blowing up at too small a
slope, or blowing up a node which is already nonsingular, precisely 
produce quasimaps which are constant and nonsingular along 
one of the components. Such quasimaps have continuous 
automorphims and so violate the definition of a stable 
relative quasimap from Section \ref{s_def_stab_rel_QM}

\subsubsection{}

It is also clear that stable reduction, in the form described above, 
involves no arbitrary choices and its final result is forced on us 
by the geometry of the quasimap in a punctured neighborhood of 
$(x,t)=(0,0)$. This can be easily formalized to show that 
the moduli space of stable relative quasimaps is separated, which 
means that the limits in $1$-parameter families are unique. 

The existence and uniqueness of limits in $1$-parameter families
is equivalent, by what is known as the valuative criterion of
properness, to 
$\QM(X)_{\textup{relative $p$}}$ being proper over $X_0$. To make 
this stament precise, we need an actual moduli space for 
quasimaps, and this is where we turn our attention now.

\subsection{Moduli of relative quasimaps}

\subsubsection{}

The construction of the moduli space of stable relative 
quasimaps may be done, in some sense, by doing 
stable reduction in reverse, starting with a universal 
deformation of the curve $C'$. 

The local model for a deformation of a node is 
\begin{equation}
xy = \varepsilon \label{xyeps}
\end{equation}
where the precise flavor of the deformation problem depends on 
how much we rigidify the two coordinate axes. If we rigidify them 
both completely by adding 2 marked points to each like in Figure
\ref{f_degeneration_4pts}, then 
$$
\varepsilon \,\, \propto \,\, \textup{crossratio}(x_1,x_2; 
\frac{\varepsilon}{y_1},\frac{\varepsilon}{y_2}) - 1\,, \quad 
\varepsilon \to 0 \,, 
$$
is the local coordinate on the moduli space $\Mbar_{0,4}$. 
\begin{figure}[!htbp]
  \centering
   \includegraphics[scale=0.64]{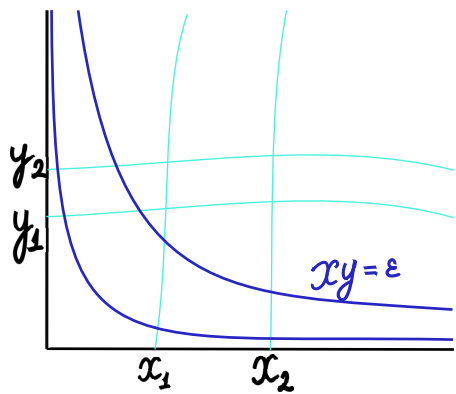}
 \caption{Degeneration of a 4-pointed curve to a stable nodal curve.}
  \label{f_degeneration_4pts}
\end{figure}

If we remove one marked point on one of the axes, then the 
nodal curve gets a $\Ct$ worth of automorphisms, while all 
$\varepsilon\ne 0$ curves are become isomorphic. We thus 
get a moduli stack of the form $\C/\Ct$, where 
$\varepsilon$ is the coordinate on $\C$. We will indicate 
this by $\C_\varepsilon$ to keep track of the name of the 
coordinate. 

The domain $C'$ of a quasimap relative one point $p$ may 
be exactly this sort of curve, with one rigid component, and 
the other component with a marked point and $\Ct$ acting 
by automorphims. A more uniform way to write its 
universal deformation is to take the trivial family $C \times
\C_\varepsilon$ and blow up the point $p$ in the central fiber. 

\subsubsection{}

The universal deformation of the domain $C'$ of a general 
relative quasimap may be written analogously. 
For concreteness, assume $C'$ has two nodes. Consider 
$$
\widetilde{\boldsymbol{C}}= 
\Bl_{\substack{
\,\,  \textup{strict transform} \\
\textup{of } \{\varepsilon_2=0\}\times \{p\} }} 
\Bl_{\{\varepsilon_1=0\}\times \{p\}} 
\C^2_{\varepsilon_1,\varepsilon_2} \times C 
$$
The toric picture of this 3-fold may be seen in Figure \ref{f_twonodes}. 
An informal, but accurate way to describe this geometry 
is to compare with an accordion, in which various sections 
of the bellows open over the divisors $\{\varepsilon_i = 0\}$. 
\begin{figure}[!htbp]
  \centering
   \includegraphics[scale=0.5]{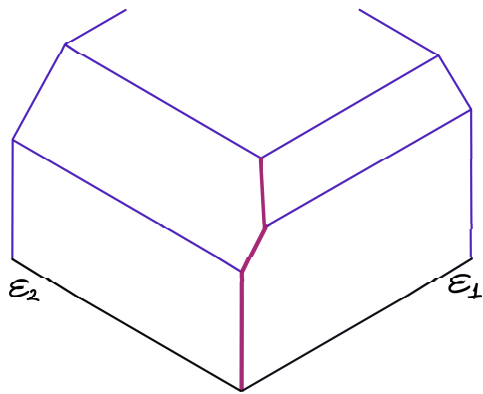}
 \caption{Universal deformation of a curve with two bubbles. 
The $i$th node remains intact over $\varepsilon_i=0$.}
  \label{f_twonodes}
\end{figure}

The fibers of the map $\bolde$ in the diagram
\begin{equation}
\xymatrix{
C'\ar[d] \ar@{^{(}->}[r]&\widetilde{\boldsymbol{C}} \ar[r]^\pi 
\ar[d]_\bolde & C \\ 
(0,0) \ar@{|->}[r] &
\C^2_{\varepsilon_1,\varepsilon_2} 
}\label{tbC}
\end{equation}
are deformations of the central fiber $C'$ and the group 
$\left(\Ct\right)^2$ acts on the base and on the fibers
by isomorphisms. This exhibits the stack of deformations of 
$C'$ as 
$$
\Def(C') = \C^2_{\varepsilon_1,\varepsilon_2} \big/ \left(\Ct\right)^2
$$
and $\widetilde{\boldsymbol{C}}$ 
as the universal family over it. 

The divisors $\varepsilon_i=0$ 
are the loci where the $i$th node remains intact.
 The strict transform of $\pi^{-1}(p)$ defines
a marked point $p'$ in the fibers of $\bolde$.

\subsubsection{}
Assume we already have a construction of ordinary, nonrelative 
quasimap moduli spaces\footnote{
In reality, \cite{CKM} construct quasimap moduli, relative or not, in one stroke 
over the moduli stack of pointed curves with a principal
$G$-bundle.}. 
Like any good moduli 
space, quasimaps moduli are defined not 
just for an individual source curve $C$, but for families 
of those, that is, for curves over some base scheme. 

An example of such family is 
the universal curve $\widetilde{\boldsymbol{C}}$ over 
the base $\C^2_{\varepsilon_1,\varepsilon_2}$ in \eqref{tbC}. 
The 
quasimaps from the fibers of $\bolde$ fit together in one 
variety over $\C^2_{\varepsilon_1,\varepsilon_2}$ and we can 
form the quotient 
\begin{equation}
\QM_{\substack{\textup{relative $p$} \\ \textup{$\le 2$ nodes}}}= 
\left\{
  \begin{matrix*}[c]
    \textup{stable quasimaps} \\
\textup{from fibers of $\bolde$ to X}
  \end{matrix*} 
\right\}  \bigg/ \left(\Ct\right)^2\label{QMC} \,. 
\end{equation}
Here stable means:
\begin{itemize}
\item[$\textup{---}$] nonsingular at $p'$ and at the nodes, 
\item[$\textup{---}$] the $\left(\Ct\right)^2$-stabilizer  is
  finite. 
\end{itemize}
The quotient \eqref{QMC} is an open set in the moduli 
spaces of quasimaps relative $p\in C$.  It contains those quasimaps 
for which the domain breaks at most twice.

\subsubsection{}
By construction, relative quasimaps come with an evaluation
map $\ev_p$ at the relative point $p'\in C'$. 

\begin{Lemma}
 The map $\ev_p$ is proper for quasimaps of fixed degree. 
\end{Lemma}

\begin{proof}
Follows from the properness of the map $\pi$ in the diagram 
$$
\xymatrix{
\QM(X)_{\textup{relative $p$}} \ar[rr]^{\ev_p} \ar[dr]_\pi && X \ar[dl] \\
& X_0 
}
$$
shown in \cite{CKM} and explained in Section \ref{s_qm_reduction}. 
\end{proof}

\subsubsection{}

The deformation theory of the prequotient in \eqref{QMC} 
is given by the deformation theory \eqref{TvirQM} of 
quasimaps from a fixed domain plus the lateral movements
along the base $\C^2_{\varepsilon_1,\varepsilon_2}$, see \cite{CKM}. Here 
it is  important that the total space $\widetilde{\boldsymbol{C}}$ is 
smooth and stable quasimaps are nonsingular at the nodes. 

Taking the quotient, we get 
\begin{equation}
T_\vir \, \QM_\textup{relative} = T^{\textup{fixed domain $C'$}}_{\vir} + 
T \Def(C') 
\,,\label{Tvir_rel}
\end{equation}
where the first 
term\footnote{The standard name for this term in algebraic geometry is 
the \emph{relative} virtual tangent space, where the adjective relative 
refers to the map to the stack of domain deformations. 
Since the world relative is already infused with a very specific and 
different meaning for us, we avoid using it here.} 
is given by \eqref{TvirQM} and 
$$
\Def(C') \cong \C^{\textup{\# of nodes}} \Big/ 
\left(\Ct\right)^{\textup{\# of nodes}}
$$
with 
$$
T \Def(C')  = \sum_{\textup{nodes $p_i\in \C'$}}  
N_{\{\varepsilon_i=0\}} - \Lie \Aut(C',\pi,p')  \,.
$$
Here $\{\varepsilon_i=0\}$ is the Cartier divisor of domains with 
$i$th node $p_i$ intacts. 

\subsubsection{}

If $p_i$ is the intersection of components
$C_{i-1}$ and $C_i$ of $C'$ then \eqref{xyeps} shows 
\begin{equation}
N_{\{\varepsilon_i=0\}} = T_{p_i} C_{i-1} \otimes  T_{p_i} C_{i}
\,. \label{N_node}
\end{equation}
Since the group $\Lie \Aut(C',\pi,p')$ acts on these tangent 
lines, they descend to line bundles on the quotient, typically 
nontrivial. We will also need the line bundle 
\begin{equation}
\psi_{p} = T_{p'}^* C' \label{psi_p}
\end{equation}
in what follows. 

\subsubsection{}
After this discussion of the deformation theory for 
relative quasimaps, we define 
\begin{equation}
\tO_\vir = \cO_\vir \otimes \left( \cK^{\textup{fixed domain}}_\vir \, 
\otimes \,  
\frac{\det \cT^{1/2} \big|_{p_2} }{ \det \cT^{1/2} \big|_{p_1}  }
\right)^{1/2} \,.\label{def_tO_rel}
\end{equation}
The existence of the square root follows from 
Lemma \ref{lsquare}.

\subsection{Degeneration formula and the glue operator} 
\label{s_degen_glue} 

\subsubsection{}

Degeneration formula in Gromov-Witten theory was proven 
in \cite{Li1,Li2}, see also \cite{LiWu} for corresponding constructions in
Donaldson-Thomas
theory. In K-theoretic GW computations, there is a correction 
to gluing, discovered in \cite{Giv} and discussed in detail in \cite{Lee}. 
All these ideas are immediately applicable to quasimaps. 

The setting of the degeneration formula is the following. 
Let a smooth curve $C_\varepsilon$ degenerate to a nodal curve 
$$
C_0=C_{0,1} \cup_p C_{0,2} \,.
$$
like in \eqref{xyeps}. 
Degeneration formula counts quasimaps from $C_\varepsilon$ in terms of 
\emph{relative} quasimaps from $C_{0,1}$ and $C_{0,2}$, where in both 
cases relative conditions are imposed at the gluing point $p$. 

\subsubsection{}

The basic idea behind the degeneration formula is that as
$C_\varepsilon$ degenerates, the quasimaps $\QM(C_\varepsilon \to X)$
degenerate to quasimaps whose domain is 
a certain destabilization of $C_0$, in which the gluing point $p$ 
is replaced by a chain of rational curves, like in Figure \ref{f_springs}. 
 This is because we require
quasimaps to be nonsingular at the nodes and letting the domain 
curve develop new components is a way to keep the singularities
from getting into the node. 
\begin{figure}[!htbp]
  \centering
   \includegraphics[scale=0.64]{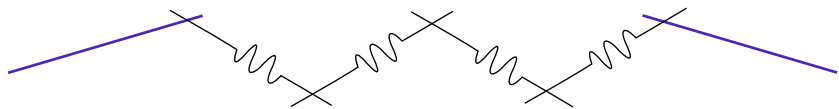}
 \caption{A semistable curve whose stabilization is the nodal curve
   $C_0$. Components with $\Ct$ automorphisms are indicated by
   springs.}
  \label{f_springs}
\end{figure}
\subsubsection{}\label{s_deg_concr}

For a concrete example, take the family \eqref{xyeps} and consider
quasimaps to $\C/\Ct$, as in Section \ref{s_qm_example}. 
A quasimap to $\C/\Ct$ is a divisor and a family of such is
given by a polynomial 
$$
f(x,y) = 0 
$$
which is not divisible by either $x$ or $y$. The problem could be 
that $f(0,0)=0$, meaning that a part of the divisor got into the 
node of $C_0$, or that the resulting quasimap $C_0 \dasharrow X$ 
is singular at the node. 

{}From Section \ref{s_toric_blowup}, we know exactly what to do. 
We take the toric blowup of the plane corresponding to the Newton 
diagram of $f$, as in Figure \ref{f_Newton_diagram}. This sends the points which 
were approaching the origin from different directions to different
components of the exceptional divisor. 

While the procedure is exactly the same, its interpretation is
different: the limiting curve is the exceptional divisor together with 
\emph{both} coordinate axes. Those are the two components of $C_0$ 
and the exceptional curves are the new components the domain had
to develop so as to keep the singularities away from the nodes. 

\subsubsection{}

This motivates defining $\QM(C_0 \to X)$ in spirit of Section 
\ref{s_def_stab_rel_QM} as the moduli spaces of quasimaps 
of the form 
\begin{equation}
\xymatrix{
C'_0 \ar[d]^\pi 
\ar@{-->}^{f'}[rr]&& X\\
C_0
}\label{def_rel2}
\end{equation}
in which
\begin{enumerate}
\item[---] the map $\pi$ collapses a chain of rational curves to the
  node of $C_0$, 
\item[---] $f'$ is nonsingular at the nodes 
of $C'_0$, 
\item[---] the automorphism group of $f'$ is finite. 
\end{enumerate}
Here the source of automorphisms is the group 
$$
\Aut(C'_0,\pi) = \left(\Ct\right)^{\textup{\# of new components}} \,. 
$$

\subsubsection{}

Domain curves of the form $C'_0$, and quasimaps from those curves, 
are commonly known as \emph{accordions}, see Figure \ref{f_accordion}. 
Indeed, they have two 
rigid components, and a chain of components rescaled by 
automorphisms. One may call those nonrigid components bellows, 
bubbles, etc. 
\begin{figure}[!htbp]
  \centering
   \includegraphics[scale=0.4]{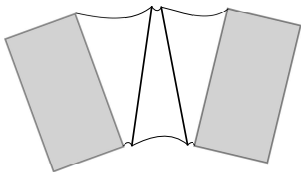}
 \caption{A poetic representation of a semistable curve which 
stabilizes to a nodal curve.}
  \label{f_accordion}
\end{figure}

\subsubsection{}
The obstruction theory is defined on the total space of the 
family 
$$
\epsilon: \QM(C_\varepsilon \to X) \to \C 
$$
and gives a sheaf whose restriction to fibers is the virtual 
structure sheaf of the fiber. 
We thus can do K-theoretic counts on any fiber by pulling 
back the corresponding point class from $\C$. In particular, 
the counts for $\varepsilon=1$ and $\varepsilon=0$ are equal. 

This is very close to what we wanted, in that it gives the counts 
of quasimaps from the generic fiber in terms of quasimaps from 
the special fiber. What remains is to separate the quasimaps 
from the special fibers into contributions of the two components. 
This is done as follows. 

\subsubsection{}

Suppose the curve $C_0'$ has $n$ nodes and let $\varepsilon_1,\dots, 
\varepsilon_n$ be smoothing parameters for each of these nodes. 
In the corresponding chart on quasimap moduli, we have the 
equality of divisors 
$$
(\varepsilon) = (\varepsilon_1 \varepsilon_2 \cdots \varepsilon_n) \,.
$$

\begin{Exercise}
See how this works in the concrete case of Section \ref{s_deg_concr}. 
\end{Exercise}

Geometrically, $\varepsilon_1 \varepsilon_2 \cdots \varepsilon_n=0$ is
the union of coordinate hyperplanes, and we have an 
inclusion-exclusion formula for the K-theory class of this union,
namely
\begin{equation}
\cO_{\varepsilon_1 \varepsilon_2 \cdots \varepsilon_n=0} = 
\sum_{\varnothing\ne S\subset \{1,\dots,n\}} (-1)^{|S|-1} 
\cO_{\bigcap_{i\in S} \{\varepsilon_i=0\}} \,, \label{O_incl_exl} 
\end{equation}
where the summation is over all nonempty subsets $S$ of hyperplanes,
that is, to a nonempty subset of nodes of $C_0'$. 

\begin{Exercise}
Prove \eqref{O_incl_exl}. 
\end{Exercise}

Geometrically, the term in \eqref{O_incl_exl} that corresponds to a
given collection of $k=|S|$ nodes computes the contribution of quasimaps 
whose domain has those $k$ nodes intact. The sum over all $S$ such 
that $|S|=k$ may be interpreted as the count of quasimaps from 
domains with $k$ marked nodes.

These quasimaps counts can be glued out of $k+1$ 
pieces as follows. 

\subsubsection{}

Let $C$ be a curve with a marked node 
obtained by gluing two curves at marked points 
$$
C = (C_1,p_1) \sqcup (C_2,p_2) \Big/ p_1 \sim p_2  \,. 
$$
A quasimap from $C$ is a quasimap from each component 
plus the requirement that it has takes well-defined and equal 
values at $p_1$ and $p_2$. In other words, we have 
a product over the evaluation maps to $X$ 
\begin{multline}
  \QM(C\to X)_\textup{marked node} = \\
 = \QM(C_1\to X)_\textup{relative $p_1$} \times_X \QM(C_2\to
 X)_\textup{relative $p_2$} 
\label{1marked_node} 
\end{multline}
and an inspection of the obstruction theory shows 
\begin{equation}
\chi(\QM(C\to X)_\textup{marked node}, \cO_\vir) = 
\chi\left(X, \ev_{p_1,*} ( \cO_\vir) \otimes \ev_{p_2,*} (\cO_\vir)\right) \,.
\label{gluing1} 
\end{equation}
The corresponding formula for the symmetrized virtual 
class $\tO_\vir$ from \eqref{def_tO} is the following. 

Note that fibers of the polarization at marked points enter
asymmetrically in the formula \eqref{def_tO}. This is very convenient 
when dealing with chains of rational curves. If the choices are 
made consistently, the polarization terms for two marked points 
glued at a node cancel. We assume this is the case, and so we are left 
with the contribution of $\cK_X$ to $\cK_\vir$. We define 
\begin{equation}
  \label{formX}
  (\cF,\cG)_X = \chi(X,\cF \otimes \cG \otimes \cK_X^{-1/2}) \,. 
\end{equation}
Then
\begin{equation}
\chi(\QM(C\to X)_\textup{marked node}, \tO_\vir) = 
\chi\left(\ev_{p_1,*} ( \tO_\vir) , \ev_{p_2,*} (\tO_\vir)\right)_X \,.
\label{gluing2} 
\end{equation}

\subsubsection{}\label{s_def_glue}

There is a parallel gluing formula for curves with $k$ marked 
nodes, except it involves a novel kind of quasimaps whose domains 
are pure accordion bellows, without rigid components. 

We define  $\QM^\text{\normalsize$\sim$}_{\textup{relative $p_1,p_2$}}$
as the moduli space of stable quasimaps with whose domain $C$ is a chain 
of rational curves joining $p_1$ and $p_2$. By definition, stability
means nonsingularity at $p_1$, $p_2$, and the nodes, together with 
being stabilized by a finite subgroup in 
$$
\Aut(C,p_1,p_2) = (\Ct)^\textup{\# of components} \,.
$$
Since a constant quasimap from a nonrigid two-pointed curve is
unstable, we make a separate definition
$$
\QM^\text{\normalsize$\sim$}_{\textup{relative $p_1,p_2$,degree 0}} =
\QM(\textup{point} \to X) = 
X \,.
$$
In other words, when bellows don't open, we define $C$ to be a point. 

We define 
\begin{align} 
\bfG &= (\ev_{p_1} \times \ev_{p_2})_* \, \tO_\vir \, z^{\deg f} \label{evevG} \\
       &= \diag \cK^{1/2}_X + O(z) \,, \notag
\end{align}
where $O(z)$ stands for the contribution of nontrivial quasimaps. 
Clearly, the counts of curves with $k$ marked nodes may be 
expressed as a certain multiple convolution of relative curve counts 
like in \eqref{gluing2} with this 
tensor. It is convenient to have an operator notation to express 
this convolution. 

\subsubsection{}

We have a natural convolution action 
$$
K(X \times X) \otimes_{K(\pt)} K(X) \to K(X)_\textup{localized}  \,.
$$
given by 
$$
(\cE,\cF) \mapsto p_{1,*} \left(\cE \otimes p_{2}^* \cF\right) \,.  
$$
For decomposable classes 
$$
K(X)^{\otimes 2} \subset K(X \times X)
$$
this action corresponds to the bilinear form 
$(\cE,\cF) \mapsto \chi(\cE\otimes \cF)$. 

In our situation, the form is twisted by $\cK^{-1/2}_X$ (which is a
pure equivariant weight, to be sure) as in \eqref{formX}, and we similarly twist the 
convolution action to 
\begin{equation}
(\cE,\cF) \mapsto p_{1,*} \left(\cE \otimes p_{2}^* \cF \otimes
  \cK^{-1/2}_X\right) \,.  \label{convolution_twisted} 
\end{equation}
With this convention (we will see in a moment that localization is not 
required) 
\begin{equation}
\label{Gas_oper}
\bfG = 1 + O(z) \in \End K(X)[[z]]
\end{equation}
as an operator on $K(X)$. We call it the \emph{gluing operator}. 
{}From \eqref{Gas_oper} we see it is invertible in $\End K(X)[[z]]$. 

\begin{Exercise}
Generalizing \eqref{gluing2} prove 
\begin{multline}
\chi(\QM(C\to X)_\textup{$k$ marked nodes}, \tO_\vir \, z^{\deg f} ) = \\ =
\chi\left((\bfG-1)^{k-1} \ev_{p_1,*} ( \tO_\vir \, z^{\deg f} ) , 
\ev_{p_2,*} (\tO_\vir \, z^{\deg f} )\right)_X \,.
\label{gluing3} 
\end{multline}
\end{Exercise}

\subsubsection{}
{} From formulas \eqref{O_incl_exl} and \eqref{gluing3} we deduce the 
following

\begin{Proposition}[\cite{Giv,Lee}]  We have 
  \begin{multline}
    \chi(\QM(C_0\to X), \tO_\vir \, z^{\deg f}  ) = \\ =
\left( \bfG^{-1} \ev_{1,*} ( \tO_\vir \, z^{\deg f} ) , \ev_{2,*}
  (\tO_\vir \, z^{\deg f} )\right)_X\,,  \label{degen_formula} 
  \end{multline}
where
$$
\ev_i : \QM(C_{0,i}\to X)_\textup{relative gluing point} \to X
$$
are the evaluation maps. 
\end{Proposition}

\begin{proof}
 Follows from
$$
\sum_{k \ge 1} (-1)^{k-1} (\bfG-1)^{k-1} = \bfG^{-1} \,. 
$$
\end{proof}

\subsubsection{}\label{s_glue_proper} 
We also have the following 

\begin{Proposition}\label{p_glue_nolocalized}
The glue operator $\bfG$ acts in nonlocalized $K$-theory of $X$. 
\end{Proposition}

\begin{proof}
Recall that $X_0$ denotes the affinization of the Nakajima variety
$X$. Since every quasimap is contracted by the map to $X_0$, the 
evaluation map in \eqref{evevG} is a proper map to 
\begin{equation}
\textup{Steinberg variety}:=X\times_{X_0} X \subset X \times X  \,.
\label{Steinberg_var} 
\end{equation}
Since the map $X\to X_0$ is proper, any class supported on the 
Steinberg variety acts in nonlocalized K-theory. 
\end{proof}

%\end{document} 

\section{Nuts and bolts} \label{s_NB} 

\subsection{The Tube}

\subsubsection{}

We define 
\begin{equation}
\Tube  = \ev_*\left(\QM_\textup{relative $p_1,p_2$}\,,
\tO_\vir \,  z^{\deg f} \right) \in K(X)^{\otimes
  2} \otimes \Q[[z]] \,, \label{Tube}
\end{equation}
and make it an operator acting from the second copy of $K(X)$ to the first
using the bilinear 
form \eqref{formX} or, more precisely, the formula 
\eqref{convolution_twisted}. 

Just like in the case of the gluing matrix $\bfG$ considered in
Proposition \ref{p_glue_nolocalized}, the operator $\Tube$ acts in nonlocalized K-theory of
$X$.

\subsubsection{}\label{s_1st_tube}

\begin{Theorem}\label{t_tube} 
We have
\begin{equation}
  \Tube = \bfG 
\end{equation}
\end{Theorem}

\noindent 
It will be instructive to go through two very different proofs of this
statement.  The first one is very short. 

\begin{proof}[First proof]
 The degeneration formula \eqref{degen_formula} gives
$$
\Tube = \Tube \, \bfG^{-1} \Tube \,. 
$$
Therefore, the operator $P=\Tube \, \bfG^{-1}$ satisfies 
$$
P^2 = P  \,.
$$
Since $P=1+O(z)$, this projector is invertible, whence $P=1$\,.
\end{proof}

\subsubsection{} 

Let $\Ct$ act on $C \cong \bP^1$ with fixed points $p_1$ and $p_2$.
Denote 
$$
q = \textup{weight of $T_{p_1} C$}  \,. 
$$
We will denote this group $\Ct_q$ to distinguish it from other tori 
present. 

For our second proof of Theorem \ref{t_tube}, we will use 
$\Ct_q$-equivariant localization.  For this, we need a description 
of the $\Ct_q$-fixed quasimaps. 

\subsubsection{}\label{s_Cq_fixed}

Let $f$ be a $\Ct_q$-fixed quasimap. By stability, its  singularities
in $C\setminus \{p_1,p_2\}$ form a set which is 
$$
\textup{finite and $\Ct_q$-invariant} \Rightarrow 
\textup{empty} \,.
$$
For maps relative $\{p_1,p_2\}$, this means that 
all singularities appear in the 
accordions and $f\big|_C$ is constant map to a point in $X$. 
In particular, all bundles $\cV_i\big|_C$ are trivial. 

In general, for the $\Ct_q$-fixed quasimaps there is a 
unique $\Ct_q$-equivariant structure on the bundles $\cV_i$ 
compatible with the quiver maps and the 
trivial\footnote{or chosen, 
in the twisted situation}
equivariant structure on the bundles $\cW_i$ and $\cQ_{ij}$. In our case, 
$\Ct_q$ acts trivially on the trivial bundles $\cV_i\big|_C$. 

\subsubsection{}\label{s_contr_accord} 
We conclude the only part of the deformation theory on which $\Ct_q$
acts are the normal direction to the fixed locus.  These correspond 
to the smoothing of the nodes at $p_i\in C$, when the nodes are 
present. For example, the contribution of the accordions 
at $p_1$ is 
\begin{equation}
\ev_* 
\left(
  \QM^\text{\normalsize$\sim$}_{\textup{relative $p'_1,p_1$}}
  \,, z^{\deg} \, \tO_\vir  \, 
\frac1{1-q^{-1} \psi_{p_1}}
\right) \in K(X)^{\otimes 2} \otimes \Q[[q^{-1}]][[z]]  
\label{accord1}
\end{equation}
because by \eqref{N_node} 
$$
N_\textup{$\exists$ node at $p_1\in C'$} = 
\psi^\vee_{p_1} \otimes q  \,.
$$

\subsubsection{}

Each $z$-coefficient in \eqref{accord1} is a 
rational function of $q$ by the following 

\begin{Lemma}\label{l_limit} 
 Let $Y$ be a projective scheme and $\cL$ an orbifold line 
bundle on $Y$. Then for any coherent sheaf $\cF$ the series
$$
\chi(q)=\chi
\left(Y,\frac{\cF}{1-q^{-1} \cL} \right)= \sum_{k\ge 0} q^{-k} \chi(Y,\cF \otimes \cL^k) 
$$
is a rational function of $q$ such that 
$$
\chi(0) = 0\,, \quad \chi(\infty) = \chi(Y,\cF) \,. 
$$
\end{Lemma}

\begin{proof}
By orbifold Riemann-Roch (or by the theory of Hilbert polynomial), we 
have 
$$
\chi(Y,\cF \otimes \cL^k) = \textup{polynomial in $\{k,\zeta_i^k\}$} 
$$
for some roots of unity $\zeta_i$ (and, in equivariant situation, also in $w_i^k$,
where $w_i$'s are certain weights of $\Aut(X)$). For any function of 
the form $f(k)=k^mt^k$ we have  
$$
\sum_{k\ge 0} f(k) \, q^{-k} = \left(
q \frac{d}{d q} \right)^m \frac{1}{1-t/q}  
\to 
\begin{cases}
 0 \,, & q \to 0 \\
f(0) \,, & q \to \infty \,,
\end{cases}
$$
the lemma follows. 
\end{proof}

\subsubsection{}\label{s_2_proof_tube}

\begin{proof}[Second proof of Theorem \ref{t_tube}]
Since the pushforward in \eqref{Tube} is proper, the 
result is a Laurent polynomial in $q$. {}From Lemma \ref{l_limit}
we conclude that 
$$
\textup{operator defined by \eqref{accord1}} \to 
\begin{cases}
1 \,, & q \to 0 \\
\bfG \,, & q \to \infty \,,
\end{cases}
$$
where $1$ is the contributions of the degree $0$, that is, 
empty accordions. 

For the accordions at $p_2$, the analysis is the same with 
$q$ replaced by $q^{-1}$, therefore 
$$
\Tube \to \bfG \,, \quad q\to\{0,\infty\}
$$
and the theorem follows. 
\end{proof}

\subsubsection{}\label{s_shorthand} 

It is convenient to use the following shorthand script for the kind of 
computations that we will be doing in this section. We compute 
with quasimaps from chains of rational curves, with two or one 
marked points at the ends, and the corresponding 
K-theoretic counts are naturally interpreted as operators or 
vectors in $K(X)[[z]]$, compare with the discussion in Section 
\ref{s_transfer_2}.  To denote these operators, and the order 
in which they compose, we can simply draw the quasimap 
domains, with the marked points, nodes, etc. 

In our chain of rational curves, we can have two kinds of 
components: the rigid (or parametrized) ones, and those with only 
two special points and $\Ct$ worth of automorphisms. We 
call the latter \emph{elastic} and denote 
\begin{alignat*}{2}
    &\includegraphics[scale=0.64]{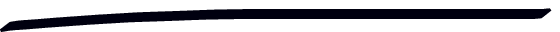}&\qquad&\textup{a
      ridig component} \\ 
 &\includegraphics[scale=0.64]{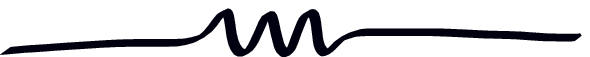}&\qquad&\textup{an 
elastic component\,.} 
  \end{alignat*}
We can have several kinds of marked points, depending on the 
constraints we put on quasimaps at that marked point. The 
basic marked points, with no constraints, is denoted 
by
$$
\includegraphics[scale=0.5]{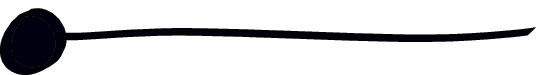}\qquad
\textup{a marked point\,.}
$$
If we require the quasimap to be nonsingular at a marked point, we
draw an open circle: 
$$
\includegraphics[scale=0.5]{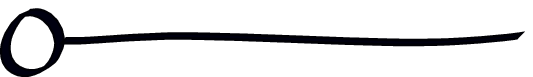}\qquad
\textup{a nonsingular point\,.}
$$
Quasimaps relative to a point will be denoted
$$
\includegraphics[scale=0.5]{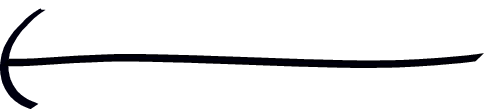}\qquad
\textup{a relative point\,.}
$$
Also, nodes count among the special points 
$$
\includegraphics[scale=0.6]{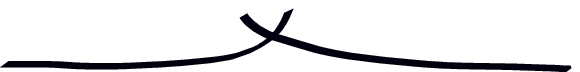}\qquad
\textup{a node\,.}
$$

\subsubsection{}
For example, in this shorthand, the degeneration and 
gluing formulas say
\begin{multline}
  \includegraphics[scale=0.5]{straight_line.eps} \quad = \\
\includegraphics[scale=0.6]{node.eps}  = \quad\\ 
\raisebox{-4pt}{\includegraphics[scale=0.5]{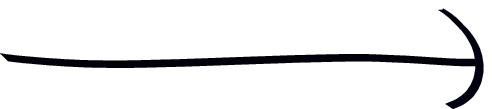}}\, 
\, \bfG^{-1} \, 
\raisebox{-4pt}{\includegraphics[scale=0.5]{relative_point.eps}}
\label{deg_glue}
\end{multline}
where 
$$
\bfG  = \raisebox{-5pt}{\includegraphics[scale=0.5]{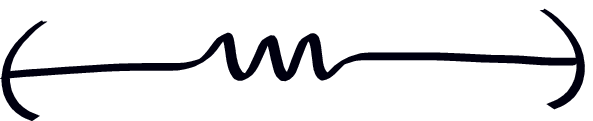}}
$$
is the gluing matrix. 

Similarly, we have 
$$
\Tube = \raisebox{-5pt}{\includegraphics[scale=0.5]{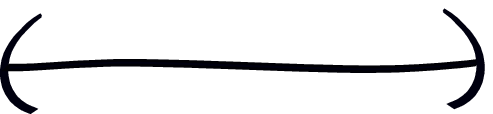}}
$$
and the argument in Section \ref{s_1st_tube} can be drawn 
as follows
\begin{multline}
  \includegraphics[scale=0.5]{tube.eps} \quad = \\
\includegraphics[scale=0.6]{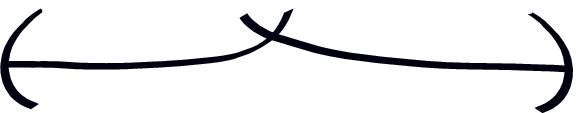}  = \quad\\ 
\raisebox{-4pt}{\includegraphics[scale=0.5]{tube.eps}}\, 
\, \bfG^{-1} \, 
\raisebox{-4pt}{\includegraphics[scale=0.5]{tube.eps}} \,. 
\label{deg_tube}
\end{multline}

\subsection{The Vertex} 

\subsubsection{}

As in \eqref{eq9}, consider the open set 
$$
\QM_\textup{nonsing $p_2$} \subset \QM 
$$
of quasimaps nonsingular at $p_2$ and define 
\begin{equation}
\Vertex  = \ev_{p_2,*} \left(\QM_\textup{nonsing $p_2$}\,,
\tO_\vir \,  z^{\deg f} \right) \in K(X)_\textup{localized} \otimes \Q[[z]] \,. \label{defVertex}
\end{equation}
This object may also be called K-theoretic Givental's $I$-function. 

While the pushforward in \eqref{defVertex} is not proper, we will see the 
fixed loci of the $\Ct_q$-action are proper, so the pushforward is 
defined in localized K-theory.

\subsubsection{}

More generally, let 
$$
\lambda \in K_G(\pt) 
$$
be a virtual representation of the group $G$ in \eqref{alg_sym}. 
It is a tensor polynomial in the defining representations $V_i$, 
something like 
$\lambda = V_1 \otimes V_2 - \Lambda^3 V_2$. 

The polynomial $\lambda$ may be evaluated on the fibers 
$\cV_i\big|_{p_1}$ over the point $p_1$. We 
can define the vertex \emph{with descendents} by 
\begin{equation}
\bra{\lambda}  = \ev_{p_2,*} \left(\QM_\textup{nonsing $p_2$}\,,
\tO_\vir \,  z^{\deg f} \, \lambda\left(\cV\big|_{p_1}\right)\right) \,.
\label{defVertex_desc}
\end{equation}
In particular
$$
\bra{1} = \Vertex  \,. 
$$
Also note that by \eqref{degf} 
\begin{equation}
\bra{\det V_i \otimes \lambda}  = 
\det V_i \otimes \bra{\lambda}  \Big|_{z_i \mapsto q z_i}
\,, \label{lambda_shift} 
\end{equation}
which is essential point behind the difference equations for 
the vertex that will be discussed at length below. 
\subsubsection{}

Let $f$ be a $\Ct_q$-fixed quasimap in $\QM_\textup{nonsing $p_2$}$. 
As in Section \ref{s_Cq_fixed}, this implies that 
$f\big|_{C\setminus\{p_1\}}$ 
is a map to a point $x\in X$. We can identify the trivial bundles 
$\cV_i\big|_{C\setminus\{p_1\}}$ and quiver maps between them with the quiver 
data for the point $x\in X$. 

Define 
$$
\bbV_i = \bigoplus_{k\in \Z}  \bbV_i[k] =
H^0(\cV_i\big|_{C\setminus\{p_2\}}) \,,
$$
where $\bbV_i[k]$ is the subspace of weight $k$ with respect to
$\Ct_q$. By invariance, all quiver maps preserve this weight 
decomposition. We define the framing spaces $\bbW[k]$  in 
the same way and obtain
\begin{equation}
\bbW_i[k] = 
\begin{cases}
W_i\,,&  k\le 0 \,,\\
0\,,&  k> 0 \,,\\
\end{cases} 
\label{bbW}
\end{equation}
because the bundles $\cW_i$ are trivial. 

Multiplication by the coordinate induces an 
embedding 
$$
\bbV_i[k]  \hookrightarrow \bbV_i[k-1] 
\hookrightarrow \dots \hookrightarrow \bbV_i[-\infty] = V_i \,,
$$
compatible with quiver maps, where $V_i$ is the quiver data for $x$. 
Thus 
\begin{equation}
\left(\QM_\textup{nonsing $p_2$}\right)^{\Ct_q} = 
\left\{
\begin{matrix}
\textup{flags of quiver subrepresentations} \\
\textup{satisfying \eqref{bbW}}
\end{matrix}
\right\} \,.  \label{flag_subr} 
\end{equation}
In particular, the space of flags \eqref{flag_subr} inherits a 
perfect obstruction theory from the ambient space of quasimaps
\footnote{
By analogy with the Hilbert scheme, it should be the case that 
 this obstruction theory may be further reduced by 
$$
0 \to \Obs_\textup{reduced} \to \Obs \to \hbar^{-1} \otimes \C^{\textup{\#
    steps}}\to 0\,,
$$
where the exponent is the number of nontrivial steps in the flag.}.

\subsubsection{}

Recall that $\theta$ denotes the stability parameter for the quiver.

\begin{Lemma}\label{l_theta}  We have 
  \begin{alignat}{4}
    \bbV[k] &= 0 \quad &&\textup{or} \quad & \theta \cdot \dim \bbV[k] &>
    0 \,,&\quad k&>0 \,, \notag\\
 \bbV[k] &= V \quad &&\textup{or} \quad & \theta \cdot \dim \bbV[k] &>
    \theta \cdot \dim V \,,& \quad k&\le 0 \,. 
\label{stability_bbV} 
  \end{alignat}
\end{Lemma}

\begin{proof}
It is convenient to use the reformulation of quiver stability 
given by Crawley-Boevey, see in particular Section 3.2 in 
\cite{GinzNak}.  In this reformulation, one trades the 
framing spaces for an extra vertex $\{\infty\}$,  with 
$$
\dim V_\infty =1 
$$
and $\dim W_i$ arrows to every other vertex $i$. 

The stability parameter is extended by 
$$
\theta_\infty = - \theta \cdot \dim V 
$$
and the stability condition reads 
$$
(\theta, \theta_\infty) \cdot (\dim V', \dim V'_\infty) > 0 
$$
for every nontrivial subrepresentation $V' \subset V$. Since 
\begin{equation}
\dim \bbV_\infty[k] = 
\begin{cases}
1\,,&  k\le 0 \,,\\
0\,,&  k> 0 \,,\\
\end{cases} 
\end{equation}
the lemma follows. 
\end{proof}

\subsubsection{}

We have 
\begin{equation}
\deg f = \left(\deg \cV_i\right) = 
\sum_{k\in \Z} \Big(\dim \bbV_i[k] - 
\delta_{\{k\le 0\}} \, \dim V_i \Big) \,.
\label{deg_sum}
\end{equation}
Let 
$$
\Camp\subset H^2(X,\R)
$$
be the closed ample cone of $X$. It is the closure of a
neighborhood of $\R_{>0} \, \theta$ 
formed by those $\theta'$ which give an isomorphic GIT quotient. 
By Lemma \ref{l_theta}, 
$$
\Big(\dim \bbV_i[k] - 
\delta_{\{k\le 0\}} \, \dim V_i \Big) \in \Camp^\vee \,, \quad \forall
k \,. 
$$
This bounds the dimensions of $\bbV_i[k]$ for fixed degree 
and gives the following 

\begin{Corollary}
 The map
$$
\ev_{p_2}:  \left(\QM_\textup{nonsing $p_2$} \right)^{\Ct_q} \to X
$$
is proper for quasimaps of fixed degree. The moduli space is 
empty for degrees outside $\Camp^\vee$. 
\end{Corollary}

\subsubsection{} 

In many instances, one can use equivariant localization with 
respect to an additional torus may be used to compute the vertex. 
If a torus $\bT$ acts on $X$ then $\Ct_q\times \bT$ fixed 
quasimaps are given in terms of flags $\bbV_i[k]$ that are 
additionally graded by the weights of $\bT$. 

The quiver 
maps are required to be $\bT$-equivariant, in particular, to 
shift the weight if the $\bT$-action on the corresponding 
arrow is nontrivial. Similarly, the framing spaces $\bbW_i$ 
acquire an additional grading.  

The characters of all these 
spaces are uniquely reconstructed by 
\begin{equation}
\bbV[k] \cong \left(\cV\big|_{p_1}\right)_\textup{$q$-weight $\ge k$}
\label{bbVcV} 
\end{equation}
from the character of $\Ct_q\times \bT$ on the fiber of $\cV$ at 
$p_1$. Similarly, the character of the deformation theory 
\eqref{TvirQM} is directly computable from the character 
of $\cV\big|_{p_1}$ by localization.

\begin{Exercise}
  Spell out the localization formula for the vertex for
  $X=T^*\Gr(k,n)$. 
\end{Exercise}

\begin{Exercise}
  Write a code that computes the vertex function for the Hilbert 
scheme of points in $\C^2$. 
\end{Exercise}

\subsubsection{}

While localization, especially in the case of isolated fixed points, 
gives a certain computational handle on the vertex, 
a much more powerful way 
to compute to compute the vertex
 is to identify a $q$-difference equation that it 
satisfies. This identification will require a certain development 
of the theory.

\subsection{The index limit} 

\subsubsection{}
Let a torus $\bA$ act on $X$ preserving the symplectic form. 
A dramatic simplification occurs in the localization formulas
if we let the equivariant variable $a\in \bA$ go to one of the 
infinities of the torus.

Very 
abstractly, suppose 
$$
X = T^*M \rd G
$$
and let a torus $\bA$ act on $M$ commuting with $G$. 
For every component $F \subset X^\bA$ we can assume that 
$\bA$ acts so that 
$$
F = T^*\!\!\left(M^\bA \right)\rd G \,.
$$
Because we consider quasimaps from a fixed domain, on 
which $\bA$ is not allowed to act, we conclude that 
$$
\iota_\QM: \QM(F)  \to \QM(X)^{\bA} 
$$
is an inclusion of a component.  Its virtual normal bundle is
\begin{equation}
N_\vir = \Hd\left(\cT^{1/2}_\mov \oplus \hbar^{-1} \left(\cT^{1/2}_\mov\right)^\vee
\right)\,, \label{NvirM} 
\end{equation}
where the subscript refers to the $\bA$-moving part of the restriction
of the polarization to $X^\bA$. 

As in Lemma \ref{lsquare}, we 
conclude 
\begin{equation}
N_\vir = \sum_{i=1}^2 \cT^{1/2}_\mov \big|_{p_i} + 
H- \hbar^{-1} H^\vee \,, \quad H=\Hd\left(\cT^{1/2}_\mov \otimes \cK_C\right) \,. 
\label{Nvir} 
\end{equation}

\subsubsection{}
Consider a balanced virtual $\bA$-module $H- \hbar^{-1} H^\vee$. 
Its contribution to $\cO_\vir \otimes \cK_\vir^{1/2}$ is 
$$
\aroof\left(H- \hbar^{-1} H^\vee\right) 
\to (-\hbar^{-1/2})^{\rk H_+ - \rk H_-} \,, \quad a\to 0\,,
$$
where 
$$
H=H_+ \oplus H_-
$$ is the decomposition of $H$ into 
attracting and repelling subspaces as $a\to 0$. By Riemann-Roch, 
$$
\rk H_+ = \deg \cT^{1/2}_{>0} - \rk \cT^{1/2}_{>0}\,,
$$
and similarly for $H_-$. 

Taking into account the asymmetry between marked points in 
\eqref{def_tO}, the contribution of the first term in \eqref{Nvir} to 
$\tO_\vir$ equals 
$$
\left(\frac{\det \cT^{1/2}_{\mov,p_2}}{\det
    \cT^{1/2}_{\mov,p_1}}\right)^{1/2} \bigotimes_{i=1}^2 \bSdh 
\left(\cT^{1/2}_\mov \big|_{p_i}\right)^\vee \,, 
$$
where the symmetric algebra with a hat was defined in
\eqref{def_bSdh}. 
Note that this has a nontrivial $q$-weight.  Its analysis leads to 
$q$-dependence in the following 

\begin{Lemma}\label{l_index_limit} 
As $a\to 0$, 
\begin{equation}
\Vertex_X \to \hbar^{\frac{\codim}{4}} \,  (-\hbar^{-1/2})^{\deg
    \cN_{>0}} q^{-\deg \cT^{1/2}_{<0}}
\, Vertex_{X^{\bA}} \,,
\label{shifta0}
\end{equation}
where 
$$
\cN_{>0}=\cT^{1/2}_{>0} + \hbar^{-1} \left(\cT^{1/2}_{<0}\right)^\vee 
$$
is the attracting part of the normal bundle to $X^\bA$. 
\end{Lemma}

\noindent 
For the full vertex, that is, the sum over all degrees with weight
$z^{\deg}$, the relationship \eqref{shifta0} means a shift of the 
argument $z$ by a monomial in $(-\hbar^{-1/2})$ and $q$.

\subsection{Cap and capping}

\subsubsection{}

In parallel to \eqref{defVertex_desc}, we define the \emph{capped
descendents}
\begin{equation}
\brar{\lambda}   = \ev_{p_2,*} \left(\QM_\textup{relative $p_2$}\,,
\tO_\vir \,  z^{\deg f} \, \lambda\left(\cV\big|_{p_1}\right) 
\right) \in K(X) \otimes \Q[[z]] \,. \label{Cap}
\end{equation}
Note that in contrast to the vertex, this is a proper 
pushforward, hence an element of nonlocalized K-theory. 
Precisely because it is a element of nonlocalized K-theory, 
it is in many ways a simpler object than the vertex. 

By construction, 
\begin{equation}
\brar{\lambda} = \lambda(V) \, \cK_X^{1/2} + O(z) 
\label{cap_class}
\end{equation}
where $O(z)$ stands for contribution of nonconstant maps. 
We will see that for many  insertions $\lambda$ the cap
\eqref{Cap}  is purely classical, that is, the $O(z)$ term in 
\eqref{cap_class} vanishes. 

\subsubsection{}

In the shorthand of Section \ref{s_shorthand}, we have 
\begin{alignat*}{2}
    &\includegraphics[scale=0.5]{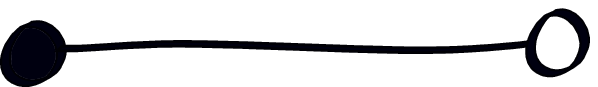}&\qquad&
\textup{the vertex\,,} \\ 
 &\includegraphics[scale=0.5]{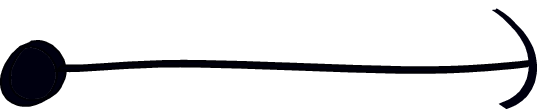}&\qquad&
\textup{the cap\,.} 
  \end{alignat*}
As in Section \ref{s_contr_accord}, localization gives
\begin{equation}
\raisebox{-2pt}{\includegraphics[scale=0.33]{cap.eps}} \quad = \quad 
\raisebox{-2pt}{\includegraphics[scale=0.33]{vertex.eps}} \, \, 
\raisebox{-9pt}{\includegraphics[scale=0.4]{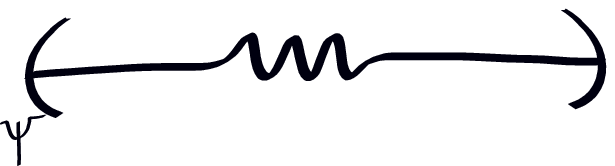}} \,, 
\label{vertex_draw}
\end{equation}
where the subscript $\psi$ in the 
\begin{equation}
\left| \textup{capping
    operator}\right|  = 
\,\, \raisebox{-11pt}{\includegraphics[scale=0.5]{capping.eps}}
\label{capping_draw} 
\end{equation}
indicates the insertion of the cotangent line in the accordion 
count as in \eqref{accord1}. 

\subsubsection{}

One can view \eqref{vertex_draw} as a linear equation for the vertex. 
The matrix \eqref{capping_draw}  of this linear equation is 
invertible and, in fact, will be identified with the fundamental 
solution of a certain $q$-difference equation in what follows. 

The constant term, or the left-hand side, of this linear 
equation will be computed using the vanishing of quantum 
terms in \eqref{cap_class} once the dimensions of the 
framing spaces are sufficiently large and the polarization is 
chosen accordingly. 

\subsection{Large framing vanishing} 

\subsubsection{}

Our proof of vanishing of the quantum terms in \eqref{cap_class} 
will follow the logic of the second proof of Theorem \ref{t_tube}, 
namely, 
it will be based on the analysis of the integrand in the 
$q\to\{0,\infty\}$ limits. In contrast to the 
situation considered in Section \ref{s_2_proof_tube} ,
 the group $\Ct_q$ may now act nontrivially 
in the fiber $\cV\big|_{p_1}$ over its fixed point $p_1$ and 
the weights of this action have to be analyzed. 

Concretely, suppose 
$$
\cV_i \cong \bigoplus_k \cO(d_{ik}) 
$$
then we get the weights $\{q^{d_{ik}}\}$ in the fiber of $\cV_i$ over
$p_1$ and hence a crude bound of the form 
\begin{equation}
\tr_{\cV_i\big|_{p_1}}  q^m = 
\begin{cases}
O\left( q^{m \deg \cV_{i,\nega}} \right)&  q\to 0\\       
O\left( q^{m \deg \cV_{i,\posi}} \right)& q\to \infty \,, 
\end{cases}
\label{q_bound}
\end{equation}
where 
\begin{equation}
\cV_{i,\nega} \cong \bigoplus_{d_{ik}<0} \cO(d_{ik}) \,, 
\quad 
\cV_{i,\posi} \cong \bigoplus_{d_{ik}>0} \cO(d_{ik}) \,.
\label{posinega} 
\end{equation}
The extension of \eqref{q_bound} 
to a general $\lambda\in K_G(\pt)$, with 
$m$ replaced by the degree of the symmetric function $\lambda$, will be 
used in what follows. 

\subsubsection{}
It is clear from \eqref{q_bound} that the key to bounding 
the degree in $q$ is to bound the degrees of $\cV_{i,\nega}$ and 
$\cV_{i,\posi}$. This is done in the following 

\begin{Lemma} \label{l_degree_split} 
Both terms in the decomposition 
$$
\deg \cV = \deg \cV_{\nega} +  \deg \cV_{\posi}
$$
belong to $\Camp^\vee$. 
\end{Lemma}

\begin{proof}
{}From \eqref{bbVcV}, we have 
\begin{align*}
\deg \cV_{\posi} & = \sum_{k>0} \dim \bbV[k] \,, \\
\deg \cV_{\nega} & = \sum_{k\le 0} \left( \dim \bbV[k] - \dim V\right)\,,
\end{align*}
and so we conclude by Lemma \ref{l_theta}. 
\end{proof}

\begin{Corollary}\label{c_theta_star} 
For a given ample cone $\Camp$, there exists $\theta_\star \in \Camp$ 
such that 
\begin{equation}
\left|\log_q \tr_{\lambda(\cV|_{p_1})} q \right| = 
O\left((\theta_\star,\deg \cV) \deg \lambda \right) \,, 
\quad q\to\{0,\infty\} \,,
\label{q_bound2}
\end{equation}
for all $\lambda\in K_G(pt)$. 
\end{Corollary}

\subsubsection{}
Recall from Section \ref{s_polar} that a polarization of $X$ is 
an equivariant $K$-theory class $T^{1/2}X$ satisfying 
\begin{equation}
T X = T^{1/2} X + \hbar^{-1} \, \left(T^{1/2} X\right)^\vee \,.
\label{def_pol} 
\end{equation}
A polarization of $X$ is unique up to adding a balanced class 
\begin{equation}
T^{1/2} X  \mapsto T^{1/2} X  + \cG - \hbar^{-1} \cG^\vee \,, \quad 
\cG\in K(X) \,,
\label{change_pol} 
\end{equation}
which affects the modified virtual canonical class 
\eqref{def_tO} by 
\begin{equation}
\tO_\vir  \mapsto \tO_\vir \, \frac{\det \cG_2}{\det \cG_1}  \,, 
\label{change_tO} 
\end{equation}
where $\cG_1$ is corresponding polynomial in $\cV_i\big|_{p_1}$. 
Note that by \eqref{lambda_shift} this also means 
\begin{equation}
\bra{\lambda} \mapsto \bra{\lambda} \Big|_{z\to q^{\det \cG}z} 
\,, 
\label{change_Vertex} 
\end{equation}
and so all polarizations are equally good for computing the 
vertex or for determining the vertex from the computation of 
the cap.

\subsubsection{}

The contribution of the polarization at $p_1$ to $\tO_\vir$ in localization 
formulas is 
\begin{equation}
(-1)^{\rk \cT^{1/2}} \bSd \cT^{1/2}_{p_1} = 
\begin{cases}
O\left( q^{-\deg \cT^{1/2}_{\nega}} \right)&  q\to 0\,,\\       
O\left( q^{-\deg \cT^{1/2}_{\posi}} \right)& q\to \infty \,, 
\end{cases} \label{T12asy} 
\end{equation}
where $\cT^{1/2}_{\posi}$ and $\cT^{1/2}_{\nega}$ are defined as in 
\eqref{posinega}. 

\begin{Definition} 
We say that the polarization $T^{1/2}$ is large with respect to 
$\theta\in \Camp$ if 
\begin{align}
\deg \cT^{1/2}_{\posi} &> \theta \cdot \deg f \,, \label{dTposi}\\ 
-\deg \cT^{1/2}_{\nega} &\ge \theta \cdot \deg f \,. \label{dTnega} 
\end{align}
for any nonconstant stable quasimap $f$ with stability parameter $\theta$. 
\end{Definition}

It is important to note that $\cT^{1/2}$ is a virtual vector bundle
and so the degree of $\cT^{1/2}_{\nega}$ need \emph{not} be negative. For 
an example of this phenomenon, one can take quasimaps 
to $T^*\Gr(n,n)$, for $n>1$. 

\subsubsection{}

Let 
$$
\bv = \dim V \,, \quad \bw = \dim W \,,
$$
be the dimension vectors corresponding to $X=X_{\bv,\bw}$ and 
recall that any Nakajima variety may be embedded 
\begin{equation}
X_{\bv,\bw} \hookrightarrow X_{\bv,\bw'} \,, \quad \bw' > \bw \,, 
\label{emb_Nak}
\end{equation}
into a larger Nakajima variety as a fixed point of the action of 
the framing torus.  The ample cones 
stabilize with respect to these embeddings. 

\begin{Proposition}\label{p_w'} 
For any $\theta$, there exists $\bw'$ and a polarization of 
$X_{\bv,\bw'}$ which is large with respect to $\theta$. 
\end{Proposition}

\begin{proof}
We take 
\begin{equation} 
T^{1/2} X_{\bv,\bw'} = \cF + \cL - \hbar^{-1} \, \cL^{-1} + \dots \,, \label{pol_w} 
\end{equation}
where 
$$
\cF = \bigoplus_{\theta_i > 0 } 
\Hom(W'_i,V_i)  +  \bigoplus_{\theta_i < 0 } \hbar^{-1} \Hom(V_i,
W'_i) \,, 
$$
$\cL$ is a line bundle of the form 
$$
\cL = \bigotimes \left(\det V_i\right)^{- C \theta_i} \,, \quad C \gg
0 
$$
and dots in \eqref{pol_w} stand for terms with $\cQ_{ij}$ from \eqref{repM}. The
degrees of all line bundles appearing in those omitted terms are bounded 
in terms of $\theta\cdot \deg f$ by Lemma \ref{l_degree_split}. 

Since $\cL$ can be made arbitrarily negative by choosing $C$ large
enough, it can be used to make the polarization \eqref{pol_w} satisfy
\eqref{dTnega}.  Now take 
$$
\bw'_i \approx C' |\theta_i|
\quad C'  \gg C \gg 0 \,. 
$$
Since $\cF$ for an actual, that is, not a virtual vector bundle $\cF$ we have 
$$
\deg \cF_{\posi} \ge \deg \cF  \approx 
\bigotimes \left(\det V_i\right)^{C' \theta_i}
$$
and this can be made arbitrarily large so that to satisfy
\eqref{dTposi}. 
\end{proof}

\subsubsection{}

Let $\theta_\star$ be as in Corollary \ref{c_theta_star}. 

\begin{Theorem}\label{t_large_framing} 
If the polarization $T^{1/2}X$  is large with respect to $(\deg \lambda)
\theta_\star$ then 
\begin{equation}
\brar{\lambda} = \lambda(V) \, \cK_X^{1/2}  
\label{cap_class_1}
\end{equation}
\end{Theorem}

\begin{proof}
  We argue as in Section \ref{s_2_proof_tube}. 
Localization may be phrased as 
\begin{equation}
\brar{\lambda}  = \bra{\lambda} \, \left| \textup{capping
    operator}\right|  \,, 
\label{bra_acc} 
\end{equation}
where the capping operator is the same as already discussed in 
Section \ref{s_contr_accord}. The matrix elements of this 
operator are given by 
\begin{equation}
\ev_* 
\left(
  \QM^\text{\normalsize$\sim$}_{\textup{relative $p_2,p'_2$}}
  \,, z^{\deg} \, \tO_\vir  \, 
\frac1{1-q \psi_{p_2}}
\right) \in K(X)^{\otimes 2} \otimes \Q[[q]][[z]]  \,. 
\label{accord2}
\end{equation}
Therefore 
  \begin{equation}
\left| \textup{capping
    operator}\right| \to 
\begin{cases}
1 \,, & q \to \infty \\
O(1) \,, & q \to 0 \,. 
\end{cases} \label{accp2} 
\end{equation}
On the other hand, 
 by Corollary \ref{c_theta_star}, \eqref{T12asy}, and the definition of a 
large polarization 
\begin{equation}
\bra{\lambda} \to 
\begin{cases}
\lambda(V) \, \cK_X^{1/2}  \,, & q \to \infty \\
O(1) \,, & q \to 0 \,. 
\end{cases} \label{braO1}
\end{equation}
The Theorem follows.  
\end{proof}

\noindent 
Observe that very crude estimates were used in the proof and 
sharper results on vanishing of quantum corrections in 
\eqref{cap_class} can be obtained along the same lines. 

\subsubsection{}\label{s_from_large_framing} 
Recall we view \eqref{bra_acc} as a system of linear equations 
that lets one determine the vertex if the cap and the capping operator 
are known. 
The capping operator will be computed as the 
fundamental solutions of a certain $q$-difference equation below. 

For any given $\lambda$, one can thus compute $\bra{\lambda}$ 
for all Nakajima varieties with sufficiently large framing vector $\bw'$, 
as in Proposition \ref{p_w'}. 

{}From this, one can go to all Nakajima varieties using 
Lemma \ref{l_index_limit}.

\section{Difference equations} \label{s_diff} 

\subsection{Shifts of K\"ahler variables}\label{s_shift_K}

\subsubsection{}

Let $\cF$ be a $\Ct_q$-equivariant coherent sheaf on $C$. We have 
\begin{equation}
\det \Hd(\cF \otimes (\cO_{p_1} - \cO_{p_2})) = q^{\deg
  \cF} \label{qdeg} \,,
\end{equation}
and this equality is $\Aut(\cF)$-equivariant. It suffices to 
check \eqref{qdeg} for $\cF\in\{\cO,\cO_{p_1}\}$, which is immediate

\subsubsection{}

Recall that for relative quasimaps, the tautological bundles $\cV_i$ 
are defined on a certain destabilization 
$$
\pi: C' \to C 
$$
of the original curve $C$. Since 
$$
\deg \cV_i = \deg \pi_* \cV_i \,, 
$$
we have
\begin{equation}
q^{\deg
  \cV_i} = \det \Hd(\cV_i \otimes \pi^*(\cO_{p_1} -
\cO_{p_2})) \label{qdeg2} \,. 
\end{equation}
This means that for any K-theoretic count of quasimaps, we have 
\begin{equation}
  \label{qdeg3}
  \bigg \langle \dots \,\, z^{\deg} \bigg \rangle \bigg|_{z_i \mapsto
    q z_i} =  \bigg \langle \dots \,\, z^{\deg} \,\, 
\det \Hd(\cV_i \otimes \pi^*(\cO_{p_1} -
\cO_{p_2}))
\bigg \rangle \,,
\end{equation}
where dots stand for arbitrary additional insertions. 
The basic relation \eqref{qdeg3} turns into $q$-difference 
equations as follows. 

\subsubsection{}
To illustrate different possibilities, consider the following setup 
that mixes relative and absolute conditions. 

We can define 
\begin{equation}
\bJ = \ev_{*}\left(\QM_{\substack{
\textup{relative $p_1$}\\
\textup{nonsing at $p_2$}}},\tO_\vir \, z^{\deg f}\right) \in K(X)^{\otimes
  2}_\textup{localized} \otimes \Q[[z]] \label{defJ}
\end{equation}
by equivariant localization. Using the bilinear 
pairing \eqref{formX} on $K(X)$, one can turn $\bJ$ into an operator acting 
from the second copy of $K(X)$ to the first. We have seen this 
operator previously, namely 
\begin{equation}
  \bJ = \left| \textup{capping at $p_1$}\right| \,.
\end{equation}
By our conventions, 
\begin{equation}
  \label{JOz} 
\bJ = 1 + O(z) \,. 
\end{equation}

\subsubsection{}

Since $\pi$ is an isomorphism over $p_2$, we have 
$$
\det \Hd(\cV_i \otimes \pi^*(\cO_{p_2})) = \ev_2^* \cL_i 
$$
where $\cL_i$ is the $i$th tautological line bundle \eqref{cLi} on $X$. 

Over $p_1$, $\pi$ is usually not an isomorphism, which allows 
room for nontrivial enumerative geometry of rational curves in $X$. 
This geometry can be conveniently packaged using degeneration
formula.

\subsubsection{}\label{sbfM}

Introduce the operator 
\begin{equation}
\bfM_{\cL_i} = \ev_*\left(\tO_\vir \,  z^{\deg f} \, 
\det \Hd\left(\cV_i \otimes
\pi^*(\cO_{p_1})\right) \right) \, \bfG^{-1} \label{defML}
\end{equation}
in which the first factor is 
defined using the moduli spaces of quasimaps relative to
\emph{both} points $p_1$ and $p_2$ and the second 
is the gluing operator. 
As in Proposition  \ref{p_glue_nolocalized}, we 
see that the operator \eqref{defML} is defined in 
nonlocalized K-theory. 

By construction, 
\begin{equation}
  \label{MOz}
  \bfM_{\cL_i} = \cL_i  + O(z) \,,
\end{equation}
where $\cL_i$ is viewed as an 
operator of tensor multiplication.

 \subsubsection{}

The curve $C$ may be equivariantly degenerated to a union 
$$
C \rightsquigarrow C_1 \cup C_2 \,,  \quad p_i \in C_i  
$$
of two $\bP^1$'s. The degeneration formula for \eqref{qdeg3} 
gives the following formula
\begin{equation}
  \label{qJi}
  \bJ(q^{\cL_i} z)  = \bfM_{\cL_i} \, \bJ(z) \, \cL_i^{-1}
\end{equation}
where 
$$
(q^{\cL} z)^d = q^{\lan \cL,d\ran} \, z^d\,,
\quad  d\in H_2(X,\Z)\,, \cL\in
\Pic(X) \,. 
$$
This immediately extends to the whole Picard group to give 
the following 
\begin{Theorem} \label{t_bfM}
  \begin{equation}
    \label{qdJ}
    \bJ(q^\cL z) \, \cL = \bfM_{\cL} \, \bJ(z) 
  \end{equation}
\end{Theorem}

\noindent 
{}From \eqref{JOz} and \eqref{MOz}, we see that the difference 
equations 
$$
\Psi(q^{\cL} z) =  \bfM_{\cL} \, \Psi(z) 
$$
have a regular singularity at $z=0$ and
that $\bJ$ is the fundamental solution. 

\subsubsection{}
Recall from Section \ref{s_2_proof_tube} that 
$$
\bJ \to 
\begin{cases}
1 \,, & q \to 0 \\
\bfG \,, & q \to \infty \,. 
\end{cases}
$$
Substituting this into \eqref{qdJ}, we obtain the following 

\begin{Corollary}
If $\cL$ is ample then 
\begin{alignat}{2}
 \bfM_{\cL} (z,q) &\to \cL \,, & \quad q&\to 0 \,, \\
\bfM_{\cL} (q^{-\cL} z,q) &\to \bfG \, \cL \,, & \quad q&\to \infty
\,. 
\end{alignat}
\end{Corollary}

\subsection{Shifts of equivariant variables}\label{s_shift_eq}

\subsubsection{}

Let $\sigma$ be a cocharacter as in \eqref{def_sigma}. 
We will consider $\sigma$-twisted quasimaps to $X$ and 
we fix the equivariant structure so that $\Ct_q$ acts 
trivially in the fiber over $p_2$. 

\subsubsection{}

We define the operator 
\begin{equation}
\bJ^\sigma: K_\bT (X) \to K_{\bT\times \Ct_q} 
(X)_\textup{localized}\otimes \Q[[z]] \label{bJsigma}
\end{equation}
 as in \eqref{defJ} but for 
twisted quasimaps. Here $\Ct_q$ acts on $X$ via $\sigma$. 

\subsubsection{}
The operator $\bJ^\sigma$ may be computed by localization with 
respect to $\C^\times_q$.  The localization formula has 
accordions opening at $p_1$ and an edge term associated to 
a ``constant'' map to $\sigma$-fixed point $x$ 
$$
f_\sigma: C \to x \in X^\sigma \,, 
$$
as in Section \ref{s_const_map}. 

The accordion terms are identical to the accordion terms 
for $\bJ$, with equivariant variables $a\in \bA$ replaced by $q^\sigma
a$, where $q^\sigma=\sigma(q)$, and therefore
\begin{equation}
\bJ^\sigma = \bJ(a q^\sigma) \, \bE(x,\sigma) \,, 
\end{equation}
where 
$$
\bE(x,\sigma) = \frac{\textup{edge term for $f_\sigma$}}
{\left(\textup{edge term for $f_0$}\right)\big|_{a\mapsto q^\sigma a}} \,.
$$
This ration of edge terms may be analyzed as follows. 

\subsubsection{}
Let $\cT_x$ be 
the vector bundle over $C$ associated to the action of $\Ct_q$ on
$T_xX$. We have 
$$
T_\vir (f_\sigma) = \Hd(\cT_x) = 
\frac{T_x^\sigma}{1-q^{-1}} + \frac{T^0_x}{1-q}\,,
$$
where $T^0_x$ means we take this vector space with a trivial $\Ct_q$
action. Therefore
\begin{equation}
T_\vir (f_\sigma) - T_\vir (f_0) \big|_{a\mapsto q^\sigma a}  = 
\frac{T_x^0 - T_x^\sigma}{1-q} \,. \label{Tvirq}
\end{equation}
The polarization terms in \eqref{def_tO} give 
$$
\textup{polarization contribution} = q^{-\frac12 \deg \det \cT^{1/2}_x}  \,. 
$$
for the twisted map and $1$ for the untwisted. 

\subsubsection{}
The entire function 
\begin{equation}
  (a)_\infty = \prod_{i=0}^\infty (1-q^i a) \label{def_phi} 
\end{equation}
is a $q$-analog of the reciprocal of the $\Gamma$-function. We 
extend it to a map 
$$
\Phi: K(X) \to K(X) \otimes \Z[[q]] 
$$
by 
$$
\Phi(\cG) = \prod (g_i)_\infty = \Ld \left(\cG \otimes \sum_{i\ge 0} q^i \right)
$$
where the product is over Chern roots $g_i$ of $\cG$. Similar 
$\Gamma$-function terms are ubiquitous in quantum cohomology, 
see for example Chapter 16 in \cite{MO1}  
and have been conceptually investigated by Iritani, see
e.g.\ \cite{Iri}.  We have
$$
T_\vir = \frac{\cG}{1-q^{-1}}  \Rightarrow \cO_\vir = \frac1{\Phi(\cG^\vee)}\,,
$$
which, in particular, applies to  \eqref{Tvirq}. 

\begin{Exercise}\label{ex_phi_pol}
If $T_\vir = H - c H^\vee$ where $c$ is a character, then 
$$
\cO_\vir \otimes \cK^{1/2}_\vir = (-c^{1/2})^{\rk H} \bSd ((1-c^{-1})
H) \,. 
$$
In particular, if $T=T^{1/2} + \hbar^{-1} \left(T^{1/2} \right)^\vee$
and $T_\vir$ is given by \eqref{Tvirq} then 
$$
T_\vir = H - c H^\vee \,, \quad H = 
\frac{T^{1/2}_{x,0} - T_{x,\sigma}^{1/2}}{1-q} \,, \quad c=
\frac{1}{\hbar q}\,, 
$$
and the rank of $H$ is the degree of the bundle $\cT^{1/2}_x$ formed
by polarizations. 
\end{Exercise}

Here we denote by $T^{1/2}_{x,\sigma}$ the polarization at 
$x$ with the action of $\Ct_q$ induced by $\sigma$. The same
virtual vector space with the trivial $\Ct_q$-action is denoted 
by $T^{1/2}_{x,0}$. 

\subsubsection{}

Including the contributions from $\cK^{1/2}_\vir$, the polarization,
and $z^{\deg}$, we obtain the following 

\begin{Lemma} We have 
\begin{equation}
\bE(x,\sigma) =
\left(z_\textup{shifted}\right)^{\lan \bmu(x), \textup{---} \otimes \sigma\ran} \, 
\, \Phi\big((1-q\hbar)\big(T^{1/2}_{x,\sigma} -
T^{1/2}_{x,0}\big)\big) \,,\label{bE}
\end{equation}
where
\begin{equation}
z_\textup{shifted}^d = z^d (-\hbar^{1/2} q)^{-\lan \det T^{1/2},d \ran} \,, 
\quad d\in H_2(X,\Z) \,. \label{z_shifted}
\end{equation}
\end{Lemma}

\begin{proof}
{} From Exercise \eqref{ex_phi_pol}, we compute 
$$
\cO_\vir \otimes \cK^{1/2}_\vir = 
\left(-\frac{1}{\sqrt{\hbar q}}\right)^{\deg \cT_x^{1/2}}
\,\, \Phi\big((1-q\hbar)\big(T^{1/2}_{x,\sigma} -
T^{1/2}_{x,0}\big)\big) \,.
$$
Together with the $q^{-\deg \cT_x^{1/2}/2}$ 
contribution of the polarization in \eqref{def_tO}, the 
prefactor becomes 
$$
(-\hbar^{1/2} q)^{-\deg \cT_x^{1/2}}  \,. 
$$
This is 
equivalent to shifting the degree variable as in \eqref{z_shifted}. 
\end{proof}

\subsubsection{}

The shift of the K\"ahler variables, as in \eqref{z_shifted}, by an
amount proportional to the determinant of the polarization also 
appears in quantum cohomology, see for example the discussion in 
Section 1.2.4 of \cite{MO1}. 

\subsubsection{}
We define an operator 
$$
\bE(\sigma) 
: K_\bT (X) \to 
\big(K_{\bT}(X) \otimes \Q[q]\big)_\textup{localized}
\otimes \Q[z^{\pm 1}] 
$$
by formula \eqref{bE} in $\sigma$-fixed points. 
Note that multiplication by $\Phi$ is defined without localization, 
but the degree term is not. Also note that once we go to $\sigma$-fixed
points, there is no difference between the domain and the source of 
the map \eqref{bJsigma}. So, localization gives 
$$
\bJ^\sigma = \bJ(q^\sigma a) \, \bE(\sigma)  \,. 
$$

\subsubsection{}

On the other hand, we can argue as in Section \ref{sbfM} and introduce
the following \emph{shift operator} 
\begin{equation}
\bfS_\sigma= 
\ev_*\left(\textup{twisted quasimaps}, \, \tO_\vir \,  z^{\deg f} \,
\right) \, \bfG^{-1} 
\label{defSs}
\end{equation}
in which the first factor is 
defined using the moduli spaces of quasimaps relative to
\emph{both} points $p_1$ and $p_2$ \,. 

Note that in contrast to the operator \eqref{defML}, 
$\Ct_q$-localization is \emph{required} to define the operator 
\eqref{defSs}. This is because a twisted rational curve in $X$ may 
run off to infinity even with two fixed marked 
points\footnote{This happens
already for $X=\C$, twisted rational curves in which are nothing but 
sections of line bundles on $\bP^1$.}.

\subsubsection{}

Degeneration formula for $\bJ^\sigma$ gives the following 

\begin{Theorem} \label{t_bfS}
  \begin{equation}
    \label{sdJ}
    \bJ(q^\sigma a)  \, \bE(\sigma) = \bfS_{\sigma} \, \bJ(a) \,.
  \end{equation}
\end{Theorem}

\subsubsection{}

It is convenient to introduce the following operator 
\begin{equation}
  \label{bJ_Phi}
  \bJ_\Phi = \bJ \, \Phi((1-q\hbar) \, T^{1/2} X) \,.
\end{equation}
It conjugates both the Kahler and equivariant shifts into 
constant $q$-difference equations 

\begin{Corollary}
  \begin{align}
    \label{conJ}
\bfM_{\cL} \, \bJ_\Phi(z) &= \bJ_\Phi(q^\cL z) \, \cL \,, \\
\bfS_{\sigma} \, \bJ_\Phi(a) &= 
\bJ_\Phi(q^\sigma a) \,  
\left(z_\textup{shifted}\right)^{\lan \bmu(\,\cdot\,), \textup{---}
  \otimes \sigma\ran} 
\,.  \notag 
    \end{align}
\end{Corollary}

Recall that \eqref{MOz} implies the difference equation in 
K\"ahler variables has a regular singularity at $z=0$. It follows 
from the results of \cite{OS} that, in fact, it has regular
singularities on the whole K\"ahler moduli space, which is the
toric variety associated to the fan of ample cones in $H^2(X,\Z)$.
 
Similarly, the 
difference equations in the equivariant variables $a\in \bA$ have 
regular singularities on a toric compactification of $\bA$ 
defined by the fan of the chambers from Section \ref{s_chambers} 
below. Below in Corollary \ref{c_reg_shift}  we will see an explicit basis in 
which the operators $\bfS_{\sigma}$ are bounded and 
invertible on the infinities of the torus $\bA$. 

Note that a similar shift operator may be defined for any cocharacter
of the full torus $\bT$, but it will lead to a difference equation
with irregular singularities.

\subsubsection{}

Conjugation to an equation with constant coefficients as in
\eqref{conJ} is as good as a solution of difference equation, and, in 
some sense, even better as it involves no multivalued
functions in $z$ and $a$. A solution of the constant coefficient 
equation, which does involve multivalued functions and, specifically, 
logarithms of K\"ahler and equivariant variables, can be written as follows. 

\subsubsection{}

A solution of a constant coefficient matrix difference equation of the form
$$
F(qz) = B F(z) 
$$
is given by $\displaystyle{F(z)=\exp\left(\frac{\ln z}{\ln q} \, \ln
    B\right)}$, where the function 
$$
\log_q z =  \frac{\ln z}{\ln q} 
$$
solves the equation $\log_q(qz) = 1+\log_q z$. Therefore, we need the 
logarithm of the operator $\cL\otimes$ acting on $K_\bT(X)$. Here 
the equivariance will be important, which is why we indicate it 
explicitly. 

\subsubsection{}

Let $R$ be an integral domain and $B: R^n \to R^n$ an operator in a free
$R$-module. Even if all eigenvalues $\{\lambda_i\}$ of $B$ lie in $R$,
most functions $f(B)$ of the operator $B$, such as projectors on generalized eigenspaces, are 
only defined after localization at $\lambda_i - \lambda_j$, and so 
$$
f(B) \in \End( R^n ) \left[ \frac{1}{\lambda_i - \lambda_j},
  f(\lambda_i)
\right] \,, 
$$
as illustrated by 
$$
\ln 
\begin{pmatrix}
a  & 1 \\ 0 & b 
\end{pmatrix}  = 
\begin{pmatrix}
\ln(a)  & \frac{\ln(a)-\ln(b)}{a-b} \\ 0 & \ln(b) 
\end{pmatrix}  \,. 
$$

\subsubsection{}

This
applies, in particular, to the operator $B=\cL\otimes$ of tensor
multiplication by a line bundle in 
$K_{\bT}(X)$. Recall from Exercise \ref{ex_otimes_cL} 
that the eigenvalues of this operators are the $\bT$-weights of $\cL$ at
the different components of $X^\bT$. In particular, for $f(x)=\ln x$, 
we need the logarithm of a weight, which is an element of $\left(\Lie
  \bT\right)^*$. 

We thus can define a map 
$$
\bxi: H^2(X,\C) \otimes \Lie \bT  \to \End K(X^\bT) \otimes \C
$$
which extends the map 
$$
\cL  \mapsto \ln (\cL \otimes \textup{---} )   
$$
by linearity to $H^2(X,\C) = \Pic(X) \otimes_\Z \C$. This is
refinement of the transpose of the map $\bmu$ from \eqref{def_mu} 
and it coincides with it in the case of isolated $\bT$-fixed points. 

\subsubsection{}

The operator
\begin{equation}
\Xi = \exp \frac{\bxi(\ln z_{\textup{shifted}},\ln t)}{\ln q}
\label{defXi} 
\end{equation}
solves the equations 
\begin{equation}
\label{eq_Xi}
\Xi(q^\cL z) = \cL \otimes \Xi\,, \quad 
\Xi(q^\sigma a) = \left(z_\textup{shifted}\right)^{\lan \bmu(\,\cdot\,), \textup{---}
  \otimes \sigma\ran} \, \Xi \,. 
\end{equation}
Therefore, we have the following 

\begin{Corollary} The operator $\bJ_\Phi(z) \, \Xi$ is a solution of  
the difference equations defined by $\bfM_{\cL}$ and 
$\bfS_{\sigma}$. 
\end{Corollary}

\subsection{Difference equations for vertices}
\label{sec:diff-equat-vert}

\subsubsection{}

Exchanging the roles of $0,\infty\in \bP^1$, we get 
\begin{equation}
\ketr{\lambda}  = \bJ \, \ket{\lambda}  \,, 
\label{ket_acc} 
\end{equation}
where 
\begin{equation}
\ket{\lambda}  = \ev_{p_1,*} \left(\QM_\textup{nonsing $p_1$}\,,
\tO_\vir \,  z^{\deg f} \, \lambda\left(\cV\big|_{p_2}\right)\right) \,,
\label{defVertex_desc_ket}
\end{equation}
and similarly for the capped vertices.

\subsubsection{}

Clearly, equation \eqref{ket_acc}  can be restated as 
\begin{equation}
\left(\bJ_\Phi \, \Xi \right)^{-1} \, \ketr{\lambda}  = 
\Xi^{-1} \, \Phi((q\hbar-1) \, T^{1/2} X)  \,  \ket{\lambda}  \,, 
\label{ket_acc2} 
\end{equation}
We have the following 

\begin{Proposition}
The vector \eqref{ket_acc2} satisfies a scalar $q$-difference
equation of order $\rk K(X)$ in variables $z$ with regular singularity at $z=0$ and 
a scalar $q$-difference
equation with regular singularities in variables $a\in\bA$. 
\end{Proposition}

We warm up for the proof with the following sequence of exercises. 

\begin{Exercise}
Let $B(a) \in GL(n,\C(a))$ be a matrix of rational functions of $a$
and let $0\ne v(a) \in \C(a)^n$ be a nonzero vector of rational
functions. If a covector $\eta(a)$ solves the equation 
\begin{equation}
\eta(qa) = B(a)^{T} \, \eta(a) \label{etaB} 
\end{equation}
then the function 
$$
f_\eta(a) = \langle \eta(a), v(a) \rangle 
$$
solves a difference equation 
\begin{equation}
\sum_{i=0}^{d} c_i(a) \, f(q^i a) = 0 \,, \quad c_i(a) \in
\C(a) \,, \label{eq_scalar} 
\end{equation}
of order $d\le n$ 
which depends on
$B(a)$ and $v(a)$, but not on a choice of $\eta(a)$ solving
\eqref{etaB}. 
\end{Exercise}

\begin{Exercise}
In the setting of the previous exercise, consider the $\C$-linear 
map 
\begin{equation}
\left\{
  \begin{matrix}
    \textup{solutions}\\ \textup{of \eqref{etaB}}
  \end{matrix}
\right\} \owns \eta \mapsto f_\eta \in 
\left\{
  \begin{matrix}
    \textup{solutions}\\ \textup{of \eqref{eq_scalar}}
  \end{matrix}
\right\} \,. \label{etafeta} 
\end{equation}
Show that if this map has kernel then the difference operator 
\begin{equation}
v(a) \mapsto B(a) \, v(qa) \label{diffBv} 
\end{equation}
is reducible, that is, preserves a nontrivial $\C(a)$-linear subspace 
$$
v(a) \in V \ne  \C(a)^n \,. 
$$
In this case, we can choose $d\le \dim_{\C(a)} V$ in \eqref{eq_scalar}
and we can always choose $d$ so that the map \eqref{etafeta} is
surjective. 
\end{Exercise}

\begin{Exercise}
If $d$ is minimal and \eqref{etaB} has regular singularities in $a$ then so does
\eqref{eq_scalar}. 
\end{Exercise}

\begin{proof}
We first assume that the framing is sufficiently large and the
polarization is chosen so that  Theorem \ref{t_large_framing}
applies. In this case, 
$$
\ketr{\lambda} \in K(X) 
$$
is independent of the K\"ahler parameters and depends rationally 
on $a\in\bA$ via the identification $K(X)\otimes_{\C[\bA]} \C(\bA) \cong K(X^\bA)
\otimes \C(\bA)$. Therefore the claim follows from the above 
Exercises. 

For arbitrary framing, we argue as in Section
\ref{s_from_large_framing} by embedding $X$ as 
$$
X = \textup{a component of    } (X')^{\bA'} \,,
$$ 
where $X'$ is a larger Nakajima variety and the torus $\bA'$ 
acts on framing spaces. 
Using 
Lemma \ref{l_index_limit} we can get the vertices for $X$ from 
vertices for $X'$. The operator $\Xi$ in \eqref{ket_acc2}  makes 
the solutions associated to different components of $(X')^{\bA'}$
grow at different exponential rates at the infinity of $\bA'$. 
Correspondingly, the differential module breaks up into 
a direct sum over the components of $(X')^{\bA'}$ in the limit. 
\end{proof}

\subsubsection{}

The differences between the vertices $\ket{\lambda}$ and
$\bra{\lambda}$ are the following: 
\begin{enumerate}
\item[---] polarization at $0,\infty$ contributes differently to 
the symmetrized virtual structure sheaf in \eqref{def_tO}, and 
\item[---] exchanging $0$ and $\infty$ sends $q$ to $q^{-1}$. 
\end{enumerate}
Therefore 
\begin{equation}
\bra{\lambda} (z,q) = \ket{\lambda} (z q^{-\det T^{1/2}}, q^{-1}) \,.
\label{braquket} 
\end{equation}

\subsubsection{}
There is no way to replace $q$ by $q^{-1}$ in \eqref{def_phi}, but
the functions 
\begin{align*}
  (a;q^{-1})_\infty  &= \prod_{i=0}^\infty (1-q^{-i} a) \\
  \frac{1}{(qa;q)_\infty} &= \prod_{i=1}^\infty (1-q^{i} a)^{-1} 
\end{align*}
solve the same difference equation 
$$
f(qa) = (1-qa) f(a)\,,
$$
reflecting the rational function identity 
$$
\frac{a}{1-q^{-1}} = - \frac{qa}{1-q} \,.
$$
Similarly
$$
\frac{q\hbar -1}{1-q} \Big|_{q\mapsto q^{-1}} = 
\frac{q-\hbar}{1-q} \,,
$$
and therefore, from the point of view of difference equations, 
$$
 \Phi((q\hbar-1) \, T^{1/2} X) \Big|_{q\mapsto q^{-1}} \equiv 
\Phi((q - \hbar) \, T^{1/2} X) \,.
$$

\subsubsection{}
Substitution \ref{braquket} takes 
$$
z_\textup{shifted} \mapsto z_\# = z \, (-\hbar^{1/2} )^{-\det T^{1/2}}
\,. 
$$
Putting it all together, we obtain the following 

\begin{Theorem}
The vector 
\begin{equation}
\Vertext(\lambda,T^{1/2})= \bra{\lambda} \, \Phi((q - \hbar) \, T^{1/2} X) \, 
\exp \frac{\bxi(\ln z_\#,\ln t)}{\ln q} \label{Vertext} 
\end{equation}
satisfies a scalar $q$-difference
equation of order $\rk K(X)$ in variables $z$ with regular singularity at $z=0$ and 
a scalar $q$-difference
equation with regular singularities in variables $a\in\bA$. 
\end{Theorem}

Formally speaking, the function $\Phi((q - \hbar) \, T^{1/2} X)$ 
may have a pole in the denominator if the $\bT$-fixed points are not 
isolated. To fix this, it may be replaced by any solution of the 
same difference equation, we can replace the polarization 
$T^{1/2} X$ by its normal directions to $X^\bT$. It is however, 
better, to interpret the whole thing as a section of a certain 
line bundle on the spectrum of $K_\bT(X)\otimes\C$, see \cite{AO}. 

\subsubsection{}
As written, the function \eqref{Vertext} and the difference equations
that it solves depends on the choice of polarization $T^{1/2}$.
However, as we will check presently, a difference choice of
polarization affects the difference equations via a shift of the
variable $z$ only.

Indeed, any other polarization may we written as $T^{1/2}+\cG - 
\hbar^{-1} \cG^\vee$ for some $\cG\in K(X)$ and we have 
the following

\begin{Exercise}
Show that 
\begin{multline}
  \Vertext(\lambda,T^{1/2}+\cG - \hbar^{-1} \cG^\vee) =
\\
  \Vertext(\lambda,T^{1/2})\Big|_{z\mapsto z q^{-\det \cG}} \Phi((q -
  \hbar) (\cG - \hbar^{-1} \cG^\vee)) \, \exp \frac{\bxi\left(\ln
    \left(\dfrac{q}{h}\right)^{\det \cG},\ln t\right)}{\ln q} \,. 
\end{multline}
\end{Exercise}

\begin{Exercise}
Show that 
\begin{equation}
  \label{PhiXi}
\Phi((q -
  \hbar) (\cG - \hbar^{-1} \cG^\vee)) \, \exp \frac{\bxi\left(\ln
    \left(\dfrac{q}{h}\right)^{\det \cG},\ln t\right)}{\ln q} 
\end{equation}
is invariant under $q$-shifts of variables in $\bA$. 
The $\Phi$-term here is a ratio of theta functions because
$$
\Phi((q-h)(a-1/a/h)) = \frac{q^{1/2}}{\hbar^{1/2}} 
\frac{\vartheta(qa)}{\vartheta(\hbar a)}\,, 
$$
where 
$$
\vartheta(x) = (x^{1/2}-x^{-1/2}) (qx)_\infty (q/x)_\infty 
$$
is the odd theta function. 
\end{Exercise}

\section{Stable envelopes and quantum groups}\label{s_Stab_Q}

\subsection{K-theoretic stable envelopes} \label{s_stab}

\subsubsection{}

At the heart of many computations done in these notes lies the 
basic principle outlined in Section \ref{s_princ}. To use it 
effectively, it is very convenient to have a supply of 
equivariant K-classes which have both 
\begin{itemize}
\item[---] small support (to help with properness), and 
\item[---] small weights of their restriction to fixed points. 
\end{itemize}
Obviously, there is a tension between these 
two desirable features as exemplified, on the one hand, by 
the structure sheaf $\cO_X$ which has zero weights but is
supported everywhere and, on the other hand, by the 
structure sheaves $\cO_x$ of torus fixed points. Those have 
minimal possible support, but their weights will exactly 
match the denominators in the localization formula. 

\subsubsection{}\label{s_chambers} 

A productive middle ground is formed by classes supported 
on the \emph{attracting} set of a subtorus 
$$
\bA \subset \Aut(X,\omega)
$$
where $\omega$ is the holomorphic symplectic form. 
Let $\xi\in \Lie \bA$ be generic and consider 
$$
\Attr_\xi = \left\{ x \in X\left|  \, \lim_{t\to \infty} e^{-t \xi} x 
\,\, \textup{exists}\right. \right\} 
\,.
$$
This is a closed set which is a union of locally closed sets 
$$
\Attr_\xi(F_i) = \left\{ x \in X\left|  \, \lim_{t\to \infty} e^{-t \xi} x 
\in F_i \right. \right\} 
\,,
$$
where $X^\bA = \bigsqcup F_i$ is the decomposition of the 
fixed-point set into connected components. The natural 
embedding 
$$
\Attr_\xi(F_i)  \subset X \times F_i 
$$
is Lagrangian with respect to the form $\omega_X - \omega_{F_i}$. 

A small variation of $\xi$ doesn't change the attracting manifolds, 
they change only if one of the weights of the normal bundle to $X^\bA$ 
in $X$ changes sign on $\xi$. These weights partition $\Lie \bA$ into 
finitely many \emph{chambers}. We denote by $\fC\owns \xi$ the 
chamber that contains $\xi$ and write $\Attr_\fC(F_i)$ is what follows, 
or simply $\Attr(F_i)$ if the choice of a chamber is understood. 

\subsubsection{}\label{s_cond_stab} 
Let 
$$
\Attr^f \subset X \times X^\bA 
$$
be \emph{full} attracting set for given $\fC$ which, 
by definition, is the smallest closed set that contains the diagonal 
in $X^\bA \times X^\bA \subset X \times X^\bA$ and is closed under 
taking $\Attr(\,\cdot\,)$. The \emph{stable envelope} in K-theory 
$$
\Stab \in K(X \times X^\bA )
$$
will be a certain improved version of $\cO_{\Attr^f}$, such that 
\begin{itemize}
\item[(1)] $\displaystyle \supp \Stab \subset \Attr^f$\,,
\item[(2)] we will have 
  \begin{equation}
\Stab = \cO_{\Attr^f} \otimes \textup{line bundle}
\label{stab_norm1} 
\end{equation}
near the diagonal in $X \times X^\bA$, which is a normalization
condition, 
\item[(3)] the weights of the restriction of $\Stab$ to off-diagonal 
components of $X^\bA \times X^\bA$ will be smaller than 
what can be expected of a general Lagrangian subvariety 
and smaller than half of the denominators in the localization formula. 
\end{itemize}

Since $\bA$ preserves the symplectic form $\omega$, the 
$\bA$-weights in the normal bundle to $X^\bA$ appear in opposite
pairs, and half of denominators above refers to a choice of one from 
each pair. 

\subsubsection{}

We first discuss the normalization, that is, the line bundle 
in \eqref{stab_norm1}.  This will require a choice of polarization 
$T^{1/2}X$. The eventual formula will make sense for a  
polarization of any kind, but to start let us assume that we 
are 
really given a half $T^{1/2}X$ of the tangent bundle $TX$, and thus a 
half $N^{1/2}$ of the normal bundle $N_F$, near a component 
$F$ of $X^\bA$.  

Since $\bA$ preserves $\omega$, all normal weights to $X^\bA$ come 
in pairs
$$
w + \hbar^{-1} w^{-1} 
$$
which we can now break according to positivity on $\fC$ or which 
of those belongs to $N^{1/2}$. We view $\fC$ as something 
that will be changing depending on circumstances, whereas the 
polarization is a choice that we are free to make and keep. The 
idea is thus to make $\Stab$ to be more like 
$\Ldm \left(N^{1/2}\right)^\vee$ and less like 
$\cO_{\Attr}$. 

The prescription for doing this is illustrated in the following table, 
where we assumed that weight $w$ appears in $N^{1/2}$ but 
may be attracting or repelling  
\begin{center}
  \begin{tabular}{ l | l | l |  l | l | l }
    $N^{1/2}$ & $\Attr$ &  $N_{<0}$ & $\cO_{\Attr}$ & $\Ldm \left(N^{1/2}\right)^\vee $ & $\Stab$
    \phantom{$\Big|$} \\  
    \hline 
    $w$ & $w$ &  $\frac1{\hbar w}$ & $(1-\hbar w)$ & $(1-\frac1w)$  & $(1-\frac1{\hbar w})
\hbar^{1/2}$
\phantom{$\Big|$}\\
    $w$ & $\frac1{\hbar w}$ &  $w$ & $(1- \frac1w)$ & $(1-\frac1w)$ 
                            & $(1-\frac1w)$\phantom{$\Big|$}\\
  \end{tabular}
\end{center}
The general formula 
\begin{equation}
\Stab\big|_{F\times F} = (-1)^{\rk N^{1/2}_{>0}} 
\left(\frac{\det N_{<0}}{\det N^{1/2}}\right)^{1/2} \, \cO_{\Attr}
\label{Stab_FF} 
\end{equation}
makes sense for an arbitrary polarization. 

One can explain this formula also from a different perspective,
namely, one would have liked to make a self-dual choice 
\begin{equation}
 \Stab\big|_{F\times F}  \approx \pm 
\left(\det N_{<0} \right)^{1/2} \, \cO_{\Attr} \,,
\label{Stab_FF_naive} 
\end{equation}
except there is no reason for this square root to exist. 
The correction by polarization in \eqref{Stab_FF} makes 
the square root well defined.

\subsubsection{}

The third property of stable envelopes in Section \ref{s_cond_stab} 
means the following. We would like a condition of the kind 
$$
\deg_\bA \Stab \Big|_{F_2 \times F_1}  < 
\deg_\bA \Stab \Big|_{F_2 \times F_2}  \,.
$$
The degree of Laurent polynomial is its \emph{Newton polygon} 
$$
\deg_\bA \sum f_{\mu} z^{\mu}  = \textup{Convex hull} \,
\left(\{ \mu, f_\mu \ne 0\} \right) \subset \bA^\wedge \otimes_\Z \Q
\,,
$$
considered \emph{up to translation}. The up-to-translation 
ambiguity corresponds to multiplication by invertible functions 
$z^\nu$, $\nu\in \bA^\wedge$. The natural partial 
ordering on Newton polygons is inclusion, again up to translation. 
This is compatible with multiplication in $\C[\bA]$. 

\subsubsection{} 

The translation ambiguity for inclusion of Newton polygons leads
to an additional degree of freedom $\cL$, called \emph{slope},
 in the following definition. 

Let $\cL$ be an equivariant line bundle on $X$. Its restriction
$\cL\big|_F$ has a well-defined equivariant weight, already 
discussed in Section \ref{s_const_map}. A fractional line 
bundle 
$$
\cL \in \Pic_\bA(X) \otimes_\Z \Q
$$
has a well-defined fractional weight at all fixed points. 

We fix a sufficiently general fractional line bundle $\cL$ and 
require that 
\begin{equation}
\deg_\bA \Stab_\cL\Big|_{F_2 \times F_1} \otimes \cL\Big|_{F_1}
\,\, \subset \,\,
\deg_\bA \Stab \Big|_{F_2 \times F_2} \otimes \cL\Big|_{F_2}
\label{degK}
\,. 
\end{equation}
Observe that
\eqref{degK} is independent on the $\bA$-linearization of $\cL$.

Also observe the inclusion in
\eqref{degK} is necessarily strict, because it is an inclusion 
of a fractional shift of an integral polygon inside an integral 
polygon. 

\subsubsection{}
It is easy to see the K-theory class $\Stab$ satisfying the 
three conditions of Section \ref{s_cond_stab}, made precise in 
equations \eqref{Stab_FF}  and \eqref{degK}, is unique if it exists,
see below. 
Existence of stable envelopes for Nakajima varieties is 
a nontrivial geometric fact, see \cite{MO2} and also \cite{AO}. 

\begin{Exercise}\label{ex_stab_polar} 
Show that stable envelopes corresponding to different 
polarizations are related by a shift of slope. 
\end{Exercise}

\subsubsection{}
Let $X$ be a hypertoric variety, that is, an algebraic 
symplectic reduction of the form 
$$
X = T^* \C^n \rd \,\, \textup{Torus}  \,.
$$
We may assume that 
$$
\textup{Torus}\cong\diag(s_1,\dots,s_k) \in GL(k)
$$
acts on $\C^n=\C^k \oplus \C^{n-k}$ via 
the defining representation on $\C^k$ and with some 
weights on the remaining coordinates. The action of 
$$
\bA = \diag(a_{k+1},\dots,a_n) \in GL(n-k)
$$
descends to $X$. The point 
$$
x = T^*\C^k\rd \textup{Torus} \in X 
$$
is fixed by $\bA$. 

To compute $\Stab(x)$ we may assume that 
\begin{itemize}
\item[---] the weights $a_i$ are attracting,
\item[---] the polarization is given by 
the cotangent directions, that is, 
$$
T^{1/2} = T^* \C^n - \hbar^{-1}\, \Lie \textup{Torus}  \,. 
$$
\end{itemize}
Indeed, if one of the weights $a_i$ is repelling, we can replace it 
in the module $\C^{n-k}$ by the corresponding cotangent direction. 
Using Exercise \ref{ex_stab_polar}, we can do the same with 
polarization. 

Let $w_{k+1},\dots,w_{n}$ denote the coordinates in cotangent 
directions of $T^*\C^{n-k}$. They descend to section of line 
bundles on $X$ and define divisors $w_i=0$ on $X$. 

\begin{Exercise}\label{exer_st_hypertor}
Show that the structure sheaf $\cO_{w_{k+1}=\dots=w_n=0}$ is the 
stable envelope of $x$ for $\cL=0$. 
\end{Exercise}

\begin{Exercise}
Show that $\Stab_{\cL}(x)$ for any other slope is given by a twist of 
$\cO_{w_{k+1}=\dots=w_n=0}$ by an integral line bundle. 
\end{Exercise}

\subsubsection{}

A typical computation with stable envelopes proceeds as
follows. One uses the support condition to argue the 
result in a Laurent polynomial in $a\in\bA$.  One then 
computes this polynomial by taking the $a\to\infty$ 
asymptotics in localization formula. The strict inclusion 
in \eqref{degK} often implies that the contribution of 
many fixed points may be discarded. 

\begin{Exercise}\label{ex_transpose} 
Show that 
  \begin{equation}
    \label{StabStab}
    \Stab^\transp_{-\fC,\, \Topp,\, -\cL}
\circ \Stab_{\fC,\, T^{1/2},\,\cL}  = \cO_{\diag X^\bA} \,,
  \end{equation}
where 
$$
\Stab^\transp \in K(X^\bA \times X) 
$$
is the transposed correspondence. It follows that 
the composition of the other order is also the identity 
  \begin{equation}
    \label{StabStab2}
    \Stab_{\fC,\, T^{1/2},\,\cL} \circ \Stab^\transp_{-\fC,\, \Topp,\, -\cL}
= \cO_{\diag X} \,. 
  \end{equation}
This gives a decomposition of the diagonal for $X$ 
if one for $X^\bA$ is known. 
\end{Exercise}

\subsection{Triangle lemma and braid relations} 

\subsubsection{}

Let 
$
\bA' \subset \bA 
$
be a subtorus such that $X^{\bA'} = X^{\bA}$, which is satisfied 
generically. Evidently, 
$$
\textup{Stable envelopes for $\bA$} = 
\textup{Stable envelopes for $\bA'$}  \,, 
$$
where 
$$
\textup{chambers in $\Lie \bA'$} = \left(\Lie \bA' \right) 
\cap 
\textup{chambers in $\Lie \bA$} \,. 
$$
In particular, $\bA'$ can be a generic $1$-parameter subgroup, 
and therefore for uniqueness of stable envelopes it is enough, 
in principle, to 
consider the case $\dim \bA =1$. 

\begin{Proposition}[\cite{MO1}] Stable envelopes are unique. 
\end{Proposition}

\begin{proof}
Let $F_i$ be a component of the fixed locus and suppose there
are two classes $S$ and $S'$ that satisfy the definition of the 
stable envelope 
$$
\Stab\big|_{F_i} \subset X \times F_i \,.
$$
We write $F_j < F_i$ if the full attracting set of $F_i$ contains
$F_j$; this is a partial order on components of the fixed locus. 
By the normalization condition, 
$$
S|_{\Attr(F_i)} = S'|_{\Attr(F_i)} 
$$
and hence $S-S'$ is supported on $\bigcup_{j<i} \Attr(F_j)$. 
Therefore, there exist a maximal $F_j < F_i$ such that 
$$
(S-S')|_{\Attr(F_j)} \ne 0  \,.
$$
The restriction of $S-S'$ to $F_j$ factors through the map 
$$
\xymatrix{
K(\Attr^f(F_j)) \ar[rrr]^{\textup{pushforward}} \ar[d] &&& K(X) \ar[d] \\
K(F_j) \ar[rrr]^{\displaystyle{\otimes \Ld N_{<0}^\vee}} &&& K(F_j) \,,
}
$$
in which the vertical maps are restrictions. Therefore the 
Newton polygon of $S-S'\big|_{F_j}$ is at least as big as the Newton polygon 
of $\Stab_{F_j \times F_j}$, in contradiction with the degree bound. 
\end{proof}

\subsubsection{}
If $\dim \bA =1$ then one can construct stable envelopes 
inductively with respect to the above partial order on the 
set of fixed components. This is essentially a version of 
the Gram-Schmidt process. In fact, instead of \eqref{degK}
one can require 
\begin{equation}
\deg_\bA \Stab \Big|_{F_2 \times F_1} 
\,\, \subset \,\,
\deg_\bA \Stab \Big|_{F_2 \times F_2}  + \textup{shift\,}_{i,j}
\label{degK2}
\,. 
\end{equation}
for generic shifts 
$$
\textup{shift\,}_{i,j} \in \Q   \,.
$$
Remarkably, a similar Gram-Schmidt process may be 
carried out not in $K(X)$, but in the derived category of 
coherent sheaves on $X$. This is a special case of the 
theory developed in \cite{DHL}, where many applications of these 
techniques can also be found. 

In that setting, $\Stab$ is a bounded complex of coherent 
sheaves on $X \times X^\bA$ and the degree bound applies to all terms of 
its derived restriction to $X^\bA \times X^\bA$, that is, to 
all groups $\Tor_k(\Stab,\cO_{F_i})$. 

For $\dim \bA > 1$, in general, it will be 
impossible to satisfy the condition \eqref{degK2} with 
generic shifts. Put it differently, 
we can still define $\Stab_\xi$ for a $1$-parameter subgroup 
generated by $\xi$ using \eqref{degK2}. However, as a function of
$\xi$, these will jump 
across some new walls in $\Lie \bA$ and these new walls will subdivide our 
old chambers $\fC$. 

These is where the importance of shifts by a weight of a line bundle 
as in \eqref{degK} comes in. Lifting \eqref{degK} to derived 
category of coherent sheaves is an active area of current research. 
It overlaps, among other things, with the study of quantizations, that
is, noncommutative deformations of the algebraic symplectic 
manifold $X$. A fractional line bundle in 
that context becomes a parameter of quantization.

\subsubsection{}

Now consider a subtorus $\bA'\subset \bA$ such that $X^{\bA'} \ne
X^{\bA}$, which means that $\Lie \bA'$ lies inside one of the walls 
in $\Lie \bA$. A chamber $\fC_\bA \subset \Lie \bA$ determines a 
pair of chambers
\begin{alignat*}{2} 
\fC_{\bA'} & = \fC_\bA \cap \Lie \bA' & \quad \in &\Lie \bA' \,,\\
\fC_{\bA/\bA'} & = \textup{Image} \, \left(\fC_\bA\right)  & \quad \in  &\Lie\bA/\bA'
\,,
\end{alignat*}
see Figure \ref{f_cones}. 
\begin{figure}[!htbp]
  \centering
   \includegraphics[scale=0.47]{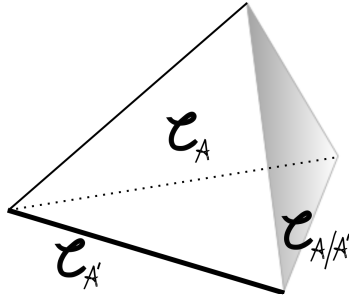}
 \caption{A cone $\fC\subset\Lie \bA$ induces a cone for each 
subtorus $\bA'$ and quotient torus $\bA/\bA'$.}
  \label{f_cones}
\end{figure}
Correspondingly, we have a diagram of maps 
\begin{equation}
\label{triang} 
\xymatrix{
K(X^\bA) \ar[rrrr]^{\Stab_{\fC_\bA}} \ar[rrd]_{\Stab_{\fC_{\bA/\bA'}}}&&&& K(X) \,.\\
&&K(X^{\bA'}) \ar[rru]_{\Stab_{\fC_{\bA'}}} 
}
\end{equation}
The shape of this diagram explains why the following statement is
known as the \emph{triangle lemma}. 

\begin{Proposition} 
The diagram \eqref{triang} commutes. 
\end{Proposition}

\begin{proof}
Consider the composition 
\begin{equation}
\Stab_{\fC_{\bA'}} \circ \Stab_{\fC_{\bA/\bA'}} \subset X \times X^{\bA}
\,. \label{StabStabTr} 
\end{equation}
It clearly satisfies the support and the normalization conditions, and so it 
is the degree condition \eqref{degK} that needs to be
checked. Further, by uniqueness of stable envelopes, it is enough to 
check the degree condition for a generic $1$-parameter subgroup
with generator 
$$
\xi = \xi' + \delta \xi \,, \quad \xi' \in \fC_{\bA'}, \,\, 
\delta \xi \in \fC_{\bA/\bA'} \,, 
$$
where $\xi'$ is generic and the perturbation 
\begin{equation}
\|\delta \xi\| \ll \| \xi' \|  \label{deltaxi} 
\end{equation}
is small. 

In order to check \eqref{degK}, we need to consider two cases. The first
case is when $F_2$ and $F_1$ belong to different components of
$X^{\bA'}$. In this case, the inequalities \eqref{degK} are 
strict for the stable envelope $\Stab_{\fC_{\bA'}}$. In particular, they 
are strict for the $1$-parameter subgroup generated by $\xi'$, 
and therefore still satisfied for a small perturbation $\xi$. 

The second case is when $F_2$ and $F_1$ belong to the 
same component of
$X^{\bA'}$. In this case, $\Stab_{\fC_{\bA'}}$ is multiplication by 
$\Ld \left(N_{X^{\bA'},<0}\right)^\vee$, twisted by polarization
terms. The corresponding terms cancel from both sides of 
\eqref{degK}, as does the $\bA'$-weight of the line bundle $\cL$. 
The degree bounds thus follows from the degree bounds 
for $\Stab_{\fC_{\bA/\bA'}}$. 
\end{proof}

% Note that we really used the fact the degree bound \eqref{degK} 
% may be checken on an arbitrary generic $\xi \in \fC$.  

\subsubsection{}
Let two chambers $\fC_{\pm}$ share a face, that is, 
let them be separated by a codimension $1$ wall 
$$
w=\Lie \bA' \subset \Lie \bA \,. 
$$
For fixed slope $\cL$, we define the following
\emph{R-matrix} 
\begin{align}
R_w &= \Stab_{\fC_-}^{-1} \, \circ \,  \Stab_{\fC_+} 
\quad \in \End K(X^\bA)_\textup{localized} \notag \\
      &=\Stab_{w,-}^{-1} \, \circ \,  \Stab_{w,+} \label{Rw} \,,
\end{align}
where 
$$
\Stab_{w,\pm} : K(X^\bA) \to K(X^{\bA'}) 
$$
are the stable envelopes for the quotient $\bA/\bA' \cong \Ct$. 
The equality of two lines in \eqref{Rw} follows from the triangle
lemma. Note from \eqref{StabStab} that the inverse of a stable 
envelope is the transpose of another stable envelope. 

{}From \eqref{Rw} we observe:
\begin{enumerate}
\item[---] the operator $R_w$ depends only on $X^\bA 
\hookrightarrow X^{\bA'}$ and, in particular, 
\item[---] as a function of equivariant parameters in $\bA$, it
factors through $\bA/\bA'$. 
\end{enumerate}

\subsubsection{}
Let $w_1,w_2,\dots$ be the walls intersecting in a given 
codimension 2 subspace of $\Lie \bA$. We may put them 
in the natural cyclic ordering as in Figure \ref{f_walls_around} 
\begin{figure}[!htbp]
  \centering
   \includegraphics[scale=0.4]{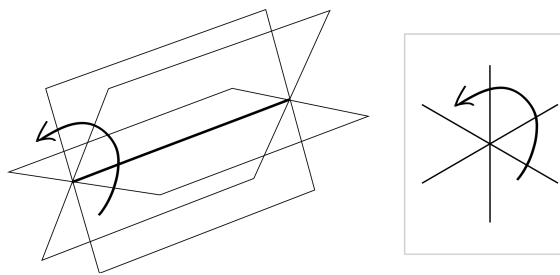}
 \caption{Product of wall $R$-matrices around a codimension 2
subspace is the identity}
  \label{f_walls_around}
\end{figure}
and then we have the following 
obvious

\begin{Proposition}
  \begin{equation}
\overset{\xleftarrow{\,\,\,\,}}{\prod} \,R_{w_i} = 1  \,.
\label{braid_relation} 
\end{equation}
\end{Proposition}

In other words, R-matrices satisfy relations of the \emph{fundamental groupoid}
of the complement of the walls in $\Lie \bA$. 

\subsubsection{}

Fix a quiver, a framing vector $\bw$ and define 
$$
X(\bw) = \bigsqcup_\bv \textup{Nakajima variety}(\bv,\bw) 
$$
where $\bv=(\dim V_i)$ is the vector of ranks of tautological 
vector bundles and some choice of the stability parameter is 
understood. Let 
$$
\bA = \{(a_1,a_2,a_3)\} \cong \left(\Ct\right)^3 
$$
act on the framing spaces and hence on $X$ so that 
$$
\bw = \sum a_i \otimes \bw^{(i)} 
$$
is the equivariant dimension vector. 

\begin{Exercise}
Describe $X^\bA$, the walls in $\Lie A$, and the corresponding fixed
loci. 
\end{Exercise}

\begin{Exercise}
Show that the braid relation \eqref{braid_relation} becomes the 
following equation 
\begin{equation}
R_{12}(a_1/a_2) R_{13}(a_1/a_3) R_{23}(a_2/a_3)  = 
 R_{23}(a_2/a_3)  R_{13}(a_1/a_3) R_{12}(a_1/a_2) 
\label{YB} 
\end{equation}
known as the \emph{Yang-Baxter equation} with spectral parameter. This 
is an equation for operators
$$
R_{ij}(a_i/a_j) \in \End K(X(\bw^{(i)})) \otimes 
K(X(\bw^{(j)}))_\textup{localized} 
$$
acting in the respective factors of 
$$
K(X(\bw^{(1)})) \otimes K(X(\bw^{(2)})) \otimes
K(X(\bw^{(3)}))_\textup{localized}  \,.
$$

\end{Exercise}

\begin{Exercise}
Show the following property 
\begin{equation}
R_{12}(a_1/a_2) \, R_{21}(a_2/a_1)  = 1 \,, \label{R_unitary} 
\end{equation}
known as the \emph{unitarity} of $R$-matrices. 
\end{Exercise}

\subsubsection{}

In particular, take 
\begin{equation}
X = X(2) = \pt \sqcup T^*\bP^1 \sqcup \pt 
\label{X2} 
\end{equation}
for the quiver with one vertex and no loops. Stable envelopes 
for $T^* \bP^1$ are a special case of the computation from 
Exercise \ref{exer_st_hypertor}, with the following result. 

Let 
$$
\bA = \diag(a_1,a_2) \subset GL(2)
$$
act on $\bP^1 = \bP(\C^2)$ in the natural way. In particular, 
it acts with weights $a_i$ on the fibers of $\cO(-1)$ over the 
fixed point $p_i\in X$. 

\begin{Exercise}
Choose the parameters of stable envelopes as follows: 
\begin{itemize}
\item[---]  the polarization $T^{1/2}$ is given by cotangent
  directions, 
\item[---]  the slope is $\cL=\cO(-\varepsilon)$, where 
$0<\varepsilon<1$, 
\item[---] the chamber $\fC$ is such that $a_1/a_2\to 0$. 
\end{itemize}
Then 
$$
\Stab_\fC(p_1) = (1-a_2 \cO(1)) \, \hbar^{1/2} \,, 
\quad
\Stab_\fC(p_2) = 1-\hbar a_1 \cO(1) \,,
$$
while
$$
\Stab_{-\fC}(p_1) = 1-\hbar a_2 \cO(1) \,,   
\quad
\Stab_{-\fC}(p_2) =  (1-a_1 \cO(1)) \, \hbar^{1/2} \,,
$$
is obtained by permuting the roles of $p_1$ and $p_2$. 
\end{Exercise}

\begin{Exercise}
Check that the computation of the $R$-matrix in the previous example
gives
$$
R_{T^*\bP^1}(u) = \Stab_{-\fC}^{-1} \, \circ \, \Stab_{-\fC} = 
\begin{pmatrix}
 \hbar^{1/2}  \, \dfrac{1-u}{\hbar-u}  & u \, \dfrac{1-\hbar}{u-\hbar} 
\vspace{5pt}\\ 
\dfrac{1-\hbar}{u-\hbar} &   \hbar^{1/2}\, \dfrac{1-u}{\hbar-u} 
 \end{pmatrix}
$$
where $u=a_1/a_2$. Together with the trivial blocks 
$$
R_{X(2)} = 
\begin{pmatrix}
  1 \\
 & R_{T^*\bP^1} \\ 
&& 1 
\end{pmatrix}
$$
for the points in \eqref{X2}, this equates 
the $R$-matrix for $X(2)$ with the standard 
$R$-matrix for $\cU_\hbar(\widehat{\mathfrak{gl}}(2))$, see 
\cite{ES,ChariPress}. 
\end{Exercise}

\subsection{Actions of quantum group}\label{s_act_Uq} 

\subsubsection{}

Let $\{F_i\}$ be a collection of vector spaces
(or free modules over some more general rings) and 
\begin{equation}
R_{F_i,F_j}(u) \in \End(F_i \otimes F_j) \otimes \C(u) \label{nabor_R}
\end{equation}
a collection of invertible operators satisfying the Yang-Baxter equation 
\eqref{YB}. From this data, one can make a certain 
\emph{Hopf algebra} $\cU$ act on modules of the form 
\begin{equation}
\label{tensor_F} 
F_{i_1}(a_1) \otimes \dots \otimes F_{i_n}(a_n) 
\end{equation}
where $F_i(a)$ is a shorthand for $F_i \otimes \C[a^{\pm1}]$ and 
the order of factors in \eqref{tensor_F} is important. 

The variables $a_i$ here are simply new indeterminates, unrelated
to the structure of the algebra $\cU$. The action of $\cU$ in 
\eqref{tensor_F} is linear over functions of $a_i$.  One should 
think of $a_i$ as points of $\Ct$ at which the modules
$F_i$ are inserted, and 
think of $\cU$ as a deformation of a gauge Lie algebra 
$$
\cU \approx  
\textup{Univ. enveloping}\left(\textup{Maps}(\Ct \to \fg)\right)\,,
$$
where each $F_i$ is a module over the Lie algebra $\fg$. 

\subsubsection{}

The definition of a Hopf algebra includes a coproduct, that 
is a coassociative homomorphism of algebras 
$$
\Delta: \cU \mapsto \cU \otimes \cU \,. 
$$
A coproduct makes modules over $\cU$ a tensor category. 
Familiar examples are a group algebra $\C G$ of a group $G$ or 
the universal enveloping algebra of a Lie algebra $\fg$. 
A group acts in tensor products via the diagonal homomorphism, 
meaning 
$$
\Delta: \C G \owns g \mapsto g\otimes g \in \C G \otimes \C G \,,
$$
which by Leibniz rule gives 
$$
\Delta(\xi) = \xi \otimes 1 + 1 \otimes \xi  \quad \xi \in \Lie G \,. 
$$
In categories related to motives, in which 
$$
\otimes = \textup{Cartesian product of varieties} \,, 
$$
tensor  product is similarly commutative.

By contrast, the coproduct for $\cU$ will not be cocommutative,
meaning that 
$$
\Delta^\opp = (12) \circ \, \Delta \ne \Delta \,,
$$
and so the order of factors in the tensor product will be important. 

\subsubsection{}
The coproduct will be not entirely noncommutative: 
after localization in $a$, tensor products in either order 
become isomorphic, that is 
$$
F_1(a_1) \otimes F_2(a_2) \otimes \C(a_1/a_2) 
\cong F_2(a_2) \otimes F_1(a_1)  \otimes \C(a_1/a_2) 
$$
as $\cU$ modules. In fact, the isomorphism is given by 
\begin{equation}
R^\vee_{F_1(a_1),F_2(a_2)} = (12) \circ R_{F_1,F_2} (a_1/a_2)
\,.  \label{Rcheck} 
\end{equation}
The same formula will work for general modules of the 
form \eqref{tensor_F}, if we extend the definition of 
$R$-matrices as follows 
$$
R_{F_1(a_1), F_2(a_2)\otimes F_3(a_3)} = 
R_{F_1(a_1), F_3(a_3)} \, R_{F_1(a_1), F_2(a_2)} 
$$
and so on, see Figure \ref{f_R_crossing} which illustrates the order 
in which the $R$-matrices should be multiplied. 
\begin{figure}[!htbp]
  \centering
   \includegraphics[scale=0.64]{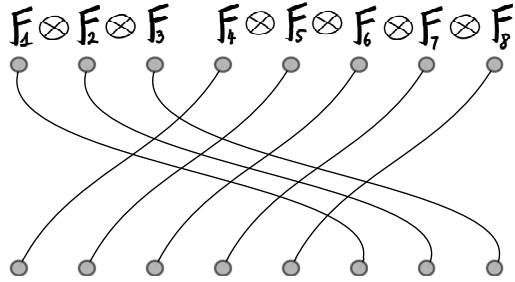}
 \caption{To write an $R$-matrix for two modules of the form
\eqref{tensor_F}, one multiplies $R$-matrices for factors in the 
order of crossings in the above diagram. Here, we get 
$R_{18} \cdots R_{24} R_{35} R_{34}$, where $R_{24}$ and $R_{35}$ 
can be places in arbitrary order.}
  \label{f_R_crossing}
\end{figure}

\begin{Exercise}
Prove that so defined $R$-matrices satisfy the same 
braid relations of the form 
$$
R_{12} R_{13} R_{23} = 
 R_{23}  R_{13} R_{12} \,,
$$
and the Yang-Baxter equation 
$$
R_{12}(u_1/u_2) R_{13}(u_1/u_3) R_{23}(u_2/u_3)  = 
 R_{23}(u_2/u_3)  R_{13}(u_1/u_3) R_{12}(u_1/u_2) 
$$
if we make $u\in \Aut(\Ct)$ act on modules \eqref{tensor_F} 
by shifting all variables $a_i \mapsto u a_i$. 
\end{Exercise}

\subsubsection{}

Let $\bF$ be a module of the form \eqref{tensor_F}. The algebra
$\cU$ may be defined in a very down-to-earth way as the 
subalgebra 
\begin{equation}
\cU \subset \prod_{\bF} \End(\bF) \label{UEndF} 
\end{equation}
generated by matrix elements of 
$$
R_{\bF,F_i(u)} \in \End(\bF) \otimes \End(F_i) \otimes \C(u) 
$$
in the auxiliary space $F_i$ for all $i$. Such matrix element 
is a rational function of $u$ with values in $\End(\bF)$ and 
we add to $\cU$ the coefficients of the expansion 
of this rational function as both $u\to 0$ and $u\to \infty$. 
As the degrees of these rational functions are unbounded, 
there is no universal relations between these coefficients. 

The natural projection 
$$
\prod_{\bF} \End(\bF) \mapsto \prod_{\bF_1, \bF_2} \End(\bF_1 \otimes 
\bF_2) 
$$
restricts to a map 
$$
\Delta: \cU \mapsto \cU \, \widehat{\otimes} \, \cU 
$$
where the completion is required if the auxiliary spaces 
are infinite-dimensional, see the discussion in Section 5 of
\cite{MO1}.  It is clear that $\Delta$ is a coassociative coproduct. 

\begin{Exercise}
Show that \eqref{Rcheck} intertwines $\cU$ action. 
\end{Exercise}

Geometrically defined $R$-matrices have an additional 
parameter, namely the slope. It can be shown \cite{OS} that the 
resulting algebra $\cU$ does not depend on the slope. 
However, the ability to vary the slope is of significant help in 
analysis of the structure of $\cU$. 

\subsubsection{}

Among $\bF$ in \eqref{UEndF} we include 
the trivial representation $\C$ such that $R_{\C,F}=1$ for 
any $F$. 

\begin{Exercise}
Show that the induced map $\cU \to \C$ is a counit for
$\Delta$. 
\end{Exercise}

\subsubsection{}

In addition to a coproduct, a Hopf algebra has an anti-automorphism
$$
S: \cU \to \cU
$$
called the \emph{antipode}, see \cite{ES,ChariPress}. It satisfies 
\begin{equation}
m \circ (1 \otimes S) \circ\Delta =
m \circ (S \otimes 1) \circ \Delta = \textup{unit} \circ
                                     \textup{counit} 
\end{equation}
where $m: \cU \otimes \cU \to \cU$ is the multiplication. 

In applications, $S$ is invertible, even though 
invertibility of the antipode is not a part of the standard axioms. 
For the category of modules, the existence of the antipode means 
existence of \emph{dual} modules. 

For noncommutative $\otimes$, one has to distinguish between 
the \emph{left} duals, with canonical maps 
\begin{equation}
F^* \otimes F \to \C \,, \quad \C \to F \otimes F^* \,,
\label{left_dual} 
\end{equation}
and and the \emph{right} duals, in which 
the order of tensor factors in \eqref{left_dual} is reversed. 

\begin{Exercise}
Show that the assignment 
$$
\cU \owns x \mapsto S(x)^* \in \End(F^*) \,,
$$
where start denotes the transpose operator, makes vector 
space dual $F^*$ into a left-dual module for a Hopf algebra $\cU$. 
Show that the inverse of $S$ is needed to make it a right-dual 
module. 
\end{Exercise}

\subsubsection{}
Following \cite{Reshet}, one can extend the collection 
\eqref{nabor_R} to dual spaces $F_i^*$ by the following rules 
\begin{alignat}{2}
&R_{F_1^*,F_2} &&= \left(R_{F_1,F_2}^{-1}\right)^{*_1} \notag\\
&R_{F_1,F_2^*} &&= \left(R_{F_1,F_2}^{*_2}\right)^{-1} \label{R_duals}\\
&R_{F_1^*,F_2^*} &&= \left(R_{F_1,F_2}\right)^{*_{12}} \,, \notag 
\end{alignat}
where $*_1$ means taking the transpose with respect to the 
first tensor factor, etc. 

\begin{Exercise}
Show that the matrices \eqref{R_duals} satisfy Yang-Baxter
equations. 
\end{Exercise}

\noindent
We can thus assume that the set $\bF$ in \eqref{UEndF} is 
closed under duals and define 
$$
S : \cU \to \widehat{\cU} 
$$
as the restriction of 
$$
\End(\bF) \xrightarrow{\quad * \quad} \End(\bF^*)  \,.
$$

\begin{Exercise}
Show this map satisfies the antipode axiom and that $\bF^*$ 
is the left-dual module for $\cU$. 
\end{Exercise}

\section{Quantum Knizhnik-Zamolodchikov equations}\label{s_KZ} 

\subsection{Commuting difference operators from $R$-matrices} 

\subsubsection{}

Consider the group 
$$
\overline{W} = S(n) \ltimes \Z^n \subset \Aut(\Z^n)
$$
generated by shifts and permutation of the coordinates. 
It contains the affine Weyl group 
$$
W_\textup{aff} \subset \overline{W} 
$$
of type $A_{n-1}$ generated by $S(n)$ and the reflection 
$$
s_0 \cdot (x_1,\dots,x_n) = (x_{n}+1,x_2,\dots,x_1-1) 
$$
in the affine root 
$$
x_1-x_n = 1 \,.
$$
Together with the simple positive roots of $A_{n-1}$, it delimits the 
fundamental alcove 
\begin{equation}
\dots\ge x_n+1 \ge x_1 \ge x_2 \ge \dots \ge x_n \ge x_1-1  \ge \dots
\,. 
\label{aff_alcove} 
\end{equation}
This stabilizer of \eqref{aff_alcove} in $\overline{W}$ is generated
by the transformation
$$
\pi \cdot (x_1,\dots,x_n) = (x_n+1,x_1,\dots,x_{n-1}) 
$$
that shifts all entries in \eqref{aff_alcove} by one. Projected 
along the $(1,1,\dots,1)$-direction, this becomes the order $n$ 
outer symmetry of the 
Dynkin diagram of $\widehat{A}_{n-1}$. 
The group $\overline{W}$ is generated by $W_\textup{aff}$ and $\pi$ 
with the relations 
\begin{equation}
\pi s_{i} = s_{i+1} \pi \,, \quad i \in \Z/n \,. \label{rel_pi_s} 
\end{equation}
In particular, we have the following expression for the shifts 
of the coordinates 
\begin{align}
(1,0,\dots,0) &= \pi \, s_{n-1} \dots s_{2} s_1 \,, \label{100} \\
(0,1,\dots,0) &= s_1 \pi \, s_{n-1} \dots s_{2} \,, \quad 
                \textup{etc.} 
\label{010} 
\end{align}

\subsubsection{}

Let $F_i$ be as in Section \ref{s_act_Uq} and define
$$
P = \bigoplus_{w\in S(n)} F_{w(1)}(a_{w(1)}) \otimes \dots \otimes 
F_{w(n)}(a_{w(n)}) \,.
$$
Consider the $\cU$-intertwiners in $P$ as in \eqref{Rcheck}
\begin{align}
 s_k & = R^\vee_k \notag\\
&=(k,k+1) \, R_{F_{w(k)},F_{w(k+1)}}(a_{w(k)}/a_{w(k+1)}) \,\,
\label{sRcheck}
\end{align}
where $(k,k+1)$ permutes the $k$th and $(k+1)$st factors. 

\begin{Exercise}
Deduce from \eqref{R_unitary} that the operators \eqref{sRcheck}
form a representation of the symmetric group. 
\end{Exercise}

Let
$$
G: F_i(a) \to F_i(a)
$$
be an invertible operator such that 
\begin{equation}
\left[ G\otimes G, R\right] = 0 \,.  \label{GGR} 
\end{equation}
Define
$$
\pi = G_1 \, (1,2,3,4\dots n) \,,
$$
where $G_1$ means it acts in the first tensor factor. 

\begin{Exercise}
Check this operator satisfies relations \eqref{rel_pi_s}. 
\end{Exercise}

\noindent
By analogy with \eqref{aff_alcove}, one can think of infinite 
quasi-periodic tensors of the form 
$$
\cdots \otimes G f_n(a_n)  \otimes f_1(a_1) \otimes f_2(a_2) \otimes \cdots \otimes f_n(a_n) \otimes 
G^{-1} f_1(a_1) \otimes G^{-1} f_2(a_2) \otimes \cdots \,.
$$
The operator $\pi$ shifts this infinite tensor by one step. 
This operation occurs naturally in the study of spin chains
with quasi-periodic boundary condition. 

It follows that we get a representation of the group $S(n)\ltimes
\Z^n$ and, in particular, $n$ commuting operators of the 
form 
\begin{align}
\textup{\eqref{100}} &= G_1 R_{1,n} \dots R_{12} \label{qKZ1} \,, \\
\textup{\eqref{010}} & = R_{2,1} G_2 R_{2,n} \dots R_{23}\,, \quad 
\textup{etc.} 
\notag  
\end{align}

\subsubsection{}
There are two ways to obtain an operator $G$ satisfying \eqref{GGR}. 
The first is to take an element of $\cU$ which is primitive, that is, 
satisfies 
$$
\Delta(G) = G\otimes G \,. 
$$
Quantum affine algebras have very few primitive elements.  For 
Nakajima varieties the only nontrivial choice is an exponentials of linear
functions of $\bv$, which can be written as $z^\bv$ in multiindex
notation. 

The other option is to introduce a shift operator,
$$
\tau_{a,q}f(a) = f(q a)  
$$
which satisfies \eqref{GGR} because the $R$-matrix 
depends only on the ratio $a_1/a_2$ of evaluation parameters. 
Combining $z^\bv$ with a shift operator in the form 
$$
G f(a) = z^{\bv} \, f(q^{-1} a) \,,
$$
 we get an action of $\Z^n$ by commuting difference operators. 

The corresponding difference equations were introduced in 
\cite{FrenResh} and are known as the quantum 
Knizhnik-Zamolodchikov equations. In the $q\to 1$ limit, 
they become the Knizhnik-Zamolodchikov differential 
equations for conformal blocks of WZW conformal field theories. 
See the book \cite{EFK} for an introduction.

\subsection{Minuscule shifts and qKZ} 

\subsubsection{}
We will identify the shifts of equivariant variable from 
Section \ref{s_shift_eq} with 
qKZ operators in the case when the cocharacter 
$\sigma$ is \emph{minuscule}. The parameters in 
the primitive element $z^\bv$ will be identified, up to a shift, with 
K\"ahler variables. 

Being minuscule is an important technical notion, and 
one of the equivalent ways to define a minuscule cocharacter
is the following. Let $X_0$ be the affinization of $X$, that is, 
the spectrum of global functions on $X$. The ring 
$\C[X_0]=H^0(\cO_X)$ is $\Z$-graded by the characters of $\bA$ 
and one can look for generators of minimal degree with 
respect to $\sigma$. By definition, $\sigma$ is minuscule
if $\C[X_0]$ is generated by elements of degree $\{\pm 1,0\}$ 
with respect to $\sigma$. 

For Nakajima varieties, $\C[X_0]$ is generated by invariant functions 
on the prequotient, and the first fundamental theorem of invariant
theory 
provides  a set of generators for it.  It  is shown in Section 
2.6 of \cite{MO1} that actions 
$$
\sigma: \Ct \to \prod GL(W_i)
$$
with the character
$$
\bw = a \bw' + \bw''\,, \quad a \in \Ct\,,
$$
are minuscule. 

\begin{Exercise}
 Check this. 
\end{Exercise}

\begin{Exercise}
Prove that the weights of a minuscule cocharacter $\sigma$ in the 
normal bundle to $X^\sigma$ are $\pm 1$. 
\end{Exercise}

\subsubsection{}
It is clear that the cocharacter $\sigma$ has to be very 
special for the shift operator to have the qKZ form. 
Indeed, the constructions of Section 
\ref{s_diff} produce difference operators 
as a formal series in $z$. For the shifts of K\"ahler 
variables it is shown in \cite{OS}, as a consequence of commutation with 
qKZ, that the corresponding series converges to a rational 
function of $z$, up to an overall scalar. 

By contrast, the qKZ operator are polynomial in $z$ and, in 
fact, they involve the variables $z$ only through the 
diagonal operator $z^\bv$ acting in one of the tensor 
factors. The possibility to factor out the entire degree 
dependence in such a strong way is very special and 
is destroyed, for example, if we replace a difference operator by 
its square. Of course, $\sigma^2$ is never minuscule
for $\sigma\ne 1$. 

\subsubsection{}
Let $\sigma$ be a minuscule cocharacter and let $X^\sigma$ 
be the fixed locus on the corresponding $1$-parameter 
subgroup. The identification of geometric 
structures with structures from quantum groups happens 
in the basis of stable envelopes and, in particular, 
the shift operator $\bfS_\sigma$ will be identified with qKZ 
after conjugation by 
\begin{equation}
\Stab_{+,T^{1/2},\cL} : K_\bT(X^\sigma) \to K_\bT(X) \label{S_+} 
\end{equation}
where plus means $\sigma$-attracting directions and 
\begin{equation}
\cL \in \textup{\{neighborhood of $0$\}} 
\cap (-\Camp) \subset  \Pic(X)\otimes \R\,.
\label{L-ampl} 
\end{equation}
We will see the importance of this particular choice of the 
slope $\cL$ later.

\subsubsection{}

The conjugation of the shift operator to qKZ is summarized 
in the following diagram: 
\begin{equation}
 \xymatrix{ K_\bT(X^\sigma) \ar[d]_{\Res\circ\Stab_+}
   \ar[rrrr]^{\tau_{\sigma}^{-1}  (-1)^{\codim/2} z^{\deg} \,R}
&&&&
   K_\bT(X^\sigma)_\textup{localized}
\ar[d]^{\Res\circ\Stab_+}\\
K_\bT(X^\sigma)\ar[rrrr]^{\tau_{\sigma}^{-1}  \, \bfS_{\sigma}} &&&&
K_\bT(X^\sigma)_\textup{localized}} 
\label{diagResStab} 
\end{equation}
where $\Stab_+$ is the operator \eqref{S_+}, followed by 
the restriction $\Res$ to fixed points, $R=\Stab_-^{-1} \Stab_+$ is 
the $R$-matrix, and 
$$
\tau_{\sigma} f(a) = f(q^\sigma a) \,. 
$$
Recall that the operators that shift equivariant variables are only 
defined after localization. In particular, to define the shift 
$\tau_{\sigma}$ we need to trivialize $K_\bT(X)$ as a 
$K_{\Ct_q}(\pt)$-module. This is achieved by restriction to $K_\bT(X^\sigma)$. 

In what follows, we will often drop the restriction to fixed points
for brevity. To simplify notation, we will also denote by 
$a$ the coordinate in the image of $\sigma$ so that
$$
\tau_\sigma f(a) = f(qa) \,. 
$$

\subsubsection{} 

The inverse of the stable envelope may be replaced, as in 
Exercise \ref{ex_transpose}, by the transposed correspondence 
with the opposite polarization, cone, and slope. Therefore, 
the diagram \eqref{diagResStab} is equivalent to the 
following 

\begin{Theorem}\label{t_R} 
If $\sigma$ is a minuscule cocharacter then 
\begin{equation}
 \Stab^\transp_{-,\, \Topp,\, -\cL}\Big|_{a\mapsto q^\sigma a} 
\,\, \bfS_\sigma \,\,  \Stab_{-,\, T^{1/2},\,\cL}   = 
(-1)^{\frac12 \codim X^\sigma} \, z^{\deg} 
\label{sSs}
  \end{equation}
for all 
$$
\cL \in \textup{\{neighborhood of $0$\}} 
\cap (-\Camp) \subset  \Pic(X)\otimes \R\,.
$$
\end{Theorem}

\noindent
The proof of this theorem will be given in Section \ref{qKZ}

\subsubsection{}
For the following exercises we do not need 
to assume $\sigma$ minuscule. 

\begin{Exercise}\label{ex_block_diag} 
Let the slope $\cL$ be as \eqref{L-ampl}. Show that 
the following limit exists and is a block-diagonal operator 
$$
\lim_{a\to 0} a^{\lan\det T^{1/2}_{>0},\sigma\ran} \Stab_+ = 
\textup{block-diagonal}  \,. 
$$
As explained before, here $\Stab_+$ really means 
$\Res \circ \Stab_+$. Meanwhile, 
$$
\lim_{a\to 0} a^{\lan\det T^{1/2}_{>0},\sigma\ran} \Stab_- = 
\textup{block-triangular}  \,, 
$$
where the triangularity is the same as for $\Stab_-$. 
\end{Exercise}

\begin{Exercise}
Prove that 
$$
\Stab_+^{-1} \, \tau_\sigma^{-1} \, \Stab_+ \to q^{\lan\det
  T^{1/2}_{>0},\sigma\ran} \, \tau_\sigma^{-1} \,, \quad a\to 0 \,.
$$
\end{Exercise}

\begin{Exercise}
  Prove that 
  \begin{equation}
    R\to \hbar^{-\codim/4} \,\,\,  \textup{block-unitriangular}\,, 
\quad a\to 0 \,, 
  \end{equation}
where the triangularity is the same as for $\Stab_-$. 
\end{Exercise}

\subsubsection{}
We will see in Section \ref{s_Glue_Stab} that the following 
limit 
\begin{equation}
\bJ_{0,+} = \lim_{a\to 0}  \Stab_+^{-1} \, \bJ(a) \, \Stab_+
\label{bJ0+} 
\end{equation}
exists, as does the limit of the operator $\bfS_\sigma$ in the 
stable basis. Again, this does not require the
minuscule assumption. Granted this, 
we deduce from \eqref{conJ} the following analog of \eqref{MOz}. 

\begin{Proposition}\label{c_S_to_0}
We have
\begin{equation}
\bJ_{0,+}^{-1}\, \Stab_+^{-1} \, \bfS_{\sigma} \, \Stab_+ \, \bJ_{0,+} \to (-\hbar^{1/2})^{\lan \det N_{<0},\sigma \ran} \,
q^{-\lan \det T^{1/2}_{>0},\sigma \ran} \, z^{\deg} \,, 
\label{S_to_0}
\end{equation}
as $a\to 0$, 
where $N_{<0}$ is the subspace of the normal bundle to $X^\sigma$
spanned by weights that go to $\infty$ as $a\to 0$. 
\end{Proposition}

\begin{proof}
Since 
$$
\Phi\big((1-q\hbar)\big(T^{1/2}_{x,\sigma} -
T^{1/2}_{x,0}\big)\big) \to (q\hbar)^{\lan \det T^{1/2}_{<0},\sigma \ran}
\,, \quad a \to 0 \,,
$$
we see that the operator $\bE(\sigma)$ in \eqref{sdJ} becomes
$$
\bE(\sigma) \to \textup{RHS of \eqref{S_to_0}}
$$
in the $a\to 0$ limit. {}From Exercise \ref{ex_block_diag}, 
$$
\Stab_+^{-1} \, \bE(\sigma) \, \Stab_+ = 
\Stab_+^{-1} \, a^{-\lan\det T^{1/2}_{>0},\sigma\ran} \,\bE(\sigma) \,
a^{\lan\det T^{1/2}_{>0},\sigma\ran} 
\Stab_+ \to \bE(\sigma)\,,
$$
as $a\to 0$, and so the conclusion follows from \eqref{sdJ}. 
\end{proof}

\begin{Exercise}
Check that with the minuscule condition on $\sigma$ 
Corollary \ref{c_S_to_0} specializes to 
\begin{equation}
\bJ_{0,+}^{-1}\, \Stab_+^{-1} \, \bfS_{\sigma} \, \Stab_+ \, \bJ_{0,+}
\to (-\hbar^{1/2})^{-\codim/2} \,
q^{-\rk T^{1/2}_{>0}} \, z^{\deg} \, \quad a\to 0\,.  
\label{S_to_0_min}
\end{equation}
\end{Exercise}

\subsubsection{}

{}From the above exercises we conclude that 
$$
 \Stab_+^{-1} \, \tau_\sigma^{-1} \, \bfS_{\sigma} \, \Stab_+
\to \tau_\sigma^{-1} \, \bJ_{0,+} \, (-\hbar)^{\codim/2} 
z^{\deg} \, \bJ_{0,+}^{-1} \,,
$$
and thus the diagram \eqref{diagResStab} implies the 
following 

\begin{Corollary}\label{c_bJ0} 
 The operator $\bJ_{0,+}$ diagonalizes $(-1)^{\codim/2} z^{\deg} R(0)$ 
and, in particular, is triangular with the same triangularity as
$\Stab_-$. 
\end{Corollary}

\subsection{Glue operator in the stable basis}\label{s_Glue_Stab} 

\subsubsection{}

Let $T^{1/2}X$ be a polarization of $X$ and assume that
the sheaf $\tO_\vir$, and all K-theoretic curve counts, 
are defined using this polarization as in \eqref{def_tO}. 
Our next goal is to prove the following 

\begin{Theorem}\label{t_SGS} 
The operator 
\begin{align}
 \bfG_{\Stab} &=   
    \Stab^{-1} 
\,\, \bfG \,\,  \Stab  
\notag \\
&=   \label{StabGStab}
    \Stab^\transp_{-\fC,\, \Topp,\, -\cL}
\,\, \bfG \,\,  \Stab_{\fC,\, T^{1/2},\,\cL}  
  \end{align}
is independent of the equivariant parameters in $\bA$ for any 
$\fC$ and for 
\begin{equation}
\cL \in \textup{\{neighborhood of $0$\}} \subset  \Pic(X)\otimes \R\,.
\label{cLnearzero} 
\end{equation}
\end{Theorem}

\noindent
The proof takes several steps. 

\subsubsection{}

As in Section \ref{s_glue_proper}, we see that the convolution in \eqref{StabGStab} 
is proper, thus a Laurent polynomial in $a\in \bA$. 
To prove it is a constant, it suffices to check it has a finite 
limit as we send $a$ to infinity in any direction. 

This
can be done by equivariant localization. 

\subsubsection{}

Suppose
$$
f\in \left(\QM^\text{\normalsize$\sim$}_{\textup{relative $p_1,p_2$}}\right)^\bA
  \,,
$$
and let $C'$ be the domain of $f$. The curve $C'$ a chain of $C'_i\cong \bP^1$. 
Since $f$ is fixed, there 
exists 
\begin{equation}
\gamma: \bA \to \Aut(C',p_1',p_2') \cong \prod_{C_i} \Ct 
\label{def_gamma}
\end{equation}
such that 
$$
a f = f \gamma(a)\,, \quad \forall a\in \bA \,. 
$$
Let $p=C'_i \cap C'_{i+1}$ be a node of $C'$. We say that $f$ 
 is \emph{broken} at $p$ if 
 \begin{equation}
\textup{weight}_\gamma \left( T_p C'_i \otimes  T_p C'_{i+1}
\right)\ne 1 \,.
\label{broken_w} 
\end{equation}
This terminology comes from \cite{OP1}. Since the space in
\eqref{broken_w} corresponds to the smoothing of node $p$ in 
$\Def(C')$, this means that broken nodes contribute to 
the virtual normal bundle in \eqref{Tvir_rel}, and namely 
\begin{equation}
  N_\vir = N_\vir^{\textup{fixed domain}} + 
\bigoplus_{
\substack{\textup{broken}\\
\textup{nodes $p$}}}  T_p C'_i \otimes  T_p C'_{i+1} \,.
\end{equation}
In localization formulas, this will give 
$$
\begin{matrix*}[c]
    \textup{contribution of} \\
\textup{broken nodes}
  \end{matrix*}  \quad =  \frac{1}{\prod (1- \textup{weight} \otimes 
\textup{line bundle})} 
$$
and this remains bounded as $a$ goes to any infinity of the torus 
avoiding the poles. 

\subsubsection{}
By Lemma \ref{lsquare}, we have 
\begin{equation}
 N_\vir^{\textup{fixed domain}} = T^{1/2}_{f(p_1),\mov} +
 T^{1/2}_{f(p_2),\mov} + 
\dots 
\end{equation}
where 
$$
 T^{1/2}_{f(p_1),\mov} =  T^{1/2}_{f(p_1),\ne 0}
$$
is the part of the polarization with nontrivial $\bA$-weights and 
dots stand for balanced K-theory class as in Definition 
\ref{def_balanced}.  These omitted term will contribute 
$$
\aroof(\dots) = O(1)\,, \quad a\to \infty \,,
$$
to the localization formulas, so it remains to deal with the 
polarization. 

\subsubsection{}

By definition \eqref{def_tO_rel}, we have 
\begin{equation}
  \begin{matrix*}[c]
    \textup{contribution of} \\
\textup{polarization}
  \end{matrix*}  
\quad =  \frac{1}{
\Ldm T^{1/2}_{f(p_1),\mov} \,\left(\Ldm
      T^{1/2}_{f(p_2),\mov}\right)^\vee} \,.
\label{contrib_pol}
\end{equation}
While a polarization is only a virtual vector space, we have
$$
T^{1/2}_{f(p_2),\mov} = \left(T_{f(p_2)} X\right)_\textup{repelling}+
\dots 
$$
where dots stand for a balanced class which will shift the 
asymptotics of 
\eqref{contrib_pol} by an overall weight. 

\subsubsection{}

To find this overall shift by a weight, we may normalize the 
exterior algebras so that they are self-dual. This gives 
$$
\frac{1}{
\left(\Ldm T^{1/2}_{f(p_2),\textup{mov}}\right)^\vee}  \propto 
\left(\frac{\det  T^{1/2}_{f(p_2),\textup{mov}} }{\det
    T_{f(p_2),\textup{rep}}}
\right)^{1/2} 
\frac{1}{
\left(\Ldm \left(T_{f(p_2),\textup{rep}} \right)\right)^\vee}
$$
as $a\to\infty$, which precisely compensates the weights 
\eqref{degK} of the restriction of a stable envelope to the 
fixed component containing $f(p_2)$ for all slopes
$\cL$ sufficiently close to $0$. 

For the contribution of polarization at $f(p_1)$, the 
conclusion is the same, with the replacement 
$$
\det  T^{1/2}_{f(p_2),\textup{mov}}  
\mapsto 
\det  T^{-1/2}_{f(p_1),\textup{mov}}   = 
\det  T^{1/2,\textup{opp}}_{f(p_1),\textup{mov}} \,.
$$
This concludes the proof of Theorem \ref{t_SGS}. 

\subsubsection{}

Recall the shift operator $\bfS_\sigma$ defined in \eqref{defSs}. 
One of its factors is the glue operator that was just discussed. 
As before, we assume that $\sigma$ is a cocharacter of $\bA$, 
that is, the shift of equivariant weights is in a direction that 
\emph{preserves} the symplectic form. 

\begin{Corollary}\label{c_reg_shift} 
The operator 
\begin{align}
\Stab_{\fC,\, T^{1/2},\,\cL}^{-1} \Big|_{a\mapsto q^\sigma a} 
\,\, \bfS_\sigma\,\,  \Stab_{\fC,\, T^{1/2},\,\cL}    \in 
\End K(X^\bA)_\textup{localized} 
\label{shift_in_stable} 
  \end{align}
is bounded and invertible at any infinity of the torus $\bA$ for 
any 
$\fC$ and for $\cL$ as in \eqref{cLnearzero}. 
\end{Corollary}

The first factor in \eqref{defSs} is only defined using 
localization, and so \eqref{shift_in_stable} is really a rational 
function of the equivariant parameters. This Corollary implies 
the $q$-difference equation in the variables $a\in \bA$ has 
regular singularities. 

Also, since $\bfS_\sigma$ is really an operator in $K(X^\bA)$, 
the stable envelopes are really the compositions 
$$
K(X^\bA) \xrightarrow{\,\, \Stab \,\,} K(X) 
\xrightarrow{\,\, \textup{restriction} \,\,} K(X^\bA) \,,
$$
and, in particular, it makes sense to shift the equivariant 
variables in stable envelopes. We have already seen this in 
in \eqref{sSs} above.

\begin{proof}
Bounding the $\bA$-weights is done as in the above argument. 
The key point is that as in Lemma \ref{lsquare} we have 
$$
T_\vir = \Hd\left(\cT^{1/2} + \hbar^{-1} \, \left(\cT^{1/2}
  \right)^\vee\right)
$$
even for twisted quasimaps as long as $\sigma$ preserves the 
symplectic form. 

The inverse of $\bfS_\sigma$ may be computed from $\bfS_{_\sigma}$, 
therefore the shift operators are both bounded and invertible at
infinities of $\bA$. 
\end{proof}

\subsubsection{}

\begin{Exercise}
  Prove the existence of the limit \eqref{bJ0+}. 
\end{Exercise}

\begin{Exercise}
Show the triangularity of the operator $\bJ_{0,+}$ from Corollary
\ref{c_bJ0} directly, by examining the contribution of the broken 
nodes in the $a\to 0$ limit. 
\end{Exercise}

\subsection{Proof of Theorem \ref{t_R}}\label{qKZ}

% \subsubsection{}

% \begin{Theorem}\label{t_R} 
% If $\sigma$ is a minuscule cocharacter then 
% %
% \begin{equation}
%  \Stab^\transp_{-,\, \Topp,\, -\cL}\Big|_{a\mapsto q^\sigma a} 
% \,\, \bfS_\sigma \,\,  \Stab_{-,\, T^{1/2},\,\cL}   = 
% (-1)^{\codim X^\bA} \hbar^{\frac12\dim X} \, z^{\deg} 
% \label{sSs}
%   \end{equation}
% %
% for all 
% $$
% \cL \in \textup{\{neighborhood of $0$\}} 
% \cap (-\Camp) \subset  \Pic(X)\otimes \R\,.
% $$
% \end{Theorem}

\subsubsection{}
We may assume that $\bA$ is a $1$-dimensional torus with 
coordinate $a$ and that 
$$
\sigma: \Ct_q  \xrightarrow{\,\, \sim\,\,} \bA 
$$
is an isomorphism. We write $qa$ in place of $q^\sigma a$. 

The two-dimensional torus $ \Ct_q  \times \bA$ acts on the 
moduli spaces of twisted quasimaps. By construction, the 
group $\Ct_q$ acts in the fiber of $\cV$ at $p'_1$ and 
\emph{does
not} act in the fiber of $\cV$ at $p'_2$, see Section \ref{s_O(p)}. 

\subsubsection{}

As discussed before, 
the pushforward in \eqref{sSs} is proper,
therefore the left-hand side is a Laurent polynomial in $q$ and $a$. 
We analyze the behavior of this polynomial as $(q,a)$ go 
to different infinities of the torus $\Ct_q  \times \bA$. 
 
In order to prove \eqref{sSs}, it suffices to show that the 
left-hand side
remains bounded at all infinities and limits to the right-hand
side as 
\begin{equation}
a\to \infty\,, \quad qa \to 0 \,.  \label{qa} 
\end{equation}

\subsubsection{}
Localization with respect to the $\Ct_q$-action gives the following 
schematic formula 
$$
\bfS_\sigma  = \bJ\big|_{a\mapsto q a} \circ 
 \left| \textup{edge}\right| \circ \bJ^{\transp} \circ
 \bfG^{-1} \,,
$$
where the edge operator is the operator associated with constant 
twisted quasimaps and 
$$
\bJ^{\transp} = \left| \textup{capping
    operator}\right| 
$$
is the operator with 
matrix coefficients \eqref{accord2}. 

We can thus factor the LHS of \eqref{sSs} as 
follows
\begin{alignat}{2}
  \textup{LHS of \eqref{sSs}}  = &
   \left(\Stab^\transp_{-,\, \Topp,\, -\cL}  \, \, \bJ  \,\, 
 \Stab_{+,\, T^{1/2},\,  \cL}  
\right)
\Big|_{a\mapsto q a}  &&\times 
\label{factor1} \\
& 
\left( \Stab^\transp_{-,\, \Topp,\, -\cL} 
\right)
\Big|_{a\mapsto q a} 
 \, \, 
 \left| \textup{edge}\right|
\,\, \Stab_{-,\, T^{1/2},\,\cL}  \quad &&\times
\label{factor2} \\
& 
\Stab^\transp_{+,\, \Topp,\,-\cL} 
\,\,  
\bJ^{\transp} \, \bfG^{-1}  \,\, \Stab_{-,\, T^{1/2},\,\cL} \,.
&&
\label{factor3}
\end{alignat}
We deal with each factor separately. 

\subsubsection{}

\begin{Lemma}
The operator in the RHS of \eqref{factor1} is bounded at 
all infinities of $\Ct_q\times \bA$ and goes to $1$ at the point 
\eqref{qa}. 
\end{Lemma}

\begin{proof}
We argue as in the proof of Theorem \ref{t_SGS}, now with the 
additional insertion of $\left(1-q^{-1} \psi_{p_1}\right)^{-1}$. 
Since $\bA$ acts on tangents spaces to all fixed points with 
weights $0,\pm 1$, we conclude 
$$
\textup{weight}(\psi_{p_1}) \in \{1, (qa)^{\pm 1} \}  \,.
$$
In every case, 
$$
\frac1{1-q^{-1} \, \textup{weight}(\psi_{p_1})} \to 0 
$$
at the point \eqref{qa} and remains bounded at all other infinities. 
Thus in the limit \eqref{qa} the accordions at $p_1$ are 
suppressed and the result follows. 
\end{proof}

\subsubsection{}

\begin{Lemma}
The operator in \eqref{factor3} is bounded at 
all infinities of $\Ct_q\times \bA$ and goes to $1$ at the point 
\eqref{qa}. 
\end{Lemma}

\begin{proof}
We argue as in the previous lemma, now with the 
 insertion of $\left(1-q \psi_{p_2}\right)^{-1}$. 
We have 
$$
q \, \textup{weight}(\psi_{p_2}) \in \{q, q a^{\pm 1} \}  \,,
$$
which goes to zero in the limit \eqref{qa}. Thus $\psi_{p_2}$ 
plays no role in this limit and
$$
\bJ^{\transp} \, \bfG^{-1} \to \bfG \, \bfG^{-1} = 1 \,,
$$
as was to be shown. 
\end{proof}

\subsubsection{}

It remains to deal with the operator in \eqref{factor2}. Let 
$f \equiv x$ be a constant map to $x \in F$. 
Since the weights of $\bA$ in $T_x X$ are $\{0,\pm 1\}$, the 
virtual normal bundle at $f$ to the space of $\bA$-fixed quasimaps is 
\begin{equation}
  N_\vir  =  H^0(\cT_{x,>0}) = \left(T_{x,>0} X \right)\big|_{a\mapsto
    qa} + T_{x,>0} X  \,,
\end{equation}
where $T_{x,>0} X$ denotes the $\bA$-attracting part of $T_x X$
and $\cT_{x,>0}$ is the bundle of these spaces over the domain 
of the quasimap. 

Therefore, the normal bundle contribution to $\cO_\vir \otimes
\cK^{1/2}_\vir$, which appears in the denominator of the 
localization formula, precisely cancels with the 
contributions \eqref{Stab_FF_naive}  from the 
diagonal part of stable 
envelopes,  which appears in the numerators. 
The 
remaining factors in \eqref{Stab_FF} combine with 
the polarization factors in \eqref{def_tO} to 
give 
\begin{multline}
 (-1)^{\rk N^{1/2,\textup{opp}}_{>0}+\rk N^{1/2}_{>0}} 
\frac1{(\det N^{1/2,\textup{opp}})\big|_{a\mapsto qa} \det N^{1/2}}
\times \\
\left(
  \frac{\det N^{1/2} }{ \det N^{1/2} \big|_{a\mapsto qa}  }
\right)^{1/2}   = 
 (-1)^{\frac12 \codim F} h^{\frac12 \codim F} 
\end{multline}
as the normal bundle contribution of the diagonal parts of 
stable envelopes. The factor with $\hbar$ disappears when 
we convert two-point invariants to operators as in 
\eqref{convolution_twisted}.  Together with the degree weight, this
gives
$$
\begin{matrix*}[c]
    \textup{diagonal} \\
\textup{contributions}
  \end{matrix*}  \quad =  (-1)^{\frac12\codim F}  z^{\deg}\,,
$$
where the degree of the constant map was computed in 
\eqref{deg_const}\,.

To suppress the off-diagonal contributions, we need that 
\begin{equation}
\weight_{qa} \frac{\cL^{-1}\big|_{F_2}} {\cL^{-1}\big|_{F_1}} \, 
\weight_{a} \frac{\cL\big|_{F_2}} {\cL\big|_{F'_1}} \to 0 \,,
\label{weight_qa} 
\end{equation}
in the limit \eqref{qa} for any triple of components $F_1,F'_1,F_2$ 
of the fixed locus such that 
\begin{equation}
F_2 \subset \Attr^f_{-\fC}(F_1), \Attr^f_{-\fC}(F'_1)
\label{repellFF} 
\end{equation}
and 
$$
(F_1,F'_1) \ne (F_2,F_2) \,. 
$$
The condition \eqref{repellFF} means there is a chain of 
torus orbits from $F_1$ and $F'_1$ to $F_2$ and computing 
the degree of that chain of curves with respect to $\cL^{-1}$ we see that 
\eqref{weight_qa} is satisfied if $\cL^{-1}$ is ample. 

\noindent 
This concludes the proof of Theorem \ref{t_R}.

%%%%%%%%%%%%%%%%%%%%%%%%%%%%%%%%%%%%%%%%%%%%%%%%%%%%%%%%%%%%%%%%%%%%%
%    
% To add references to your document, replace the two \bib commands below. 
%
%         1. You can use a list of \bib commands for the items you reference as is
%         done in our toy example here.
%
%         2. A second option is to use the command 
%             \bibselect{yourltbfile}
%         to point to a file of \bib commands that should be named 
%         yourltbfile.ltb and be placed in the same folder as your LaTeX
%         source files. 
%
%         3. A third option is to use the command 
%             \bibliography{yourbibfile}
%         to point to a file of BibTeX \bib commands that should be named 
%         yourltbfile.bbl and be placed in the same folder as your LaTeX
%         source files. 
%   
% If you use option 3. above, you should comment out or delete the lines
%            \begin{bibdiv}
%                \begin{biblist}
%        before the \bib command below as well as the line
%                  \end{biblist}
%              \end{bibdiv}
%        after it. 
%
% If you use options 2. or 3. and wish to make your source file self-contained you may
%         for final submission, simply copy the \bib entries to your \LaTeX\ file and
%         wrap them, if necessary, as indicated above.
%  
%%%%%%%%%%%%%%%%%%%%%%%%%%%%%%%%%%%%%%%%%%%%%%%%%%%%%%%%%%%%%%%%%%%%%

\bibspread

\vfill\eject

% Andrei: delete the lines below if you see no problem with the converted graphics.
%\includegraphics[scale=0.25]{chimney_picture1.eps}
%\includegraphics[scale=0.25]{Om1Om1boxes.eps}
%\includegraphics[scale=0.25]{partition_legs}

\end{document}